\documentclass[12pt,reqno]{amsart}
\usepackage{}
\usepackage{amsmath}
\usepackage{amsfonts}
\usepackage{amssymb}
\usepackage[all,cmtip]{xy}           
\usepackage{ulem}

\usepackage{bbm}
\usepackage{bbding}
\usepackage{txfonts}
\usepackage[shortlabels]{enumitem}
\usepackage{ifpdf}
\ifpdf
\usepackage[colorlinks,final,backref=page,hyperindex]{hyperref}
\else
\usepackage[colorlinks,final,backref=page,hyperindex,hypertex]{hyperref}
\fi
\usepackage{amscd}
\usepackage{tikz-cd}
\usepackage[active]{srcltx}

\usepackage{soul}

\makeatletter

\newtheorem{theorem}{Theorem}[section]
\newtheorem{prop}[theorem]{Proposition}

\newtheorem{prop-def}{Proposition-Definition}[section]

\theoremstyle{definition}
\newtheorem{defn}[theorem]{Definition}

\newtheorem{remark}[theorem]{Remark}
\newtheorem{exam}[theorem]{Example}

\newcommand{\nc}{\newcommand}

\nc{\delete}[1]{{}}
\nc{\mmargin}[1]{}

\nc{\mlabel}[1]{\label{#1}}  
\nc{\mcite}[1]{\cite{#1}}  
\nc{\mref}[1]{\ref{#1}}  
\nc{\mbibitem}[1]{\bibitem{#1}} 

\delete{
	\nc{\mlabel}[1]{\label{#1}  
		{\hfill \hspace{1cm}{\bf{{\ }\hfill(#1)}}}}
	\nc{\mcite}[1]{\cite{#1}{{\bf{{\ }(#1)}}}}  
	\nc{\mref}[1]{\ref{#1}{{\bf{{\ }(#1)}}}}  
	\nc{\mbibitem}[1]{\bibitem[\bf #1]{#1}} 
}

 \font\cyrs=wncyr7

\nc{\vep}{\varepsilon}
\nc{\bin}[2]{ (_{\stackrel{\scs{#1}}{\scs{#2}}})}  
\nc{\binc}[2]{(\!\! \begin{array}{c} \scs{#1}\\
		\scs{#2} \end{array}\!\!)}  
\nc{\bincc}[2]{  ( {\scs{#1} \atop
		\vspace{-1cm}\scs{#2}} )}  
\nc{\oline}[1]{\overline{#1}}
\nc{\mapm}[1]{\lfloor\!|{#1}|\!\rfloor}
\nc{\bs}{\bar{S}}
\nc{\la}{\longrightarrow}
\nc{\ot}{\otimes}
\nc{\rar}{\rightarrow}
\nc{\lon }{\,\rightarrow\,}
\nc{\dar}{\downarrow}
\nc{\dap}[1]{\downarrow \rlap{$\scriptstyle{#1}$}}
\nc{\defeq}{\stackrel{\rm def}{=}}
\nc{\dis}[1]{\displaystyle{#1}}
\nc{\dotcup}{\ \displaystyle{\bigcup^\bullet}\ }
\nc{\hcm}{\ \hat{,}\ }
\nc{\hts}{\hat{\otimes}}
\nc{\hcirc}{\hat{\circ}}
\nc{\lleft}{[}
\nc{\lright}{]}
\nc{\curlyl}{\left \{ \begin{array}{c} {} \\ {} \end{array}
	\right .  \!\!\!\!\!\!\!}
\nc{\curlyr}{ \!\!\!\!\!\!\!
	\left . \begin{array}{c} {} \\ {} \end{array}
	\right \} }
\nc{\longmid}{\left | \begin{array}{c} {} \\ {} \end{array}
	\right . \!\!\!\!\!\!\!}
\nc{\ora}[1]{\stackrel{#1}{\rar}}
\nc{\ola}[1]{\stackrel{#1}{\la}}
\nc{\scs}[1]{\scriptstyle{#1}} \nc{\mrm}[1]{{\rm #1}}
\nc{\dirlim}{\displaystyle{\lim_{\longrightarrow}}\,}
\nc{\invlim}{\displaystyle{\lim_{\longleftarrow}}\,}
\nc{\dislim}[1]{\displaystyle{\lim_{#1}}} \nc{\colim}{\mrm{colim}}
\nc{\mvp}{\vspace{0.3cm}} \nc{\tk}{^{(k)}} \nc{\tp}{^\prime}
\nc{\ttp}{^{\prime\prime}} \nc{\svp}{\vspace{2cm}}
\nc{\vp}{\vspace{8cm}}
\nc{\modg}[1]{\!<\!\!{#1}\!\!>}
\nc{\intg}[1]{F_C(#1)}
\nc{\lmodg}{\!<\!\!}
\nc{\rmodg}{\!\!>\!}
\nc{\cpi}{\widehat{\Pi}}
\nc{\ssha}{{\mbox{\cyrs X}}} 
\nc{\tsha}{{\mbox{\cyrt X}}}
\nc{\shpr}{\diamond}    
\nc{\labs}{\mid\!}
\nc{\rabs}{\!\mid}

\nc{\ad}{\mrm{ad}}
\nc{\ann}{\mrm{ann}}
\nc{\Aut}{\mrm{Aut}}
\nc{\md}{\mbox{-}\mathsf{mod}}
\nc{\br}{\mrm{bre}}
\nc{\can}{\mrm{can}}
\nc{\Cont}{\mrm{Cont}}
\nc{\rchar}{\mrm{char}}
\nc{\cok}{\mrm{coker}}
\nc{\de}{\mrm{dep}}
\nc{\dtf}{{R-{\rm tf}}}
\nc{\dtor}{{R-{\rm tor}}}

\nc{\Div}{{\mrm Div}}
\nc{\Diff}{\mrm{DA}}
\nc{\Diffl}{\mathsf{DA}_\lambda}
\nc{\diffo}{{\mathsf{DO}_\lambda}}
\nc{\alg}{\mathsf{Alg}}
\nc{\End}{\mrm{End}}
\nc{\Ext}{\mrm{Ext}}
\nc{\Fil}{\mrm{Fil}}
\nc{\Fr}{\mrm{Fr}}
\nc{\Frob}{\mrm{Frob}}
\nc{\Gal}{\mrm{Gal}}
\nc{\GL}{\mrm{GL}}
\nc{\Hom}{\mrm{Hom}}
\nc{\Hoch}{\mrm{Hoch}}
\nc{\hsr}{\mrm{H}}
\nc{\hpol}{\mrm{HP}}
\nc{\id}{\mrm{id}}
\nc{\im}{\mrm{im}}
\nc{\Id}{\mrm{Id}}
\nc{\ID}{\mrm{ID}}
\nc{\Irr}{\mrm{Irr}}
\nc{\incl}{\mrm{incl}}
\nc{\length}{\mrm{length}}
\nc{\NLSW}{\mrm{NLSW}}
\nc{\Lie}{\mrm{Lie}}
\nc{\mchar}{\rm char}
\nc{\mpart}{\mrm{part}}
\nc{\ql}{{\QQ_\ell}}
\nc{\qp}{{\QQ_p}}
\nc{\rank}{\mrm{rank}}
\nc{\rcot}{\mrm{cot}}
\nc{\rdef}{\mrm{def}}
\nc{\rdiv}{{\rm div}}
\nc{\rtf}{{\rm tf}}
\nc{\rtor}{{\rm tor}}
\nc{\res}{\mrm{res}}
\nc{\SL}{\mrm{SL}}
\nc{\Spec}{\mrm{Spec}}
\nc{\tor}{\mrm{tor}}
\nc{\Tr}{\mrm{Tr}}
\nc{\tr}{\mrm{tr}}
\nc{\wt}{\mrm{wt}}
\nc{\op}{\mrm{op}}

\nc{\mbf}{\mathbf}
\nc{\bfk}{{\mathbf k}}
\nc{\bft}{{\mathbf t}}
\nc{\bfone}{{\bf 1}}
\nc{\bfzero}{{\bf 0}}
\nc{\detail}{\marginpar{\bf More detail}
	\noindent{\bf Need more detail!}
	\svp}
\nc{\gap}{\marginpar{\bf Incomplete}\noindent{\bf Incomplete!!}
	\svp}
\nc{\FMod}{\mathbf{FMod}}
\nc{\Int}{\mathbf{Int}}
\nc{\Mon}{\mathbf{Mon}}
\nc{\remarks}{\noindent{\bf Remarks: }}
\nc{\Rep}{\mathbf{Rep}}
\nc{\Rings}{\mathbf{Rings}}
\nc{\Sets}{\mathbf{Sets}}
\nc{\ob}{\mathsf{Ob}}

\nc{\BA}{{\mathbb A}}   \nc{\CC}{{\mathbb C}}
\nc{\DD}{{\mathbb D}}   \nc{\EE}{{\mathbb E}}
\nc{\FF}{{\mathbb F}}   \nc{\GG}{{\mathbb G}}
\nc{\HH}{{\mathbb H}}   \nc{\LL}{{\mathbb L}}
\nc{\NN}{{\mathbb N}}   \nc{\PP}{{\mathbb P}}
\nc{\QQ}{{\mathbb Q}}   \nc{\RR}{{\mathbb R}}
\nc{\TT}{{\mathbb T}}   \nc{\VV}{{\mathbb V}}
\nc{\ZZ}{{\mathbb Z}}   \nc{\TP}{\widetilde{P}}
\nc{\m}{{\mathbbm m}}

\nc{\cala}{{\mathcal A}}    \nc{\calc}{{\mathcal C}}
\nc{\cald}{\mathcal{D}}     \nc{\cale}{{\mathcal E}}
\nc{\calf}{{\mathcal F}}    \nc{\calg}{{\mathcal G}}
\nc{\calh}{{\mathcal H}}    \nc{\cali}{{\mathcal I}}
\nc{\call}{{\mathcal L}}    \nc{\calm}{{\mathcal M}}
\nc{\caln}{{\mathcal N}}    \nc{\calo}{{\mathcal O}}
\nc{\calp}{{\mathcal P}}    \nc{\calr}{{\mathcal R}}
\nc{\cals}{{\mathcal S}}
\nc{\calT}{{\mathcal T}}   \nc{\calt}{{\Omega}}
\nc{\calv}{{\mathcal V}}    \nc{\calw}{{\mathcal W}}
\nc{\calx}{{\mathcal X}}

\nc{\fraka}{{\mathfrak a}}
\nc{\frakb}{\mathfrak{b}}
\nc{\frakg}{{\frak g}}
\nc{\frakl}{{\frak l}}
\nc{\fraks}{{\frak s}}
\nc{\frakB}{{\frak B}}
\nc{\frakm}{{\frak m}}
\nc{\frakM}{{\frak M}}
\nc{\frakp}{{\frak p}}
\nc{\frakW}{{\frak W}}
\nc{\frakX}{{\frak X}}
\nc{\frakS}{{\frak S}}
\nc{\frakA}{{\frak A}}
\nc{\frakx}{{\frakx}}

\nc{\lir}[1]{\textcolor{red}{\underline{Li:}#1 }}

\nc{\tred}[1]{\textcolor{red}{#1}} \nc{\tgreen}[1]{\textcolor{green}{#1}}
\nc{\tblue}[1]{\textcolor{blue}{#1}} \nc{\tpurple}[1]{\textcolor{purple}{#1}}
\nc{\ra}{\rightarrow}
\nc{\yuan}[1]{\tred{\underline{Yuan:}#1 }}
\nc{\cal}{\mathcal}

\nc{\bbim}[2]{#1 #2} \nc{\bbbim}[2]{#1,\, #2} \nc{\RBF}{{\rm RBF}}
\nc{\frbf}{F_{\RBF}} \nc{\shaf}{\ssha_{\Omega}} \nc{\sham}{\diamond_{\Omega}}
\nc{\lf}{\lfloor} \nc{\rf}{\rfloor} \nc{\shan}{\ssha_{\lambda}}
\nc{\rlex}{{\rm lex}}

\def\de{\delta}

\def\la{\lambda}

\def\C{\mathbb{C}}

\def\DD{{\cal D}}
\def\Ker{{\rm Ker}}

\def\rar{\rightarrow}
\def\bs{\backslash}

\nc{\lbar}[1]{\overline{#1}}
\def\Cdot#1,#2,#3;{\cdot_{#1;\,#2,\,#3}}
\def\Mu#1,#2,#3;{\mu_{#1;\,#2,\,#3}}
\def\Lambda#1,#2,#3;{\lambda_{#1,\,(#2,\,#3)}}
\def\Rhd#1,#2,#3;{\rhd_{#1;\,#2,\,#3}}
\def\Lhd#1,#2,#3;{\lhd_{#1;\,#2,\,#3}}
\def\Star#1,#2,#3;{\star_{#1;\,#2,\,#3}}
\nc{\RBO}{{\mathrm{RBF}_q}}
\nc{\RBA}{{\mathrm{RBFA}_q}}
\nc{\Alg}{\mathrm{Alg,\Omega}}
\nc{\NjO}{\mathrm{NjO}}
\nc{\rmH}{\mathrm{H}}

\begin{document}

\title[Rota-Baxter family $\Omega$-associative conformal algebras and their cohomology theory ]{Rota-Baxter family $\Omega$-associative conformal algebras and their cohomology theory}

\author{Yuanyuan Zhang}
\address{School of Mathematics and Statistics, Henan University, Henan, Kaifeng 475004, P.\,R. China}
\email{zhangyy17@henu.edu.cn}

\author{Jun Zhao}
\address{School of Mathematics and Statistics, Henan University, Henan, Kaifeng 475004, P.\,R. China}
\email{zhaoj@henu.edu.cn}

\author{Genqiang Liu}
\address{School of Mathematics and Statistics, Henan University, Henan, Kaifeng 475004, P.\,R. China}
\email{liugenqiang@henu.edu.cn}

\date{\today}

\begin{abstract}
In this paper, we first propose the concept of Rota-Baxter family $\Omega$-associative conformal algebras, then we study the cohomology theory of Rota-Baxter family $\Omega$-associative conformal algebras of any weight and justify it by interpreting the lower degree cohomology groups as formal deformations.
\end{abstract}

\subjclass[2010]{
16W99, 
16E40,   
16S80,   
16S70   
}

\keywords{cohomology, extension, deformation, Rota-Baxter family algebra, conformal algebra}

\maketitle

\tableofcontents

\allowdisplaybreaks

\section{Introduction}
Rota-Baxter algebra was introduced by Baxter~\cite{Bax}. Later some combinatoric properties of Rota-Baxter algebras were studied by Rota~\cite{Rot69} and Cartier~\cite{Car}.
Let $\mathbb{F}$ be any field and $\lambda\in \mathbb{F}$. A Rota-Baxter algebra of weight $\lambda$ is an associative $\mathbb{F}$-algebra with a Rota-Baxter operator $P:A\ra A$ satisfying
\[P(a)P(b)=P(aP(b)+P(a)b+\lambda ab),\,\text{ for }\, a,b\in A.\]
The concept of algebras with multiple linear operators (also called $\Omega$-algebra) was first introduced by A. G. Kurosch in ~\cite{Kur}. The concept of Rota-Baxter family algebra is a generalization of Rota-Baxter algebras~\cite{Gub}, which was proposed by Guo. It arises naturally in renormalization of quantum field theory~(\cite[Proposition~9.1]{FBP} and \cite[Theorem 3.7.2]{DK}). Recently, many scholars begin to pay attention to families algebraic structures, such as Foissy~\cite{Foissy}, Zhang and Gao~\cite{ZGG,GGZ21}, Aguiar~\cite{Aguiar}, Das~\cite{Das1} and so on. \\

The notion of a conformal algebra encodes an axiomatic description of the operator product expansion (OPE) of chiral fields in conformal field theory and it is an adequate tool for the study of infinite-dimensional Lie algebras satisfying the locality property~\cite{Kac1998, Kac1996, Kac99} and Hamiltonian formalism in the theory of nonlinear evolution equations~\cite{Dor}. Lie conformal algebra appears as a useful tool in studying vertex algebra.
 An algebraic formalization of the properties of the OPE in 2-dimensional conformal field theory~\cite{BPZ} gives rise to a new class of algebraic systems, vertex operator algebras~\cite{BO, FL}. The singular part of the OPE describes the commutator of two fields, and the corresponding algebraic structures are called conformal (Lie) algebras~\cite{Kac1996} or vertex Lie algebras~\cite{FB}.
 Lie conformal algebras correspond to vertex algebras by the same way as Lie algebras correspond to their associative enveloping algebras. In particular, associative conformal algebras naturally appear in the representation theory of Lie conformal algebras. In~\cite{HB1}, Hong and Bai developed a bialgebra theory for associative conformal algebras, which is not only a conformal analogue of associative bialgebras, but also an associative analogue of conformal bialgebras. They also introduced the notions of $\cal{O}$-operators of associative conformal algebras and dendriform conformal algebras to construct solutions of associative conformal Yang-Baxter equation. In~\cite{H1}, the authors give a study of the $\mathbb{F}[\partial]$-split extending structures problem for associative conformal algebras. \\

The formal deformation theory of algebraic structures was first developed for associative algebras in the classical work of Gerstenhaber~\cite{Ger1963,Ger}, which is closely related to the cohomology theory of Hochschild cohomology on associative algebras~\cite{Hoch}.
Then Nijenhuis and Richard~\cite{NR} extended it to Lie algebras. Recently, Tang, Bai, Guo and Sheng~\cite{RBGS} developed deformation theory and cohomology theory of $\cal{O}$-operators on Lie algebras. Das~\cite{Das1} developed the corresponding weight zero cohomology theory for Rota-Baxter associative algebra. Wang and Zhou~\cite{WZ1, WZ2} defined cohomology theory for Rota-Baxter associative algebras of any weight, determined the underlying $L_\infty$-algebra and also showed that the dg operad of homotopy Rota-Baxter associative algebras is the minimal model of that of Rota-Baxter associative algebras. In~\cite{BKV}, the authors develop a cohomology theory of conformal algebras with coefficients in an arbitrary module and it possess standards properties of cohomology theories. In~\cite{Lamei}, the authors study $\cal{O}$-operator of associative conformal algebras with respect to conformal bimodules and construct cohomology of $\cal{O}$-operators.\\

In the present paper, we introduce the  Rota-Baxter family version of associative conformal algebras, and their cohomology theory. This paper is organized as follows. In Section~\ref{sec:conformal algebras}, we mainly introduce some basic concepts of conformal algebras, Rota-Baxter family conformal algebras, $\Omega$-associative conformal algebras and construct a new bimodule structure on $\Omega$-associative conformal algebras. In Section~\ref{Sect: Cohomology theory of Rota-Baxter algebras}, we define a cohomology theory for Rota-Baxter family $\Omega$-associative conformal algebras of any weight and built a chain map between the cohomology of Rota-Baxter family operator and the cohomology of Rota-Baxter family $\Omega$-associative conformal algebras. In Section~\ref{sec:formal deformations}, we study formal deformations of Rota-Baxter family on $\Omega$-associative conformal algebras and interpret them  via lower degree cohomology groups.
In Section~\ref{sec:abelian extension}, we study abelian extensions of Rota-Baxter family $\Omega$-associative conformal algebras and show that they are classified by the second cohomology, as one would expect of a good cohomology theory.

\smallskip

\subsection*{Notations}
Throught this paper, we denote by $\mathbb{F}$ any field  and $\mathbb{R}$ the field of real numbers. 
 All tensors over $\mathbb{F}$ are denoted by $\ot.$ If $A$ is a vector space, then the space of polynomials of $\lambda$ with coefficients in $A$ is denoted by $A[\lambda].$

\section{Rota-Baxter family conformal algebras and modules}
\label{sec:conformal algebras}

In this section, we mainly give some basic concepts and examples.

\subsection{Rota-Baxter family conformal algebras}
In this subsection,  we first recall the concept of Rota-Baxter family algebras,
which arises naturally in renormalization of quantum field theory~\cite[Proposition~9.1]{FBP}, some related results can be found in~\cite{Aguiar, Guo09,ZG,ZGM,ZGM2}. Then we recall some definitions, notations and results about conformal algebras and Rota-Baxter conformal algebras~\cite{Kac1998,Kac99, HB1,HB2}.

\begin{defn}(\mcite{FBP,Guo09})
Let $\Omega$ be a semigroup and $q\in \mathbb{F}$ be given.
A {\bf Rota-Baxter family} of weight $q$ on an algebra $A$ is a collection of linear operators $(P_\omega)_{\omega\in\Omega}:A\ra A$ such that
\begin{equation*}
P_{\alpha}(a)P_{\beta}(b)=P_{\alpha\beta}\left( P_{\alpha}(a)b  + a P_{\beta}(b) + q ab \right),\, \text{ for }\, a, b \in A\,\text{ and }\, \alpha,\, \beta \in \Omega.
\mlabel{eq:RBF}
\end{equation*}
Then the pair $(A,\, (P_\omega)_{\omega\in\Omega})$ is called a {\bf Rota-Baxter family algebra} of weight $q$.
\mlabel{def:pp}
\end{defn}

\begin{exam}\cite{ZGM2}\label{exam:Rota-Baxter family}
Let~$\Omega=(\mathbb R,+)$ be a semigroup, and $R$ is the $\mathbb R$-algebra consisting of all continous functions from $\mathbb R$ to $\mathbb R$, the multiplication is defined by
 \[(fg)(x):=f(x)g(x),\text{ for }\,f,g\in R.\]
 And for any $\alpha\in\Omega$, define a family linear operator~$P_\alpha:R\ra R$ as follows
\[P_\alpha(f)(x)=e^{-\alpha A(x)}\int_0^{x} e^{\alpha A(t)}f(t)\,dt,\,\text{ for }\, \alpha\in\Omega,\]
where $A$ is a fixed nonzero element of $R$. Then $(P_\alpha)_{\alpha\in\Omega}$ is a Rota-Baxter family of weight $0$. In fact, for any $f,g\in R$, we have
\begin{align*}
&P_{\alpha+\beta}\Big(P_\alpha(f)g+fP_\beta(g)\Big)(x)=e^{-(\alpha+\beta)A(x)}\int_0^x e^{(\alpha+\beta)A(t)}
\Big(P_\alpha(f)(t)g(t)+f(t)P_\beta(g)(t)\Big)\,dt\\
&= e^{-(\alpha+\beta)A(x)}\int_0^x e^{(\alpha+\beta)A(t)}
\Biggl(g(t) e^{-\alpha A(t)}\int_0^t e^{\alpha A(s)}f(s)\,ds+f(t) e^{-\beta A(t)}\int_0^t e^{\beta A(s)}g(s)ds\Biggl)\,dt\\
&= e^{-(\alpha+\beta)A(x)}\Biggl(\int_0^xg(t) e^{\beta A(t)}\Big(\int_0^t e^{\alpha A(s)}f(s)\,ds\Big)\,dt+\int_0^x e^{\alpha A(t)}f(t)\Big(\int_0^t e^{\beta A(s)}g(s)\,ds\Big)\,dt\Biggl)\\
&= e^{-(\alpha+\beta)A(x)}\left(\iint_{0\le s\le t\le x}e^{\beta A(t)}g(t) e^{\alpha A(s)}f(s)\,dsdt+ \iint_{0\le t\le s\le x}e^{\beta A(t)}g(t) e^{\alpha A(s)}f(s)\,dsdt\right)\\
&= e^{-(\alpha+\beta)A(x)}\int_0^x e^{\alpha A(t)}f(t)\,dt\int_0^x e^{\beta A(t)}g(t)\,dt=P_\alpha f(x)P_\beta g(x).
\end{align*}
Hence $P_{\alpha+\beta}\Big(P_\alpha(f)g+fP_\beta(g)\Big)=P_\alpha (f)P_\beta (g).$
\end{exam}



\begin{defn}\cite{Kac1998} An associative conformal algebra is a $\mathbb{F}[\partial]$-module $A$ endowed with a $\lambda$-multiplication $a_\lambda b$ which defines a $\mathbb{F}$-linear map
\[\mu_{\lambda}: A\ot A\ra A[\lambda],\quad (a, b)\mapsto a\cdot_\lambda b:= a_\lambda b,\,\text{ for }\, a,b\in A,\]
satisfying the following axiom
\begin{align*}
(\partial a)_\lambda b=&-\lambda a_\lambda b,\quad a_\lambda\partial b=(\partial+\lambda)a_\lambda b,\quad\text{(conformal sesquilinearity-I)}\\
&(a_\lambda b)_{\lambda+\mu}c=a_\lambda(b_\mu c).\quad \text{(associativity)}
\end{align*}
\end{defn}

A conformal algebra is called {\bf finite} if it is finitely generated as  $\mathbb{F}[\partial]$-module. The {\bf rank} of a conformal algebra $A$ is its rank as a $\mathbb{F}[\partial]$-module.

\begin{remark}\cite{Kac1998}
The following properties always hold in an associative conformal algebra:
$$
\begin{aligned}
a_{\lambda}\big(b_{-\partial-\mu} c\big) &=\big(a_{\lambda} b\big)_{-\partial-\mu} c, \\
a_{-\partial-\lambda}\big(b_{\mu} c\big) &=\big(a_{-\partial-\mu} b\big)_{-\partial+\mu-\lambda} c, \\
a_{-\partial-\lambda}\big(b_{-\partial-\mu} c\big) &=\big(a_{-\partial+\mu-\lambda} b\big)_{-\partial-\mu} c .
\end{aligned}
$$
\end{remark}


\begin{exam}\cite{HB1}\label{exam:assr}
 Let $(A, \cdot)$ be an associative algebra. Then $\operatorname{Cur}(A)=\mathbb{F}[\partial] \otimes A$ is an associative conformal algebra with the following $\lambda$-product:
$$
(p(\partial) a)_{\lambda}(q(\partial) b)=p(-\lambda) q(\lambda+\partial)(a \cdot b),\,\text{ for }\, p(\partial), q(\partial) \in \mathbb{F}[\partial]\,\text{ and }\,a, b \in A.
$$
\end{exam}

The author in~\cite{HB1} introduced the associative conformal Yang-Baxter equation as a conformal analogue of the associative equation and interpret it in terms of its operator forms by introducing the notion of $\mathcal{O}$-operators (generalized Rota-Baxter operator) of associative conformal algebras. They also proved an antisymmetric solution of the associative conformal Yang-Baxter equation corresponds to the skew-symmetric part of a conformal linear map $P$.


\begin{defn}\cite{HB1}
Let $A$ be an associative conformal algebra. If $P:A\ra A$ is a $\mathbb{F}[\partial]$-module homomorphism satisfying
\begin{equation*}
P(a)_\lambda P(b)=P(P(a)_\lambda b)+P(a_\lambda P(b))+qP(a_\lambda b),\,\text{ for }\,a,b\in A,
\end{equation*}
then $P$ is called a Rota-Baxter operator of weight $q$ on $A$, where $q\in\mathbb{F}.$
\end{defn}


 \begin{exam}\cite{HB1}
 Let $A$ be a finite associative conformal algebra which is free as a $\mathbb{F}[\partial]$-module and $r\in A\ot A$ be a symmetric Frobenius conformal algebra, that is, $A$ has a non-degenerate symmetric invariant conformal bilinear form, then $r$ is a solution of associative conformal Yang-Baxter equation if and only if $P_{0}^{r}$ is a Rota-Baxter operator of weight zero on $A$.
 \end{exam}
 Rota-Baxter family algebra is a generalization of Rota-Baxter algebra, which arises naturally in renormalization of quantum field theory~\cite{FBP}. Recently, some scholars begin to pay attention to the study of family algebra structures, such as Aguiar~\cite{Aguiar}, Foissy~\cite{Foissy} and so on. Now we propose the concept of Rota-Baxter family conformal algebras of weight $q$.
\begin{defn}\label{defn:rbfconformal}
Let $\Omega$ be a semigroup and $A$ an associative conformal algebra. If $(P_\omega)_{\omega\in\Omega}:A\ra A$ is a family of $\mathbb{F}[\partial]$-module homomorphisms satisfying
\begin{equation*}
P_\alpha(a)_\lambda P_\beta(b)=P_{\alpha\beta}(P_\alpha(a)_\lambda b+a_\lambda P_\beta(b)+q a_\lambda b),\,\text{ for }\,a,b\in A,
\end{equation*}
then $(P_\omega)_{\omega\in\Omega}$ is called {\bf a Rota-Baxter family operator of weight $q$} on $A$, where $q\in\mathbb{F}.$
\end{defn}

\begin{exam}
 Let $(R,(P_\alpha)_{\alpha\in\Omega})$ be a Rota-Baxter family algebra defined just as Example~\ref{exam:Rota-Baxter family}. Define
\[f\cdot_\lambda g:= fg,\,\text{ for }\,f,g\in R \text{ and  satisfy}\,P_\alpha\partial=\partial P_\alpha.\]
Then $\operatorname{Cur}(R)=\mathbb{R}[\partial] \otimes R$ endowed with a natural Rota-Baxter family conformal associative algebra and $(P_\alpha)_{\alpha\in\Omega}$ is a Rota-Baxter family of weight $0$.
\end{exam}

The following Propositions provide the way to construct Rota-Baxter family conformal algebras.
\begin{prop}
Let $\Omega$ be a semigroup and $\mathbb{F}\Omega$ a semigroup algebra.
Then $(A,\mu_\lambda,(P_\omega)_{\omega\in\Omega})$ is a Rota-Baxter family conformal algebra of weight $q$, if and only if,
$(A \ot\mathbb{F}\Omega,P)$ is a Rota-Baxter conformal algebra of weight $q$, where $P: A\ot\mathbb{F}\Omega\ra A\ot\mathbb{F}\Omega, a\ot \omega\mapsto P_\omega(a)\ot\omega.$
\end{prop}

\begin{proof}
For $x,y\in A$ and $\alpha,\beta\in \Omega,$ we have
\begin{align*}
P(x\ot\alpha)_\lambda P(y\ot\beta)&= \bigl(P_\alpha(x)\ot\alpha\bigr)_\lambda (P_\beta(y)\ot \beta)
= \bigl(P_\alpha(x)_\lambda P_\beta(y)\bigr)\ot\alpha\beta\\
&= P_{\alpha\beta}\bigl({P_\alpha(x)}_\lambda y+x_\lambda P_\beta(y)+q x_\lambda y\bigr)\ot \alpha\beta\\
&= P\Bigl(\bigl({P_\alpha(x)}_\lambda y+x_\lambda P_\beta(y)+q x_\lambda y\bigr)\ot\alpha\beta\Bigr)\\
&= P\bigl({P_\alpha(x)}_\lambda y\ot\alpha\beta+x_\lambda P_\beta(y)\ot\alpha\beta+q x_\lambda y\ot\alpha\beta\bigr)\\
&= P\Bigl(\bigl(P_\alpha(x)\ot\alpha\bigr)_\lambda(y\ot \beta)+(x\ot\alpha)_\lambda\bigl(P_\beta(y)\ot\beta\bigr)+q(x\ot\alpha)_\lambda(y\ot\beta)\Bigr)\\
&= P\bigl(P(x\ot\alpha)_\lambda(y\ot\beta)+(x\ot\alpha)_\lambda P(y\ot\beta)+qa(x\ot\alpha)_\lambda(y\ot\beta)\bigr),
\end{align*}
as required.
\end{proof}

\begin{exam}
Based on Example~\ref{exam:assr}, let $\Omega$ be a semigroup. Then $(\operatorname{Cur}(A)=\mathbb{F}[\partial] \otimes A, (P_\omega)_{\omega\in\Omega})$ is a Rota-Baxter family conformal algebra, if and only if, $P$ is a Rota-Baxer operator on $\operatorname{Cur}(A\ot\mathbb{F}\Omega)=\mathbb{F}[\partial] \otimes A\ot\mathbb{F}\Omega$  with the following $\lambda$-product:
\begin{align*}
P(p(\partial) a\ot\alpha):={}&P_\alpha(p(\partial) a)\ot\alpha,\\
(p(\partial) a\ot\alpha)\cdot_{\lambda}(q(\partial) b\ot\beta)={}&p(-\lambda) q(\lambda+\partial)(a \cdot b)\ot\alpha\beta,
\end{align*}
where $p(\partial), q(\partial) \in \mathbb{F}[\partial]$ and $a, b \in A,\alpha,\beta\in\Omega.$
\end{exam}

\begin{prop}
Let $(A,\,\mu_{A,\,\lambda}, (P_\omega)_{\omega\in\Omega})$ and $(B, \mu_{B,\,\lambda}, (P_\omega)_{\omega\in\Omega})$ be two Rota-Baxter family conformal algebras. The tensor product $(A \otimes B, \mu_{A\ot B, \lambda}, (P_\omega)_{\omega\in\Omega})$ is still a Rota-Baxter family conformal algebra with the multiplication
$$
\mu_{A \otimes B, \lambda}:(A \otimes B) \otimes(A \otimes B) \xrightarrow{\mathrm{id}_{A} \otimes \tau \otimes \mathrm{id}_{B}} A \otimes A \otimes B \otimes B\xrightarrow{\mu_{A,\,\lambda}\ot \mu_{B,\,\lambda}} A \otimes B[\lambda],
$$
$$
\left(a_{1} \otimes b_{1}\right) \otimes\left(a_{2} \otimes b_{2}\right) \mapsto\left(a_{1}\cdot_\lambda a_{2}\right) \otimes\left(b_{1}\cdot_\lambda b_{2}\right), a_{1}, a_{2} \in A, b_{1}, b_{2} \in B \text {, }
$$
with $\tau: B \otimes A \rightarrow A \otimes B, b \otimes a \mapsto a \otimes b$ being the switch map,
and define
\[P_\omega(a_1\ot b_1):=P_\omega(a_1)\ot b_1.\]
\end{prop}

\begin{proof}
We only need to prove the Rota-Baxter family conformal relation, that is
\[P_\alpha(a_1\ot b_1)\cdot_\lambda P_\beta(a_2\ot b_2)
=P_{\alpha\beta}\Big(P_\alpha(a_1\ot b_1)\cdot_\lambda (a_2\ot b_2)+(a_1\ot b_1)\cdot_\lambda P_\beta(a_2\ot b_2)+q(a_1\ot b_1)\cdot_\lambda (a_2\ot b_2)\Big).\]
 We have
\begin{align*}
&P_\alpha(a_1\ot b_1)\cdot_\lambda P_\beta(a_2\ot b_2)\\
={}&(P_\alpha(a_1)\ot b_1)\cdot_\lambda (P_\beta(a_2)\ot b_2)
=\Big(P_\alpha(a_1)\cdot_\lambda P_\beta(a_2)\Big)\ot(b_1\cdot_\lambda b_2)\\
={}&P_{\alpha\beta}\Big((P_\alpha(a_1)\cdot_\lambda a_2)
+(a_1\cdot_\lambda P_\beta(a_2))
+q(a_1\cdot_\lambda a_2)\Big)\ot(b_1\cdot_\lambda b_2)\\
={}&P_{\alpha\beta}\Big(P_\alpha(a_1\ot b_1)\cdot_\lambda(a_2\ot b_2)
+(a_1\ot b_1)\cdot_\lambda P_\beta(a_2\ot b_2)
+q(a_1\ot b_1)\cdot_\lambda (a_2\ot b_2)\Big).\tag*{\qedhere}
\end{align*}
\end{proof}

\begin{defn}\cite{ZGGS}
Let $q$ be a given element in $\mathbb{F}$. A {\bf differential algebra of weight $q$} is a pair $(A, d)$ consisting of an algebra $A$
with a linear operator $d: A \rightarrow A$ that satisfies the differential equation
\[
d(ab) = d(a)b + ad(b) + q d(a)d(b),\text{ for }\, a,b\in A.
\]
Then $d$ is called a {\bf differential operator of weight $q$}.
\end{defn}

\begin{defn}
Let $\Omega$ be a semigroup. A {\bf differential family conformal algebra of weight $q$} is a pair $(A, (d_\omega)_{\omega\in\Omega})$ consisting of an associative conformal algebra $A$
with a family of linear operators $(d_\omega)_{\omega\in\Omega}: A \rightarrow A$, satisfying $\mathbb{F}[\partial]$-module homomorphism
\begin{equation}\label{eq:differential family conformal algebras}
d_{\alpha\beta}(a_\lambda b) = d_\alpha (a)_\lambda b + a_\lambda d_\beta(b) + q d_\alpha(a)_\lambda d_\beta(b),\text{ for }\, a,b\in A.
\end{equation}
Then $(d_\omega)_{\omega\in\Omega}$ is called a {\bf differential family operator of weight $q$}.
\end{defn}

\begin{prop}
Let $\Omega$ be a semigroup and  $(d_\omega)_{\omega\in\Omega}$  a family of invertible maps. Let $(A, (d_\omega)_{\omega\in\Omega})$ be a differential family conformal algebra of weight $q$. Then the collection $(d^{-1}_\omega)_{\omega\in\Omega}$ is a Rota-Baxter family operator of weight $q$.
\end{prop}

\begin{proof}
Since $(d_\omega)_{\omega\in\Omega}$ is an invertible family, suppose that $a=d^{-1}_\alpha(u)$ and $b=d^{-1}_\beta(v)$ in Eq.~(\ref{eq:differential family conformal algebras}). Then we have
\[d_{\alpha\beta}(d^{-1}_\alpha(u)_\lambda d^{-1}_\beta(v))=u_\lambda d^{-1}_\beta(v)+d^{-1}_\alpha(u)_\lambda v+qu_\lambda v,\]
which is equivalent to
\[d^{-1}_\alpha(u)_\lambda d^{-1}_\beta(v)=d^{-1}_{\alpha\beta}(d^{-1}_\alpha(u)_\lambda v+u_\lambda d^{-1}_\beta(v)+qu_\lambda v).\]
This shows that the collection $(d^{-1}_\omega)_{\omega\in\Omega}$ is a Rota-Baxter family operator of weight $q$.
\end{proof}

\subsection{Rota-Baxter family conformal modules}
In this subsection, we first recall conformal module from~\cite{Kac1998, HB1, HB2}, then we give the definition of Rota-Baxter family conformal bimodule and construct a new Rota-Baxter family conformal algebra via semi-direct product.

\begin{defn}\cite{Kac1998}\label{defn:associative conformal algebra}
{\rm  Let $A$ be an associative conformal algebra and $M$ a $\mathbb{F}[\partial]$-module.
\begin{enumerate}
\item A left module $M$ over an associative conformal algebra $A$ is a $\mathbb{F}[\partial]$-module endowed with a $\mathbb{F}$-bilinear map $A\times M \rightarrow M[\lambda]$, $(a,m)\mapsto a \cdot_\lambda m$, and for all $a,b\in A$ and $m\in M$ satisfying
\begin{align*}
(\partial a)\cdot_\lambda m&=-\lambda a\cdot_\lambda m,\\
a\cdot_\lambda (\partial m)&=(\partial+\lambda)(a\cdot_\lambda m),\\
(a\cdot_\lambda b)\cdot_{\lambda+\mu} m&=a\cdot_\lambda(b\cdot_\mu m).
\end{align*}

\item A right module $M$ over an associative conformal algebra $A$ is a $\mathbb{F}[\partial]$-module endowed with a $\mathbb{F}$-bilinear map $M\times A \rightarrow M[\lambda]$, $(m,a)\mapsto m \cdot_\lambda a$, and for all $a,b\in A$ and $m\in M$ satisfying
\begin{align*}
(\partial m)\cdot_\lambda a&=-\lambda m\cdot_\lambda a,\\
m\cdot_\lambda(\partial a)&=(\partial+\lambda)(m\cdot_\lambda a),\\
(m\cdot_\lambda a)\cdot_{\lambda+\mu} b&=m\cdot_\lambda(a\cdot_\mu b).
\end{align*}

\item $M$ is called a conformal bimodule of $A$ (or conformal $A$-bimodule) if it is both a left conformal module and a right conformal module, and for all $a,b\in A$ and $m\in M$ satisfying the following compatible condition
\begin{eqnarray*}\label{B}
(a\cdot_\lambda m)\cdot_{\lambda+\mu} b=a\cdot_\lambda(m\cdot_\mu b).
\end{eqnarray*}
\end{enumerate}}
\end{defn}

\begin{defn}\label{defn:conformal homomorphism}\cite{Kac1998}
Let $M, N$ be two $\mathbb{F}[\partial]$-modules.  {\bf A conformal linear map} from  $M$ to $N$ is a $\mathbb{F}$-linear map $f: M \rightarrow N[\lambda]$, denoted $f_{\lambda}: M \rightarrow N$, such that $f_{\lambda} \partial=(\partial+\lambda) f_{\lambda}$. The space of such maps is denoted by $\operatorname{Chom}(M, N)$. It has canonical structures of a $\mathbb{F}[\partial]$-module
\begin{align*}
(\partial f)_{\lambda}&=-\lambda f_{\lambda},\,\text{ for }\,f \in \operatorname{Chom}(M, N).
\end{align*}
\end{defn}
In the special case $M$, let $\operatorname{Cend}(M)=\operatorname{Chom}(M, M)$ denote the space of conformal linear endomorphisms of $M$. Then $\operatorname{Cend}(M)$ carries the natural structure of an associative conformal algebra, defined by
$$
\left(f_{\lambda} g\right)_{\mu} m:=f_{\lambda}\big(g_{\mu-\lambda} m\big), \quad f, g \in \operatorname{Cend}(M), m \in M.
$$

\begin{remark}\cite{HB1}
Set $a \cdot_{\lambda} v=l_{A}(a)_{\lambda} v$ and $v \cdot_{\lambda} a=r_{A}(a)_{-\lambda-\partial} v$. Then a structure of a bimodule $M$ over an associative conformal algebra $A$ is the same as two $\mathbb{F}[\partial]$-module homomorphisms $l_{A}$ and $r_{A}$ from $A$ to $\operatorname{Cend}(M)$ such that the following conditions hold:
\begin{align*}
l_{A}\left(a_{\lambda} b\right)_{\lambda+\mu} v&=l_{A}(a)_{\lambda}(l_{A}(b)_{\mu} v), \\
r_{A}(b)_{-\lambda-\mu-\partial}(r_{A}(a)_{-\lambda-\partial} v)&=r_{A}(a_{\mu} b)_{-\lambda-\partial} v ,\\
l_{A}(a)_{\lambda}(r_{A}(b)_{-\mu-\partial} v)&=r_{A}(b)_{-\lambda-\mu-\partial}(l_{A}(a)_{\lambda} v),
\end{align*}
for all $a,b\in A$ and $v\in M$. Denoted by this bimodule by $(M,l_A,r_A).$
\end{remark}

\begin{defn}\label{Def: Rota-Baxter family bimodules}
Let $(A,\mu_\lambda,(P_\omega)_{\omega\in\Omega})$ be a Rota-Baxter family conformal algebras of weight $q$ and $M$ be a bimodule
over associative conformal algebra $(A,\mu_\lambda)$. We say that {\bf $M$ is a bimodule
over Rota-Baxter family conformal algebras} $(A,\mu_\lambda, (P_\omega)_{\omega\in\Omega})$  or a Rota-Baxter family conformal bimodule if $M$ is endowed with a family of
linear operators $(P_{M,\,\omega})_{\omega\in\Omega}: M\rightarrow M$ such that the following
equations
\begin{eqnarray}
\label{eq:left}P_\alpha(a)_\lambda P_{M,\,\beta}(m)&=&P_{M,\alpha\beta}\big(a_\lambda P_{M,\,\beta}(m)+P_\alpha(a)_\lambda m+q a_\lambda m\big),\\
\label{eq:right}P_{M,\,\alpha}(m)_\lambda P_\beta(a)&=&P_{M,\,\alpha\beta}\big(m_\lambda P_\beta(a)+P_{M,\,\alpha}(m)_\lambda a+q m_\lambda a\big),
\end{eqnarray}
hold for any $a\in A$ and $m\in M$.
\end{defn}
Obviously, $(A,\mu_\lambda, (P_\omega)_{\omega\in\Omega})$ itself is a bimodule over the Rota-Baxter family conformal algebras $(A,\mu_\lambda,(P_\omega)_{\omega\in\Omega}$), called the regular Rota-Baxter family conformal  bimodule.

\begin{prop}  \label{Prop: trivial extension of Rota-Baxter family bimodule}
Let $(A,\mu_\lambda,(P_\omega)_{\omega\in\Omega})$ be a Rota-Baxter family conformal algebras and $M$ be a bimodule over associative conformal algebra $(A,\mu_\lambda)$. Denoted by $\iota: A\to A\oplus M, a\mapsto (a, 0)$ and $\pi: A\oplus M\to A, (a, m)\mapsto a$.
\begin{enumerate}
\item \label{item:multiplication} Then $A\oplus M$
	 becomes  an associative conformal algebra whose  multiplication is
	\begin{eqnarray*}(a,m)\bullet_\lambda (b,n):=(a_\lambda b, a_\lambda n+m_\lambda b).\end{eqnarray*}
\item \label{item:Rota-Baxter conformal construction}   Then $A\oplus M$ is a Rota-Baxter family conformal algebra such that $\iota$ and $\pi$ are both morphisms of Rota-Baxter family conformal algebras if and only if $M$ is a Rota-Baxter family conformal bimodule over $A$.

 This new Rota-Baxter family conformal algebras will be denoted by $A\ltimes  M$, called the semi-direct product (or trivial extension) of $A$ by $M$.
 \end{enumerate}
	\end{prop}

\begin{proof}
It is the direct calculation.
\end{proof}

\subsection{$\Omega$-associative conformal algebras}

In this subsection, we mainly propose the concept of $\Omega$-associative conformal  algebra and construct new Rota-Baxter family bimodules from old ones.

\begin{defn}\label{defn:conformal algebras}
Let $\Omega$ be a semigroup. {\bf An $\Omega$-associative conformal algebra $A$} is a $\mathbb{F}[\partial]$-module endowed with a $\mathbb{F}$-bilinear map
\[\Mu\lambda,\alpha,\beta;: A\ot A\ra A[\lambda],\quad (a, b)\mapsto \Mu\lambda,\alpha,\beta;(a, b):=a\Cdot\lambda,\alpha,\beta; b,\]
and for all $a,b,c\in A$ satisfying the following axiom
\begin{align*}
(\partial a)\cdot_{\lambda;\,\alpha,\,\beta} b=&-\lambda a\cdot_{\lambda;\,\alpha,\,\beta} b,\quad a\cdot_{\lambda;\,\alpha,\,\beta}\partial b=(\partial+\lambda)a\cdot_{\lambda;\,\alpha,\,\beta} b,\quad\text{(conformal sesquilinearity-II)}\\
&(a\cdot_{\lambda;\,\alpha,\,\beta} b)\cdot_{\lambda+\mu;\,\alpha\beta,\,\gamma}c=a\cdot_{\lambda;\,\alpha,\,\beta\gamma}
(b\cdot_{\mu;\,\beta,\gamma} c).\quad \text{(associativity)}
\end{align*}
\end{defn}

\begin{defn}\label{defn:bimodule}
Let $(A , (\Mu \lambda,\alpha,\beta;)_{\alpha,\, \beta \in \Omega})$ be an $\Omega$-associative  conformal algebra. A {\bf bimodule} over it consists of a vector space $M$ together with a collection of bilinear maps and for $a, b \in A, m \in M \text{ and } \alpha, \beta, \gamma \in \Omega$ such that
\begin{align*}
  A \otimes M \rightarrow M[\lambda] ,& \quad (a, m) \mapsto a \Cdot \lambda,\alpha,\beta; m, \\
 M \otimes A \rightarrow M[\lambda],  &  \quad (m, a) \mapsto m \Cdot \lambda,\alpha,\beta; a,
\end{align*}
such that
\begin{align*}
(\partial a)\Cdot \lambda,\alpha,\beta;v=-\lambda a\Cdot \lambda,\alpha,\beta;v,&\quad a\Cdot \lambda,\alpha,\beta; (\partial v)=(\partial+\lambda)a\Cdot \lambda,\alpha,\beta;v,\\
(a \Cdot \lambda,\alpha,\beta; b) \Cdot \lambda+\mu,\alpha\beta,\gamma; m = a \Cdot \lambda,\alpha,\beta\gamma; (b \Cdot \mu,\beta,\gamma;m), &\quad
(a \Cdot \lambda,\alpha,\beta; m)\Cdot \lambda+\mu,\alpha\beta,\gamma;  b = a \Cdot \lambda,\alpha,\beta\gamma; (m \Cdot \mu,\beta,\gamma; b),\\
 (m \Cdot \lambda,\alpha,\beta; a)\Cdot \lambda+\mu,\alpha\beta,\gamma;  b &= m \Cdot \lambda,\alpha,\beta\gamma; (a \Cdot \mu,\beta,\gamma; b).
\end{align*}
\end{defn}

Combining Definition~\ref{defn:rbfconformal} and Definition~\ref{defn:conformal algebras}, we immediately obtain
\begin{equation*}
P_\alpha(a)\cdot_{\lambda;\,\alpha,\,\beta}P_\beta(b)
=P_{\alpha\beta}(P_\alpha(a)\cdot_{\lambda;\,\alpha,\,\beta}b
+a\cdot_{\lambda;\,\alpha,\,\beta}P_\beta(b)+q a\cdot_{\lambda;\,\alpha,\,\beta}b),
\end{equation*}
which we call the pair $(A, (\mu_{\lambda;\,\alpha,\,\beta})_{\alpha,\,\beta\in\Omega},(P_\omega)_{\omega\in\Omega})$ {\bf Rota-Baxter family $\Omega$-associative conformal algebra}.

 \begin{defn}
 Let $\left(A,(P_\omega)_{\omega\in\Omega}\right)$ and $\left(A^{\prime},(P'_\omega)_{\omega\in\Omega}\right)$ be two Rota-Baxter family $\Omega$-associative conformal algebras of weight $q$. A family of maps $(f_\omega)_{\omega\in\Omega}$ is called a Rota-Baxter family $\Omega$-associative conformal algebra morphism of weight $q$ if $f_\omega:A\ra B$ satisfy $f_{\omega}\circ P_\omega=P'_\omega\circ f_{\omega}$ and $f_{\,\alpha\beta}\circ\mu_{\lambda;\,\alpha,\,\beta}=\mu'_{\lambda;\,\alpha,\,\beta}\circ(f_{\alpha}\ot f_{\beta})$, for any $\omega, \alpha, \beta\in\Omega.$
 \end{defn}

  The following interesting result is useful and interesting.
\begin{prop} \label{Prop: new RBF algebra}
	Let $(A, (\mu_{\lambda;\,\alpha,\,\beta})_{\alpha,\,\beta\in\Omega}, (P_\omega)_{\omega\in\Omega})$ be a Rota-Baxter family $\Omega$-associative conformal algebras of weight $q$. Define a new binary operation as:
\begin{eqnarray*}
a\Star\lambda,\alpha,\beta; b:=a\cdot_{\lambda;\,\alpha,\,\beta} P_\beta(b)+P_\alpha(a)\cdot_{\lambda;\,\alpha,\,\beta} b+q a\cdot_{\lambda;\,\alpha,\,\beta} b
\end{eqnarray*}
for any $a,b\in A$. Then
\begin{enumerate}

\item   the family $(\Star\lambda,\alpha,\beta;)_{\alpha,\,\beta\in\Omega} $ is associative and  $(A, (\Star\lambda,\alpha,\beta;)_{\alpha,\,\beta\in\Omega})$ is a new $\Omega$-associative conformal algebra;

    \item the triple  $(A, (\Star\lambda,\alpha,\beta;)_{\alpha,\,\beta\in\Omega} ,(P_\omega)_{\omega\in\Omega})$ also forms a Rota-Baxter family $\Omega$-associative conformal algebras of weight $q$  and denote it by $A_{\star}$;

\item the family $(P_\omega)_{\omega\in\Omega}:(A,(\Star\lambda,\alpha,\beta;)_{\alpha,\,\beta\in\Omega}, (P_\omega)_{\omega\in\Omega})\rightarrow (A, (\cdot_{\lambda;\,\alpha,\,\beta})_{\alpha,\,\beta\in\Omega}, (P_\omega)_{\omega\in\Omega})$ is a  morphism of Rota-Baxter family $\Omega$-associative conformal algebras.
    \end{enumerate}
	\end{prop}

\begin{proof}
It is the direct calculation.
\end{proof}

One can also construct new Rota-Baxter family bimodules from old ones.

\begin{prop}\label{Prop:new-bimodule}
Let $\Omega$ be a semigroup. Let $(A, (\mu_{\lambda;\,\alpha,\,\beta})_{\alpha,\,\beta\in\Omega},(P_\omega)_{\omega\in\Omega})$ be a Rota-Baxter family $\Omega$-associative conformal algebra of weight $q$ and $(M, (P_{M,\,\omega})_{\omega\in\Omega})$ be a Rota-Baxter family $\Omega$-associative conformal bimodule over it.  For $a,b\in A,m\in M$, and $\alpha,\beta\in\Omega,$ define a family of left actions $(\Rhd\lambda,\alpha,\beta;)_{\alpha,\,\beta\in\Omega} $ and a family of right actions $(\Lhd\lambda,\alpha,\beta;)_{\alpha,\,\beta\in\Omega}$ as follows:
\begin{align}
\label{eq:right module}\Rhd \lambda,\alpha,\beta; : A \otimes M \rightarrow M, \quad & a \Rhd \lambda,\alpha,\beta; m := P_\alpha(a) \cdot_{\lambda;\,\alpha,\,\beta} m - P_{M,\,\alpha \beta} (a \cdot_{\lambda;\,\alpha,\,\beta} m),\\
\label{eq:left module}\Lhd \lambda,\alpha,\beta; : M \otimes A \rightarrow M ,\quad & m \Lhd \lambda,\alpha,\beta; a := m \cdot_{\lambda;\,\alpha,\,\beta} P_\beta(a) - P_{M,\,\alpha \beta} (m \cdot_{\lambda;\,\alpha,\,\beta} a).
\end{align}
Then $M$ is a Rota-Baxter family $\Omega$-associative bimodule over $A_{\star}.$
\end{prop}

\begin{proof}
It is easy to see that the two $\lambda$-actions defined by Eqs.~(\ref{eq:right module}-\ref{eq:left module}) are conformal sequilinear maps.
	Firstly, we show that $(M, (\Rhd \lambda,\alpha,\beta;)_{\alpha,\,\beta\in\Omega})$ is a left module over $(A, (\Star\lambda,\alpha,\beta;)_{\alpha,\,\beta\in\Omega})$. On the one hand, we have
	\begin{align*}
	&a\Rhd \lambda,\alpha,\beta\gamma;(b\Rhd \mu,\beta,\gamma;m)\\
={}& a\Rhd \lambda, \alpha,\beta\gamma;(P_\beta(b)\cdot_{\mu;\,\beta,\,\gamma} m-P_{M,\,\beta\gamma}(b\cdot_{\mu;\,\beta,\,\gamma} m))\\
	={}&P_\alpha(a)\Cdot \lambda,\,\alpha,\,\beta\gamma;\Big(P_\beta(b)\Cdot\mu,\,\beta,\,\gamma; m-P_{M,\,\beta\gamma}(b\Cdot\mu,\,\beta,\,\gamma; m)\Big)\\
&-P_{M,\,\alpha\beta\gamma}\Big(a\Cdot\lambda,\alpha,\beta\gamma;(P_\beta(b)\Cdot\mu,\beta,\gamma; m)-a\Cdot\lambda,\alpha,\beta\gamma; P_{M,\,\beta\gamma}(b\Cdot\mu,\beta,\gamma; m)\Big)\\
	={}&P_\alpha(a)\Cdot\lambda,\alpha,\beta\gamma;\Big(P_\beta(b)\Cdot\mu,\beta,\gamma;m\Big)
-\underline{P_\alpha(a)\cdot_{\lambda;\,\alpha,\,\beta\gamma} P_{M,\,\beta\gamma}(b\cdot_\mu m)}\\
&-P_{M,\,\alpha\beta\gamma}\Big(a\cdot_{\lambda;\,\alpha,\,\beta\gamma} (P_\beta(b)\Cdot\mu,\beta,\gamma; m)-a \Cdot \lambda,\alpha,\beta\gamma; P_{M,\,\beta\gamma}(b \Cdot \mu,\beta,\gamma; m)\Big)\\
&\hspace{5cm}(\text{by Eq.~(\ref{eq:left}) of the underline item})\\
	={}&P_\alpha(a) \Cdot \lambda,\alpha,\beta\gamma;\Big(P_\beta(b) \Cdot \mu,\beta,\gamma; m\Big)-P_{M,\,\alpha\beta\gamma}\Big(P_\alpha(a) \Cdot \lambda,\alpha,\beta\gamma;(b \Cdot \mu,\beta,\gamma; m)\Big)\\
&-P_{M,\,\alpha\beta\gamma}
\Big(a \Cdot \lambda,\alpha,\beta\gamma;(P_\beta(b) \Cdot \mu,\beta,\gamma; m)\Big)-qP_{M,\,\alpha\beta\gamma}\Big(a \Cdot \lambda,\alpha,\beta\gamma;(b \Cdot \mu,\beta,\gamma; m)\Big).
\end{align*}
 On the other hand, we have
 \begin{align*}
	&(a\Star\lambda,\alpha,\beta; b)\Rhd \lambda+\mu,\alpha\beta,\gamma; m\\
={}&P_{\alpha\beta}(a\Star\lambda,\alpha,\beta;  b)\cdot_{\lambda+\mu;\,\alpha\beta,\,\gamma} m-P_{M,\alpha\beta\gamma}\big((a\Star\lambda,\alpha,\beta; b)\cdot_{\lambda+\mu;\,\alpha\beta,\,\gamma} m\big)\\
	={}&(P_\alpha(a)\Cdot\lambda,\alpha,\beta; P_\beta(b))\Cdot\lambda+\mu,\alpha\beta,\gamma; m-P_{M,\alpha\beta\gamma}\Big((a\Cdot\lambda,\alpha,\beta; P_\beta(b))\Cdot\lambda+\mu,\alpha\beta,\gamma; m+(P_\alpha(a)\Cdot\lambda,\alpha,\beta; b)\Cdot\lambda+\mu,\alpha\beta,\gamma; m\\
&+q (a\Cdot\lambda,\alpha,\beta; b)\Cdot\lambda+\mu,\alpha\beta,\gamma; m \Big).
	\end{align*}
Applying Definition~\ref{defn:associative conformal algebra} and comparing the items of both sides, we have
\[	a\Rhd \lambda,\alpha,\beta\gamma;(b\Rhd \mu,\beta,\gamma;m)=(a\Star\lambda,\alpha,\beta;  b)\Rhd \lambda+\mu,\alpha\beta,\gamma; m.\]
Thus the family $(\Rhd\lambda, \alpha,\beta;)_{\alpha,\,\beta\in\Omega}$ makes $M$ into a left module over $(A, (\Star \lambda,\alpha,\beta;)_{\alpha,\,\beta\in\Omega})$. Similarly, one can check that the family $(\Rhd\lambda, \alpha,\beta;)_{\alpha,\,\beta\in\Omega}$ defines a right module structure on $M$ over $(A, (\Star \lambda,\alpha,\beta;)_{\alpha,\,\beta\in\Omega})$.

Now, we are going to check the compatibility of operations $(\Rhd\lambda, \alpha,\beta;)_{\alpha,\,\beta\in\Omega}$ and $(\Lhd\lambda, \alpha,\beta;)_{\alpha,\,\beta\in\Omega}$. On the one hand, we have
\begin{align*}
	&(a\Rhd \lambda,\alpha,\beta; m)\Lhd \lambda+\mu,\alpha\beta,\gamma; b\\
={}& \Big(P_\alpha(a)\Cdot \lambda,\alpha,\beta; m-P_{M,\alpha\beta}(a\Cdot \lambda,\alpha,\beta; m)\Big)\Lhd \lambda+\mu,\alpha\beta,\gamma; b\\
	={}& \Big(P_\alpha(a)\Cdot \lambda,\alpha,\beta; m-P_{M,\alpha\beta}(a\Cdot \lambda,\alpha,\beta; m)\Big)\Cdot \lambda+\mu,\alpha\beta,\gamma;P_\gamma(b)-P_{M,\alpha\beta\gamma}\big((P_\alpha(a)\Cdot \lambda,\alpha,\beta; m)\Cdot \lambda+\mu,\alpha\beta,\gamma; b\\
&-P_{M,\alpha\beta}(a\Cdot \lambda,\alpha,\beta; m)\Cdot \lambda+\mu,\alpha\beta,\gamma; b\big)\\
={}& \Big(P_\alpha(a)\Cdot \lambda,\alpha,\beta; m\Big)\Cdot \lambda+\mu,\alpha\beta,\gamma;P_\gamma(b)-\underline{P_{M,\alpha\beta}(a\Cdot \lambda,\alpha,\beta; m)\Cdot \lambda+\mu,\alpha\beta,\gamma;P_\gamma(b)}\\
&-P_{M,\alpha\beta\gamma}\big((P_\alpha(a)\Cdot \lambda,\alpha,\beta; m)\Cdot \lambda+\mu,\alpha\beta,\gamma; b
-P_{M,\alpha\beta}(a\Cdot \lambda,\alpha,\beta; m)\Cdot \lambda+\mu,\alpha\beta,\gamma; b\big)\\
&\hspace{3cm}(\text{by Eq.~(\ref{eq:right}) of the underline item})\\
={}&(P_\alpha(a)\Cdot \lambda,\alpha,\beta; m)\cdot _{\lambda+\mu;\,\alpha\beta,\,\gamma}P_\gamma(b)-P_{M,\alpha\beta\gamma}\big(P_{M,\alpha\beta}(a\Cdot \lambda,\alpha,\beta; m)\Cdot \lambda+\mu,\alpha\beta,\gamma; b+(a\Cdot \lambda,\alpha,\beta; m)\Cdot \lambda+\mu,\alpha\beta,\gamma; P_\gamma(b)\\
&+q (a\Cdot \lambda,\alpha,\beta; m)\Cdot \lambda+\mu,\alpha\beta,\gamma; b\big)
-P_{M,\alpha\beta\gamma}\big((P_\alpha(a)\Cdot \lambda,\alpha,\beta; m)\Cdot \lambda+\mu,\alpha\beta,\gamma; b\big)\\
&+P_{M,\alpha\beta\gamma}\big(P_{M,\alpha\beta}(a\Cdot \lambda,\alpha,\beta; m)\Cdot \lambda+\mu,\alpha\beta,\gamma;b\big)\\
={}& (P_\alpha(a)\Cdot \lambda,\alpha,\beta; m)\Cdot \lambda+\mu,\alpha\beta,\gamma; P_\gamma(b)-P_{M,\alpha\beta\gamma}\big((a\Cdot \lambda,\alpha,\beta; m)\Cdot \lambda+\mu,\alpha\beta,\gamma; P_\gamma(b)\big)\\
&-P_{M,\alpha\beta\gamma}((P_\alpha(a)\Cdot \lambda,\alpha,\beta; m)\Cdot \lambda+\mu,\alpha\beta,\gamma; b)-q P_{M,\alpha\beta\gamma}((a\Cdot \lambda,\alpha,\beta; m)\Cdot \lambda+\mu,\alpha\beta,\gamma; b).
\end{align*}
On the other hand, we have
\begin{align*}
&a\Rhd\lambda,\alpha,\beta\gamma;(m\Lhd\mu,\beta,\gamma;b)\\
={}&a\Rhd\lambda,\alpha,\beta\gamma;(m\Cdot \mu,\beta,\gamma; P_\gamma(b)-P_{M,\,\beta\gamma}(m\Cdot \mu,\beta,\gamma; b))\\
={}&P_\alpha(a)\Cdot \lambda,\alpha,\beta\gamma;\big(m\Cdot \mu,\beta,\gamma; P_\gamma(b)-P_{M,\,\beta\gamma}(m\Cdot \mu,\beta,\gamma; b)\big)-P_{M,\alpha\beta\gamma}\big(a\Cdot \lambda,\alpha,\beta\gamma;(m\Cdot \mu,\beta,\gamma; P_\gamma(b))\\
&-a\Cdot \lambda,\alpha,\beta\gamma; P_{M,\,\beta\gamma}(m\Cdot \mu,\beta,\gamma; b)\big)\\
={}&P_\alpha(a)\Cdot \lambda,\alpha,\beta\gamma;\big(m\Cdot \mu,\beta,\gamma; P_\gamma(b)\big)-\underline{P_\alpha(a)\Cdot \lambda,\alpha,\beta\gamma;\big(P_{M,\,\beta\gamma}(m\Cdot \mu,\beta,\gamma; b)\big)}-P_{M,\alpha\beta\gamma}\big(a\Cdot \lambda,\alpha,\beta\gamma;(m\Cdot \mu,\beta,\gamma; P_\gamma(b))\\
&-a\Cdot \lambda,\alpha,\beta\gamma; P_{M,\,\beta\gamma}(m\Cdot \mu,\beta,\gamma; b)\big)\quad(\text{by Eq.~(\ref{eq:left}) of the underline item})\\
={}&P_\alpha(a)\Cdot \lambda,\alpha,\beta\gamma; (m\Cdot \mu,\beta,\gamma; P_\gamma(b)-P_{M,\alpha\beta\gamma}\big(a\Cdot \lambda,\alpha,\beta\gamma; P_{M,\,\beta\gamma}(m\Cdot \mu,\beta,\gamma; b)+P_\alpha(a)\Cdot \lambda,\alpha,\beta\gamma;(m\Cdot \mu,\beta,\gamma; b)\\
&+q a\Cdot \lambda,\alpha,\beta\gamma; (m\Cdot \mu,\beta,\gamma; b)\big)
-P_{M,\alpha\beta\gamma}\big(a\Cdot \lambda,\alpha,\beta\gamma;(m\Cdot \mu,\beta,\gamma; P_\gamma(b))\big)+P_{M,\alpha\beta\gamma}\big(a\Cdot \lambda,\alpha,\beta\gamma; P_{M,\,\beta\gamma}(m\Cdot \mu,\beta,\gamma; b)\big)\\
={}& P_\alpha(a)\cdot_{\lambda;\,\alpha,\,\beta\gamma}(m\Cdot \mu,\beta,\gamma; P_\gamma(b))-P_{M,\alpha\beta\gamma}\big(P_\alpha(a)\Cdot \lambda,\alpha,\beta\gamma; (m\Cdot \mu,\beta,\gamma; b)\big)-P_{M,\alpha\beta\gamma}\big(a\Cdot \lambda,\alpha,\beta\gamma; (m\Cdot \mu,\beta,\gamma; P_\gamma(b))\big)\\
&-q P_{M,\alpha\beta\gamma}\big(a\Cdot \lambda,\alpha,\beta\gamma; (m\Cdot \mu,\beta,\gamma; b)\big).
\end{align*}
Thus we have
\[(a\Rhd \lambda,\alpha,\beta; m)\Lhd \lambda+\mu,\alpha\beta,\gamma; b=a\Rhd\lambda,\alpha,\beta\gamma;(m\Lhd\mu,\beta,\gamma;b),\]
that is, operations $(\Rhd\lambda,\alpha,\beta;)_{\alpha,\,\beta\in\Omega}$ and $(\Lhd\lambda,\alpha,\beta;)_{\alpha,\,\beta\in\Omega}$ make $M$ into a bimodule over associative conformal algebra $(A, (\Star\lambda,\alpha,\beta;)_{\alpha,\,\beta\in\Omega})$.
Finally, we show that $ M$ is a Rota-Baxter family $\Omega$-associative conformal bimodule over $A_{\star} $. That is, for any $a\in A$ and $m\in M$,
\begin{align*}
 P_\alpha(a)\Rhd\lambda,\alpha,\beta;  P_{M,\,\beta}(m)&=P_{M,\alpha\beta}\big(a\Rhd\lambda,\alpha,\beta; P_{M,\,\beta}(m)+P_\alpha(a)\Rhd\lambda,\alpha,\beta;  m+q a\Rhd\lambda,\alpha,\beta;  m\big),\\
P_{M,\alpha}(m)\Lhd\lambda,\alpha,\beta;  P_\beta(a)&=P_{M,\alpha\beta}\big(m \Lhd\lambda,\alpha,\beta;  P_\beta(a)+P_{M,\alpha}(m) \Lhd\lambda,\alpha,\beta;  a+q m\Lhd\lambda,\alpha,\beta;  a\big).
\end{align*}
We only prove the first equality, the second being similar.

In fact,
\begin{align*}
&P_\alpha(a)\Rhd\lambda,\alpha,\beta;  P_{M,\,\beta}(m)\\
={}&P_\alpha^2(a)\Cdot \lambda,\alpha,\beta; P_{M,\,\beta}(m)-P_{M,\alpha\beta}(P_\alpha(a)\Cdot \lambda,\alpha,\beta; P_{M,\,\beta}(m))\\
={}&  P_{M,\alpha\beta} (P_\alpha(a)\Cdot \lambda,\alpha,\beta; P_{M,\,\beta}(m) + P_\alpha^2(a)\Cdot \lambda,\alpha,\beta; m +q  P_\alpha(a)\Cdot \lambda,\alpha,\beta; m )-P_{M,\alpha\beta} \big(P_\alpha(a)\Cdot \lambda,\alpha,\beta; P_{M,\,\beta}(m)\big)  \\
={}& P_{M,\alpha\beta}\big(P_\alpha^2(a)\Cdot \lambda,\alpha,\beta; m +q  P_\alpha(a)\Cdot \lambda,\alpha,\beta; m \big).
\end{align*}
Also, we have
\begin{align*}
 & P_{M,\alpha\beta}\left(a\Rhd\lambda,\alpha,\beta;  P_{M,\,\beta}(m)+P_\alpha(a)\Rhd\lambda,\alpha,\beta;  m+q a\Rhd\lambda,\alpha,\beta;  m\right),\\
={}&P_{M,\alpha\beta}\bigg( P_\alpha(a)\Cdot \lambda,\alpha,\beta; P_{M,\,\beta}(m)-P_{M,\alpha\beta}(a\Cdot \lambda,\alpha,\beta; P_{M,\,\beta}(m))+ P_\alpha^2(a)\Cdot \lambda,\alpha,\beta; m-P_{M,\alpha\beta}(P_\alpha(a)\Cdot \lambda,\alpha,\beta; m)\\
&+q  P_\alpha(a)\Cdot \lambda,\alpha,\beta; m-q P_{M,\alpha\beta}(a\Cdot \lambda,\alpha,\beta; m)\bigg)\\
={}&P_{M,\alpha\beta}\Big(P_\alpha(a)\Cdot \lambda,\alpha,\beta; P_{M,\,\beta}(m)-P_{M,\alpha\beta}(a\Cdot \lambda,\alpha,\beta; P_{M,\,\beta}(m)
+P_\alpha(a)\Cdot \lambda,\alpha,\beta; m+q a\Cdot \lambda,\alpha,\beta; m)\\
&+P^2_\alpha(a)\Cdot \lambda,\alpha,\beta; m+q P_\alpha(a)\Cdot \lambda,\alpha,\beta; m\Big)\\
={}& P_{M,\alpha\beta}\big(   P_\alpha^2(a)\Cdot \lambda,\alpha,\beta; m +q  P_\alpha(a)\Cdot \lambda,\alpha,\beta; m \big)\\
={}& P_\alpha(a)\Rhd\lambda,\alpha,\beta;  P_{M,\,\beta}(m).\tag*{\qedhere}
\end{align*}
\end{proof}

\section{Cohomology theory of Rota-Baxter family $\Omega$-associative conformal algebras} \label{Sect: Cohomology theory of Rota-Baxter algebras}
In this  section, we will define a cohomology theory for Rota-Baxter family $\Omega$-associative conformal algebras of any weight.
\subsection{Hochschild cohomology of $\Omega$-associative conformal algebras}\
For the rest of the present subsection, we assume that $\Omega$ is a semigroup with unit $1 \in \Omega$. The unital condition of $\Omega$ is only useful in the coboundary operator of the cohomology at the degree $0$ level.

Let $(A , (\Mu\lambda,\alpha,\beta;)_{\alpha,\, \beta \in \Omega})$ be an $\Omega$-associative conformal algebra and $M$ be a bimodule over it. Now we describe the Hochschild cohomology complex $C^\bullet( A ,M)$ for an $\Omega$-associative
conformal algebra $ A $ with coefficients in a conformal $ A $-bimodule $M$ by means of $\lambda $-products~\cite{D,KK}. For each $n \geq 0$, we define an abelian group $C^n (A,M)$ consists of all multilinear maps of the form
\begin{align*}
f_{\lambda_1,\dots,\lambda_{n-1};\alpha_1,\dots,\alpha_n}:\, A ^{\otimes n}&\longrightarrow M\\
a_1\otimes\dots\otimes a_n&\longmapsto f_{\lambda_1,\dots,\lambda_{n-1};\alpha_1,\dots,\alpha_n}(a_1,\dots,a_n)
\end{align*}
satisfying the following sesquilinearity conditions:
\begin{align*}
f_{\lambda_1,\dots,\lambda_{n-1};\alpha_1,\dots,\alpha_n}(a_1,\dots,\partial a_i,\dots,a_n)&=-\la_if_{\lambda_1,\dots,\lambda_{n-1};\alpha_1,\dots,\alpha_n}(a_1,\dots,a_n), ~~i=1,\dots,n-1,\\
f_{\lambda_1,\dots,\lambda_{n-1};\alpha_1,\dots,\alpha_n}(a_1,\dots,\partial a_n)&=(\partial+\la_1+\dots+\la_{n-1})f_{\lambda_1,\dots,
\lambda_{n-1};\alpha_1,\dots,\alpha_n}(a_1,\dots,a_n).
\end{align*}
The $\Omega$-associative conformal Hochschild differential $\delta^n:C^n( A , M)\rightarrow C^{n+1}( A , M)$ is defined by
\begin{equation*}
  (\delta^0 m)_\alpha(a):=a\cdot_{0;\,\alpha,\,1}
  -m\cdot_{-\partial;\,1,\,\alpha}a,\,\text{ for }\, n=0\,\text{ and }\,m\in M=C^0( A , M).
\end{equation*}
For $n\geq 1$, define
\begin{align}
&(\delta^n f )_{\lambda_1,\dots,\lambda_{n};\,\alpha_1,\dots,\alpha_{n+1}}(a_1,\dots,a_{n+1})\nonumber\\
={}&a_{1}
\cdot_{{\lambda_1;\,\alpha_1,\,\alpha_2\dots\alpha_{n+1}}} f _{\lambda_2,\dots,\lambda_n;\,\alpha_2,\dots,\alpha_{n+1}}(a_2,\dots,a_{n+1})
\nonumber\\
&+\sum_{i=1}^{n}(-1)^i f _{\lambda_1,\dots,\lambda_i+\lambda_{i+1},\dots,
\lambda_{n};\,\alpha_1,\dots,\alpha_i\alpha_{i+1},\dots,\alpha_{n+1}}
(a_1,\dots,a_i\,\cdot_{\lambda_i;\,\,\alpha_i,\,\,
\alpha_{i+1}}\,a_{i+1},\dots,a_{n+1})\nonumber\\
&+(-1)^{n+1}f_{\lambda_1,\dots,\lambda_{n-1};\,\alpha_1,\dots,\alpha_n}
(a_1,\dots,a_n)
\cdot_{\la_1+\dots+\la_{n};\,\alpha_1\dots\alpha_n,\,\,\alpha_{n+1}}a_{n+1}.\label{co3}
\end{align}
for $f\in C^n(A,M).$ One can show that $\delta^n\circ\delta^{n-1}=0$~\cite{Das2,Lamei}.
So $(C^n(A,M),\delta^n)$ is a cochain complex and the corresponding cohomology groups are called the cohomology of the $\Omega$-associative conformal algebra $A$ with coefficient in the bimodule $M$.

\begin{defn}
An $n$-cochain $ f_{\lambda_1,\dots,\lambda_{n-1};\,\alpha_1,\dots,\alpha_n} \in C^n( A , M)$ is called an {\bf $n$-cocycle} if
\[(\delta^n f)_{\lambda_1,\dots,\lambda_n;\,\alpha_1,\dots,\alpha_{n+1}} =0\]
 and an element of the form $(\delta^{n-1} f)_{\lambda_1,\dots,\lambda_{n-1};\,\alpha_1,\dots,\alpha_n} $, where $ f_{\lambda_1,\dots,\lambda_{n-2};\,\alpha_1,\dots,\alpha_{n-1}} \in C^{n-1}( A , M)$, is called an {\bf $n$-coboundary}.
Denote by $Z^n( A , M)$ and $B^n( A , M)$ the subspaces of $n$-cocycles and $n$-coboundaries,
respectively. Then the quotient space $$H^n( A , M) = Z^n( A , M)/B^n
( A , M)$$ is called the {\bf $n$th Hochschild cohomology group} of $ A $ with coefficients in $M$.
\end{defn}

For example, the space of $1$-cocycle $Z^1( A , M) = \Ker \,{\delta^1} \subseteq  C^1( A , M)$ consists of all
$\mathbb{F}[\partial]$-linear maps $ (f_{\alpha})_{\alpha\in\Omega}:  A  \rightarrow M$ such that
\begin{align}\label{1-cocycle}
0=(\delta^1 f )_{\lambda;\,\alpha,\,\beta}(a,b)=a \cdot_{\lambda;\,\alpha,\,\beta} f_\beta  (b)- f_{\alpha\beta} (a\cdot_{\lambda;\,\alpha,\,\beta} b)+ f_\alpha (a)\cdot_{\lambda;\,\alpha,\,\beta}b,
\end{align}
and the space of $2$-cocycles $Z^2( A , M) = \Ker \,{\delta^2} \subseteq  C^2( A , M)$ consists of all conformal
sesquilinear maps $ (f_{\lambda;\,\alpha,\,\beta})_{\alpha,\,\beta\in\Omega}:  A \otimes A  \rightarrow M[\la]$ such that
\begin{align}\label{2-cocycle}
0&=(\delta^2 f )_{\lambda,\,\mu;\,\alpha,\,\beta,\,\gamma}(a,b,c)\nonumber\\
&=a\cdot_{\lambda;\,\alpha,\,\beta\gamma}  f _{\mu;\,\beta,\,\gamma} (b,c)- f _{\lambda+\mu;\,\,\alpha\beta,\,\gamma}(a\cdot_{\lambda;\,\alpha,\,\beta} b,c)+ f _{\lambda;\,\alpha,\,\beta\gamma}(a, b\cdot_{\mu;\,\beta,\,\gamma} c)- f _{\lambda;\,\alpha,\,\beta}(a,b)\cdot_{\lambda+\mu;\,\alpha\beta,\,\gamma}c.
\end{align}

\begin{remark}
From Eq.~(\ref{1-cocycle}), we have
\[ f_{\alpha\beta} (a\cdot_{\lambda;\,\alpha,\,\beta} b)=f_\alpha (a)\cdot_{\lambda;\,\alpha,\,\beta}b +a \cdot_{\lambda;\,\alpha,\,\beta} f_\beta  (b).\]
By Eq.~~(\ref{eq:differential family conformal algebras}), which shows that $(f_\alpha)_{\alpha\in\Omega}$ is a family derivation of weight $0$. That is to say the derivations of weight $0$ on an $\Omega$-associative conformal algebra are exactly in $C^1(A,M).$
\end{remark}

\subsection{Cohomology of Rota-Baxter family operators}\
\label{Subsect: cohomology RB operator}
Recall that Proposition~\ref{Prop: new RBF algebra} and Proposition~\ref{Prop:new-bimodule}, we  give a new
$\Omega$-associative conformal algebra  $A_{\star} $ and
  a new   Rota-Baxter family $\Omega$-associative conformal bimodule  $M$  by the left action $(\lhd_{\alpha,\,\beta})_{\alpha,\,\beta\in\Omega}$ and the right action $(\rhd_{\alpha,\,\beta})_{\alpha,\,\beta\in\Omega}$ over $A_{\star} $.
 Consider the Hochschild cochain complex of $A_{\star} $ with
 coefficients in $ M$:
 $$C^\bullet_{\Alg}(A_{\star} , M)=\bigoplus\limits_{n=0}^\infty C^n_{\Alg}(A_{\star} , M).$$

More precisely, for $n\geqslant 0$, denote by \[C^n_{\Alg}(A_{\star}, {M}):=\Hom  (A^{\ot n},M)\,\text{ and }\, C^0_{\Alg} (A_{\star} , M) = M,\]
and for $\alpha\in\Omega$, define
\begin{align*}
\partial^0(m)_\alpha(a):=P_\alpha(a)\cdot_{0;\,\alpha,\,1}m
-m\cdot_{-\partial;\,1,\,\alpha}P_\alpha(a)
-P_{M,\,\alpha}(a\cdot_{0;\,\alpha,\,1}m-m\cdot_{-\partial;\,1,\,\alpha}a).
\end{align*}
For $\alpha_1,\,\dots,\alpha_n\in\Omega$, denote by
\begin{align*}
&C^{n \geq 1}_{\Alg} (A_{\star} , M) = \big\{  f = \{ f_{\lambda_1,\dots,\lambda_{n-1};\alpha_1, \dots, \alpha_n} \}_{\alpha_1, \dots, \alpha_n \in \Omega} ~|~ f_{\lambda_1,\dots,\lambda_{n-1};\alpha_1, \dots, \alpha_n} : A^{\otimes n} \rightarrow M \text{ is multilinear} \big\},
\end{align*}
 and its differential
 \[\partial^n : C^n_{\Alg} (A_{\star} , M) \rightarrow C^{n+1}_{\Alg}(A_{\star} , M)\]
 is defined by
\begin{align*}
&\big( \partial^n (f) \big)_{\lambda_1,\dots,\lambda_{n};\,\alpha_1, \dots, \alpha_{n+1}} (u_1, \dots, u_{n+1}) \nonumber\\
={}& P_{\alpha_1} (u_1) \Cdot\lambda_1,\alpha_1,\alpha_2\dots\alpha_{n+1}; f_{\lambda_2,\dots,\lambda_{n};\,\alpha_2, \dots, \alpha_{n+1}} (u_2, \dots, u_{n+1})\\
&- P_{M,\alpha_1 \dots \alpha_{n+1}} \big( u_1 \Cdot\lambda_1,\alpha_1,\alpha_2\dots\alpha_{n+1}; f_{\lambda_2,\dots,\lambda_{n};\,\alpha_2, \dots, \alpha_{n+1}} (u_2, \dots, u_{n+1})\nonumber\\
&+ \sum_{i=1}^n (-1)^{i} f_{\lambda_1,\dots,\lambda_i+\lambda_{i+1},\dots\lambda_{n};\,\alpha_1, \dots, \alpha_i \alpha_{i+1}, \dots, \alpha_{n+1}} \Big( u_1, \dots, \Big(P_{\alpha_i} (u_i) \Cdot\lambda_i,\alpha_i,\alpha_{i+1}; u_{i+1} + u_i \Cdot\lambda_i,\alpha_i,\alpha_{i+1}; P_{\alpha_{i+1}} (u_{i+1})\\
&+q u_i\Cdot\lambda_i,\alpha_i,\alpha_{i+1}; u_{i+1}\Big),\dots,u_{n+1}\Big) + (-1)^{n+1}\Big(f_{\lambda_1,\dots,\lambda_{n-1};\,\alpha_1, \dots, \alpha_n} (u_1, \dots, u_n) \Cdot\lambda_1+\dots+\lambda_n,\alpha_1\dots\alpha_n,\alpha_{n+1}; P_{\alpha_{n+1}} (u_{n+1}) \Big)\\
&-(-1)^{n+1} P_{M,\alpha_1 \dots \alpha_{n+1}} \big(  f_{\lambda_1,\dots,\lambda_{n-1};\,\alpha_1, \dots, \alpha_n} (u_1, \dots, u_n) \Cdot\lambda_1+\dots+\lambda_n,\alpha_1\dots\alpha_n,\alpha_{n+1};  u_{n+1} \big)\Big).\nonumber
\end{align*}

\begin{defn}
 	Let $A=(A, (\Mu\lambda,\alpha,\beta;)_{\alpha,\,\beta\in\Omega}
 ,(P_\omega)_{\omega\in\Omega})$ be a Rota-Baxter family $\Omega$-associative conformal algebra of weight $q$ and $M=(M,(P_{M,\,\omega})_{\omega\in\Omega})$ be a Rota-Baxter family $\Omega$-associative conformal bimodule over it. Then the cochain complex $(C^\bullet_\Alg(A_{\star}, M),\partial^n)$ is called {\bf the cochain complex of Rota-Baxter family operator} $(P_\omega)_{\omega\in\Omega}$ with coefficients in $(M, (P_{M,\,\omega})_{\omega\in\Omega})$,  denoted by $C_{\RBO}^\bullet(A, M)$. The cohomology of $C_{\RBO}^\bullet(A, $ $M)$, denoted by $\mathrm{H}_{\RBO}^\bullet(A,M)$, are called {\bf the cohomology of Rota-Baxter family  operator} $(P_\omega)_{\omega\in\Omega}$ with coefficients in $(M, (P_{M,\,\omega})_{\omega\in\Omega})$.
 \end{defn}

\begin{remark}
When $M=A$, we call $(M, (P_{M,\,\omega})_{\omega\in\Omega})$ the regular Rota-Baxter family $\Omega$-associative conformal bimodule $(A,(P_\omega)_{\omega\in\Omega})$.
\end{remark}

\subsection{Cohomology of Rota-Baxter family $\Omega$-associative conformal algebras}\
\label{Subsec:chomology RB}
In this subsection, we will combine the Hochschild cohomology of $\Omega$-associative conformal algebras and the cohomology of Rota-Baxter family operators to define a cohomology theory for Rota-Baxter family $\Omega$-associative conformal algebras.

Let $M=(M,(P_{M,\,\omega})_{\omega\in\Omega})$ be a Rota-Baxter family $\Omega$-associative conformal bimodule of weight $q$. Now, let's construct a chain map   $$\Phi^\bullet: C ^\bullet_{\Alg}(A,M) \rightarrow C_{\RBO}^\bullet(A,M),$$ i.e., the following commutative diagram:
\[\xymatrix{
		 C ^0_{\Alg}(A,M)\ar[r]^-{\delta^0}\ar[d]^-{\Phi^0}& C ^1_{\Alg}(A,M)\ar@{.}[r]\ar[d]^-{\Phi^1}& C ^n_{\Alg}(A,M)\ar[r]^-{\delta^n}\ar[d]^-{\Phi^n}& C ^{n+1}_{\Alg}(A,M)\ar[d]^{\Phi^{n+1}}\ar@{.}[r]&\\
		 C ^0_{\RBO}(A,M)\ar[r]^-{\partial^0}& C ^1_{\RBO}(A,M)\ar@{.}[r]& C ^n_{\RBO}(A,M)\ar[r]^-{\partial^n}& C ^{n+1}_{\RBO}(A,M)\ar@{.}[r]&
.}\]

Define $\Phi^0=\Id_{\Hom(k,M)}=\Id_M$. For $n=1$ and $f=(f_\alpha)_{\alpha\in\Omega}\in C^1_{\Alg}(A,M)$, define
\begin{equation*}
\Phi^1(f)_\alpha(a):=f_\alpha(P_\alpha(a))-P_{M,\,\alpha}(f_\alpha(a)),\,\text{ for }\,\alpha\in\Omega.
\end{equation*}
For  $n\geq 2$ and $f=(f_{\lambda_1,\dots,\lambda_{n-1};\alpha_1, \dots, \alpha_n})_{\alpha_1,\ldots,\alpha_n\in\Omega}\in  C^n_{\Alg}(A,M)$,
 define
\begin{align*}
 &\Phi^n(f)_{\lambda_1,\dots,\lambda_{n-1};\,\alpha_1, \dots, \alpha_n}(u_1,\dots,  u_n) \\
={}&f_{\lambda_1,\dots,\lambda_{n-1};\,\alpha_1, \dots, \alpha_n}(P_{\alpha_1}(u_1), \dots, P_{\alpha_n}(u_n))\\
&-\sum_{k=0}^{n-1}q^{n-k-1}\sum_{1\leqslant i_1<i_2<\dots<i_k\leqslant n}P_{M,\,{\alpha_1 \dots \alpha_n}}\circ f_{\lambda_1,\dots,\lambda_{n-1};\,\alpha_1, \dots, \alpha_n}\\
&\hspace{6cm}(u_{1, i_1-1},\, P_{\alpha_{i_1}}(u_{i_1}), u_{i_1+1, i_2-1}, P_{\alpha_{i_2}}(u_{i_2}), \dots,P_{\alpha_{i_k}}(u_{i_k}), u_{i_k+1, n}).
\end{align*}


\begin{prop}\label{Prop: Chain map Phi}
	The map $\Phi^\bullet: C^\bullet_\Alg(A,M)\rightarrow C^\bullet_{\RBO}(A,M)$ is a chain map.
\end{prop}

\begin{proof}
We leave the long proof of this result to Appendix~\ref{Appendix}.
\end{proof}

\begin{defn}
 Let $M=(M,(P_{M,\,\omega})_{\omega\in\Omega})$ be a  Rota-Baxter family $\Omega$-associative conformal bimodule of weight $q$.  We define the cochain complex $(C^\bullet_{\RBA}(A,M), d^\bullet)$  of Rota-Baxter family $\Omega$-associative conformal algebra $(A, (\mu_{\lambda;\,\alpha,\,\beta})_{\alpha,\,\beta\in\Omega},(P_{\omega})_{\omega\in\Omega})$ with coefficients in $(M,(P_{M,\,\omega})_{\omega\in\Omega})$, that is,   let
\[ C ^0_{\RBA}(A,M)= C ^0_\Alg(A,M)  \quad  \mathrm{and}\quad    C ^n_{\RBA}(A,M)= C ^n_\Alg(A,M)\oplus C ^{n-1}_{\RBO}(A,M),\,\text{ for }\, n\geqslant 1,\]
 and the differential $d^n:  C ^n_{\RBA}(A,M)\rightarrow  C ^{n+1}_{\RBA}(A,M)$ is given by
\begin{align}
&d^n(f, g)_{\lambda_1,\dots,\lambda_n,\,\mu_1,\dots,\,\mu_m;\,\alpha_1,\dots,\alpha_{n+1},\,\beta_1,\ldots,\,\beta_n}
(a_1,\dots,a_{n+1})\nonumber\\
={}&\Big(\delta^n(f)_{\lambda_1,\dots,\lambda_{n};\,\alpha_1, \ldots, \alpha_{n+1}},\,
 -\partial^{n-1}(g)_{\mu_1,\dots,\,\mu_{m};\,\beta_1, \ldots,\, \beta_{n}} -\Phi^n(f)_{\mu_1,\dots,\
 \mu_{m};\,\beta_1,\, \ldots, \,\beta_n}\Big)\label{eq:diff}
 \end{align}
 for any $f\in  C ^n_\Alg(A,M)$ and $g\in C ^{n-1}_{\RBO}(A,M)$.

The  cohomology of $( C ^\bullet_{\RBA}(A,M), d^\bullet)$, denoted by $\rmH_{\RBA}^\bullet(A,M)$,  is called {\bf the cohomology of the Rota-Baxter family $\Omega$-associative conformal algebra} $(A, (\mu_{\lambda;\,\alpha,\,\beta})_{\alpha,\,\beta\in\Omega},(P_\omega)_{\omega\in\Omega})$ with coefficients in $(M,(P_{M,\,\omega})_{\omega\in\Omega})$.
\end{defn}

\begin{remark}
When $M=A$, we just denote $ C ^\bullet_{\RBA}(A,A)$, $\rmH^\bullet_{\RBA}(A,A)$   by $ C ^\bullet_{\RBA}(A),  $  $\rmH_{\RBA}^\bullet(A)$ respectively, and call  them the cochain complex, the cohomology of Rota-Baxter family $\Omega$-associative conformal algebra  $(A, (\mu_{\lambda;\,\alpha,\,\beta})_{\alpha,\,\beta\in\Omega},(P_{\omega})_{\omega\in\Omega})$ respectively.
\end{remark}

\section{Formal deformations of Rota-Baxter family on $\Omega$-associative conformal algebras and cohomological interpretation}
\label{sec:formal deformations}

In this section, we will study the formal deformations of Rota-Baxter family $\Omega$-associative conformal algebras and interpret them  via lower degree cohomology groups of Rota-Baxter family $\Omega$-associative conformal algebras defined in last section.

\subsection{Formal deformations of $\Omega$-associative conformal algebras}\
Let $ A $ be an $\Omega$-associative conformal algebra, and $(\omega_{\lambda;\,\alpha,\,\beta})_{\alpha,\,\beta\in\Omega}:  A  \times  A  \rightarrow A [\lambda]$ is a family of conformal bilinear maps. For all $a,b\in A$ and $\alpha,\beta\in\Omega$, we consider a family of $t$-parameterized bilinear $\lambda$-multiplications
\begin{equation*}
a \circ_{\lambda;\,\alpha,\,\beta}^{t} b=a\cdot_{\lambda;\,\alpha,\,\beta} b+t \omega_{\lambda;\,\alpha,\,\beta}(a, b).
\end{equation*}

\begin{defn}\label{defn:deformation}
If all the $\lambda$-multiplications $(\circ_{\lambda;\,\alpha,\,\beta}^{t})_{\alpha,\,\beta\in\Omega}$ endow $A$ with $\Omega$-associative conformal algebra structures, then we call that $(\omega_{\lambda;\,\alpha,\,\beta})_{\alpha,\,\beta\in\Omega}$ is {\bf a family deformation} of the $\Omega$-associative conformal algebra $A$.
\end{defn}

If $(A, (\circ_{\lambda;\,\alpha,\,\beta}^t)_{\alpha,\,\beta\in\Omega})$ is an $\Omega$-associative conformal algebra, i.e.,
\[(a\circ_{\lambda;\,\alpha,\,\beta}^t b)\circ_{\lambda+\mu;\,\alpha\beta,\,\gamma}^t c
=a\circ_{\lambda;\,\alpha,\,\beta\gamma}^t(b\circ_{\mu;\,\beta,\,\gamma}^tc).\]
Then we have
\begin{align*}
&(a\cdot_{\lambda;\,\alpha,\,\beta}b)\cdot_{\lambda;\,\alpha\beta,\,\gamma}c
+t\omega_{\lambda;\,\alpha,\,\beta}(a,b)\cdot_{\lambda+\mu;\,\alpha\beta,\,\gamma}c
+t\omega_{\lambda+\mu;\,\alpha\beta,\,\gamma}(a\cdot_{\lambda;\,\alpha,\,\beta}b,c)
+t^2\omega_{\lambda+\mu;\,\alpha\beta;\,\gamma}(\omega_{\lambda;\,\alpha,\,\beta}(a,b),c)\\
={}&a\cdot_{\lambda;\,\alpha,\,\beta\gamma}(b\cdot_{\mu;\,\beta,\,\gamma}c)
+ta\cdot_{\lambda;\,\alpha,\,\beta\gamma}(\omega_{\mu;\,\beta,\,\gamma}(b,c))
+t\omega_{\lambda;\,\alpha,\,\beta\gamma}(a,b\cdot_{\mu;\,\beta,\,\gamma}c)
+t^2\omega_{\lambda;\,\alpha,\,\beta\gamma}(a,\omega_{\mu;\,\beta,\,\gamma}(b,c)).
\end{align*}
Equivalently,
\begin{align}
\label{eq:condition1} \omega_{\lambda+\mu;\,\alpha\beta,\,\gamma}(a\cdot_{\lambda;\,\alpha,\,\beta}b,c)
+\omega_{\lambda;\,\alpha,\,\beta}(a,b)\cdot_{\lambda+\mu;\,\alpha\beta,\,\gamma}c
&=\omega_{\lambda;\,\alpha,\,\beta\gamma}(a,b\cdot_{\mu;\,\beta,\,\gamma}c)
+a\cdot_{\lambda;\,\alpha,\,\beta\gamma}(\omega_{\mu;\,\beta,\,\gamma}(b,c)),\\
\label{eq:condition2} \omega_{\lambda+\mu;\,\alpha\beta;\,\gamma}(\omega_{\lambda;\,\alpha,\,\beta}(a,b),c)
&=\omega_{\lambda;\,\alpha,\,\beta\gamma}(a,\omega_{\mu;\,\beta,\,\gamma}(b,c)).
\end{align}
Hence, by Eq.~(\ref{eq:condition2}), we know that $(\omega_{\lambda;\,\alpha,\,\beta})_{\alpha,\,\beta\in\Omega}$ is an $\Omega$-associative conformal algebra structure, moreover satisfying condition Eq.~(\ref{eq:condition1}).

Recall that the definition of the $\Omega$-associative conformal Hochschild differential defined by Eq.~(\ref{co3}) and Eq.~(\ref{2-cocycle}), we have $(\delta^2 \omega )_{\lambda,\,\mu;\,\alpha,\,\beta,\,\gamma}=0$. Namely, the $\Omega$-associative conformal bilinear map $(\omega_{\lambda;\,\alpha,\,\beta})_{\alpha,\,\beta\in\Omega}$ is a 2-cocyle in $C^{2}(A, A)$.

\begin{defn}
A family of deformation $(\omega_{\lambda;\,\alpha,\,\beta})_{\alpha,\,\beta\in\Omega}$ is said to be {\bf trivial} if there exists a family of $\mathbb{F}[\partial]$-linear maps $(N_\omega)_{\omega\in\Omega}: A \rightarrow A$ such that for $T_{t,\,\omega}=\mathrm{id}+t N_\omega$ there holds
$$
T_{t,\,\alpha\beta}\left(a \circ_{\lambda;\,\alpha,\,\beta}^{t} b\right)=T_{t,\,\alpha}(a)\cdot_{\lambda;\,\alpha,\,\beta} T_{t,\,\beta}(b), \text { for all } a, b \in A,\alpha,\beta\in\Omega.
$$
\end{defn}
Expanding the above equation of both sides, we have
$$
\begin{aligned}
T_{t,\,\alpha\beta}\big(a \circ_{\lambda;\,\alpha,\,\beta}^{t} b\big) &=a\cdot_{\lambda;\,\alpha,\,\beta} b+t\big(\omega_{\lambda;\,\alpha,\,\beta}(a, b)+N_{\alpha\beta}\big(a\cdot_{\lambda;\,\alpha,\,\beta} b\big)\big)+t^{2} N_{\alpha\beta} (\omega_{\lambda;\,\alpha,\,\beta}(a, b)), \\
T_{t,\,\alpha}(a)\cdot_{\lambda;\,\alpha,\,\beta} T_{t,\,\beta}(b) &=a\cdot_{\lambda;\,\alpha,\,\beta} b+t\big(N_\alpha(a)\cdot_{\lambda;\,\alpha,\,\beta} b+a\cdot_{\lambda;\,\alpha,\,\beta} N_\beta(b)\big)+t^{2} N_\alpha(a)\cdot_{\lambda;\,\alpha,\,\beta} N_\beta(b).
\end{aligned}
$$
The triviality of deformation is equivalent to the following equations:
\begin{align}
\label{eq:deformation1}\omega_{\lambda;\,\alpha,\,\beta}(a, b) &=N_\alpha(a)\cdot_{\lambda;\,\alpha,\,\beta} b+a\cdot_{\lambda;\,\alpha,\,\beta} N_\beta(b)-N_{\alpha\beta}\big(a\cdot_{\lambda;\,\alpha,\,\beta} b\big), \\
\label{eq:deformation2}N_{\alpha\beta} (\omega_{\lambda;\,\alpha,\,\beta}(a, b)) &=N_\alpha(a)\cdot_{\lambda;\,\alpha,\,\beta} N_\beta(b).
\end{align}

It follows from Eq.~(\ref{eq:deformation1}) and Eq.~(\ref{eq:deformation2}) that $(N_\omega)_{\omega\in\Omega}$ must satisfy the following condition:
\begin{equation}\label{eq:nijenhuis family algebras}
N_\alpha(a)\cdot_{\lambda;\,\alpha,\,\beta} N_\beta(b)=N_{\alpha\beta}\big(N_\alpha(a)\cdot_{\lambda;\,\alpha,\,\beta} b+a\cdot_{\lambda;\,\alpha,\,\beta} N_\beta(b)-N_{\alpha\beta}\big(a\cdot_{\lambda;\,\alpha,\,\beta} b\big)\big)
\end{equation}
Thus, Eq.~(\ref{eq:nijenhuis family algebras}) is a generalization of the following Nijenhuis equation~\cite{Lamei}
\[N(a)_{\lambda} N(b)=N\left(N(a)_{\lambda} b+a_{\lambda} N(b)-N\left(a_{\lambda} b\right)\right),\]
and also shows that any trivial deformation produces a Nijenhuis family operator.

\begin{prop}\label{prop:associativity}
 Let $(N_\omega)_{\omega\in\Omega}$ be a family of Nijenhuis operators over an $\Omega$-associative conformal algebra $ A $. Define
$$
a \circ_{\lambda;\,\alpha,\,\beta}^{N} b:=N_\alpha(a)\cdot_{\lambda;\,\alpha,\,\beta} b+a\cdot_{\lambda;\,\alpha,\,\beta} N_\beta(b)-N_{\alpha\beta} \big(a\cdot_{\lambda;\,\alpha,\,\beta} b \big),\,\text{ for }\, a, b \in A.
$$
Then we have
\begin{enumerate}
\item \label{item:associativity} The pair $ \big( A , (\circ_{\lambda;\,\alpha,\,\beta}^{N} )_{\alpha,\,\beta\in\Omega}\big)$ is a new $\Omega$-associative conformal algebra, denoted by $ A^{N}$.
 \item \label{item:homomorphism} $(N_\omega)_{\omega\in\Omega}$ is a family algebra homomorphism from $ A^{N}$ to the original $\Omega$-associative conformal algebra $ A $, namely
$$
N_{\alpha\beta} \big(a \circ_{\lambda;\,\alpha,\,\beta}^{N} b \big)=N_\alpha(a)\cdot_{\lambda;\,\alpha,\,\beta} N_\beta(b), \,\text{ for }\, a, b \in  A.
$$
\end{enumerate}
\end{prop}

\begin{proof}
(\ref{item:associativity}). For $a,b,c\in A^N$ and $\alpha,\beta,\gamma\in\Omega$, on the left hand side, we have
\begin{align*}
&(a\circ^N_{\lambda;\,\alpha,\,\beta}b)\circ_{\lambda+\mu;\,\alpha\beta,\,\gamma}^N c\\
={}&\left(N_\alpha(a)\cdot_{\lambda;\,\alpha,\,\beta}b+a\cdot_{\lambda;\,\alpha,\,\beta}N_\beta(b)
-N_{\alpha\beta}(a\cdot_{\lambda;\,\alpha,\,\beta} b)\right)\circ_{\lambda+\mu;\,\alpha\beta,\,\gamma}^N c\\
={}&N_{\alpha\beta}\left(N_\alpha(a)\cdot_{\lambda;\,\alpha,\,\beta}b+a\cdot_{\lambda;\,\alpha,\,\beta}N_\beta(b)
-N_{\alpha\beta}(a\cdot_{\lambda;\,\alpha,\,\beta} b)\right)\cdot_{\lambda+\mu;\alpha\beta,\,\gamma}c\\
&+\left(N_\alpha(a)\cdot_{\lambda;\,\alpha,\,\beta}b+a\cdot_{\lambda;\,\alpha,\,\beta}N_\beta(b)
-N_{\alpha\beta}(a\cdot_{\lambda;\,\alpha,\,\beta} b)\right)\cdot_{\lambda+\mu;\,\alpha\beta,\,\gamma}N_\gamma(c)\\
&-N_{\alpha\beta\gamma}\left(\left(N_\alpha(a)\cdot_{\lambda;\,\alpha,\,\beta}b
+a\cdot_{\lambda;\,\alpha,\,\beta}N_\beta(b)
-N_{\alpha\beta}(a\cdot_{\lambda;\,\alpha,\,\beta} b)\right)\cdot_{\lambda+\mu;\,\alpha\beta,\,\gamma}c\right)\\
={}&\left(N_\alpha(a)\cdot_{\lambda;\,\alpha,\,\beta}N_\beta(b)\right)\cdot_{\lambda+\mu;\alpha\beta,\,\gamma}c
+\left(N_\alpha(a)\cdot_{\lambda;\,\alpha,\,\beta} b\right)\cdot_{\lambda+\mu;\alpha\beta,\,\gamma}N_\gamma(c)
+\left(a\cdot_{\lambda;\,\alpha,\,\beta}N_\beta(b)\right)\cdot_{\lambda+\mu;\alpha\beta,\,\gamma}N_\gamma(c)\\
&\underline{-N_{\alpha\beta}(a\cdot_{\lambda;\,\alpha,\,\beta}b)
\cdot_{\lambda+\mu;\alpha\beta,\,\gamma}N_\gamma(c)}
-N_{\alpha\beta\gamma}\left(\left(N_\alpha(a)\cdot_{\lambda;\,\alpha,\,\beta}b\right)
\cdot_{\lambda+\mu;\alpha\beta,\,\gamma}c\right)
-N_{\alpha\beta\gamma}\left(\left(a\cdot_{\lambda;\,\alpha,\,\beta}N_\beta(b)\right)
\cdot_{\lambda+\mu;\alpha\beta,\,\gamma}c\right)\\
&+N_{\alpha\beta\gamma}\left(N_{\alpha\beta}(a\cdot_{\lambda;\,\alpha,\,\beta}b)
\cdot_{\lambda+\mu;\alpha\beta,\,\gamma}\right)\quad{\text{ (by Eq.~(\ref{eq:nijenhuis family algebras}) of the underline item})}\\
={}&\left(N_\alpha(a)\cdot_{\lambda;\,\alpha,\,\beta}N_\beta(b)\right)\cdot_{\lambda+\mu;\alpha\beta,\,\gamma}c
+\left(N_\alpha(a)\cdot_{\lambda;\,\alpha,\,\beta} b\right)\cdot_{\lambda+\mu;\alpha\beta,\,\gamma}N_\gamma(c)
+\left(a\cdot_{\lambda;\,\alpha,\,\beta}N_\beta(b)\right)\cdot_{\lambda+\mu;\alpha\beta,\,\gamma}N_\gamma(c)\\
&-\underline{N_{\alpha\beta\gamma}\left(N_{\alpha\beta}\left(a\cdot_{\lambda;\,\alpha,\,\beta}b\right)
\cdot_{\lambda+\mu;\alpha\beta,\,\gamma}c\right)}
-N_{\alpha\beta\gamma}\left(\left(a\cdot_{\lambda;\,\alpha,\,\beta}b\right)
\cdot_{\lambda+\mu;\alpha\beta,\,\gamma}N_\gamma(c)\right)
+N^2_{\alpha\beta\gamma}\left(\left(a\cdot_{\lambda;\,\alpha,\,\beta}b\right)
\cdot_{\lambda+\mu;\alpha\beta,\,\gamma}c\right)\\
&-N_{\alpha\beta\gamma}\left(\left(N_\alpha(a)\cdot_{\lambda;\,\alpha,\,\beta}b\right)
\cdot_{\lambda+\mu;\alpha\beta,\,\gamma}c\right)
-N_{\alpha\beta\gamma}\left(\left(a\cdot_{\lambda;\,\alpha,\,\beta}N_\beta(b)\right)
\cdot_{\lambda+\mu;\alpha\beta,\,\gamma}c\right)
+\underline{N_{\alpha\beta\gamma}\left(N_{\alpha\beta}(a\cdot_{\lambda;\,\alpha,\,\beta}b)
\cdot_{\lambda+\mu;\alpha\beta,\,\gamma}\right)}.
\end{align*}

On the right hand side, we have
\begin{align*}
&a\circ_{\lambda;\,\alpha,\,\beta\gamma}^N(b\circ_{\mu;\,\beta,\,\gamma}^Nc)\\
={}&a\circ_{\lambda;\,\alpha,\,\beta\gamma}^N\left(N_\beta(b)\cdot_{\mu;\,\beta,\,\gamma}c
+b\cdot_{\mu;\,\beta,\,\gamma} N_\gamma(c)-N_{\beta\gamma}(b\cdot_{\mu;\,\beta,\,\gamma}c)\right)\\
={}&N_\alpha(a)\cdot_{\lambda;\,\alpha,\,\beta\gamma}\left(N_\beta(b)\cdot_{\mu;\,\beta,\,\gamma}c
+b\cdot_{\mu;\,\beta,\,\gamma} N_\gamma(c)-N_{\beta\gamma}(b\cdot_{\mu;\,\beta,\,\gamma}c)\right)\\
&+a\cdot_{\lambda;\,\alpha,\,\beta\gamma}N_{\beta\gamma}\left(N_\beta(b)\cdot_{\mu;\,\beta,\,\gamma}c
+b\cdot_{\mu;\,\beta,\,\gamma} N_\gamma(c)-N_{\beta\gamma}(b\cdot_{\mu;\,\beta,\,\gamma}c)\right)\\
&-N_{\alpha\beta\gamma}\left(a\cdot_{\lambda;\,\alpha,\,\beta\gamma}\left(N_\beta(b)\cdot_{\mu;\,\beta,\,\gamma}c
+b\cdot_{\mu;\,\beta,\,\gamma} N_\gamma(c)-N_{\beta\gamma}(b\cdot_{\mu;\,\beta,\,\gamma}c)\right)\right)\\
={}&N_\alpha(a)\cdot_{\lambda;\,\alpha,\,\beta\gamma}\left(N_\beta(b)\cdot_{\mu;\,\beta,\,\gamma}c\right)
+N_\alpha(a)\cdot_{\lambda;\,\alpha,\,\beta\gamma}\left(b\cdot_{\mu;\,\beta,\,\gamma}N_\gamma(c)\right)
\underline{-N_\alpha(a)\cdot_{\lambda;\,\alpha,\,\beta\gamma}N_{\beta\gamma}(b\cdot_{\mu;\,\beta,\,\gamma}c)}\\
&+a\cdot_{\lambda;\,\alpha,\,\beta\gamma}N_{\beta\gamma}(N_\beta(b)\cdot_{\mu;\,\beta,\,\gamma}c)
+a\cdot_{\lambda;\,\alpha,\,\beta\gamma}N_{\beta\gamma}\left(b\cdot_{\mu;\,\beta,\,\gamma}N_\gamma(c)\right)
-a\cdot_{\lambda;\,\alpha,\,\beta\gamma}N_{\beta\gamma}^2(b\cdot_{\mu;\,\beta,\,\gamma}c)\\
&-N_{\alpha\beta\gamma}\left(a\cdot_{\lambda;\,\alpha,\,\beta\gamma}
\left(N_\beta(b)\cdot_{\mu;\,\beta,\,\gamma}c\right)\right)-N_{\alpha\beta\gamma}
\left(a\cdot_{\lambda;\,\alpha,\,\beta\gamma}\left(b\cdot_{\mu;\,\beta,\,\gamma}N_\gamma(c)\right)\right)
+N_{\alpha\beta\gamma}
\left(a\cdot_{\lambda;\,\alpha,\,\beta\gamma}N_{\beta\gamma}(b\cdot_{\mu;\,\beta,\,\gamma}c)\right)\\
&\hspace{5cm}{\text{ (by Eq.~(\ref{eq:nijenhuis family algebras}) of the underline item})}\\
={}&N_\alpha(a)\cdot_{\lambda;\,\alpha,\,\beta\gamma}\left(N_\beta(b)\cdot_{\mu;\,\beta,\,\gamma}c\right)
+N_\alpha(a)\cdot_{\lambda;\,\alpha,\,\beta\gamma}\left(b\cdot_{\mu;\,\beta,\,\gamma}N_\gamma(c)\right)
+a\cdot_{\lambda;\,\alpha,\,\beta\gamma}\left(N_\beta(b)\cdot_{\mu;\,\beta,\,\gamma}N_\gamma(c)\right)\\
&-N_{\alpha\beta\gamma}\left(N_\alpha(a)\cdot_{\lambda;\,\alpha,\,\beta\gamma}
\left(b\cdot_{\mu;\,\beta,\,\gamma}c\right)\right)
-\underline{N_{\alpha\beta\gamma}\left(a\cdot_{\lambda;\,\alpha,\,\beta\gamma}N_{\beta\gamma}
\left(b\cdot_{\mu;\,\beta,\,\gamma}c\right)\right)}
+N^2_{\alpha\beta\gamma}\left(a\cdot_{\lambda;\,\alpha,\,\beta\gamma}
\left(b\cdot_{\mu;\,\beta,\,\gamma}c\right)\right)\\
&-N_{\alpha\beta\gamma}\left(a\cdot_{\lambda;\,\alpha,\,\beta\gamma}
\left(N_\beta(b)\cdot_{\mu;\,\beta,\,\gamma}c\right)\right)-N_{\alpha\beta\gamma}
\left(a\cdot_{\lambda;\,\alpha,\,\beta\gamma}\left(b\cdot_{\mu;\,\beta,\,\gamma}N_\gamma(c)\right)\right)
+\underline{N_{\alpha\beta\gamma}
\left(a\cdot_{\lambda;\,\alpha,\,\beta\gamma}N_{\beta\gamma}(b\cdot_{\mu;\,\beta,\,\gamma}c)\right)}.
\end{align*}
Now we see, deleting the same items labelled by underline and using the associativity of $\Omega$-associative conformal algebras, then the $i$-th term in the expansion of the left hand side equals to the $\sigma(i)$-th term in the
 expansion of the right hand side, where $\sigma$ is the following permutation of order $7$:
 \begin{equation*}
\begin{pmatrix}
     i \\
     \sigma(i)
\end{pmatrix}
=
\begin{pmatrix}
1 & 2 & 3 & 4 & 5 & 6 & 7 \\
1 & 2 & 3 & 7 & 5 & 4 & 6
\end{pmatrix}.
\end{equation*}
Thus
\[(a\circ^N_{\lambda;\,\alpha,\,\beta}b)\circ_{\lambda+\mu;\,\alpha\beta,\,\gamma}^N c
=a\circ_{\lambda;\,\alpha,\,\beta\gamma}^N(b\circ_{\mu;\,\beta,\,\gamma}^Nc).\]

(\ref{item:homomorphism}). It is directly from Eq.~(\ref{eq:nijenhuis family algebras}).
This completes the proof.
\end{proof}

\begin{theorem}\label{thm:trival deformation}
 Let $(N_\omega)_{\omega\in\Omega}:  A  \rightarrow  A $ be a family of Nijenhuis operators. For all $a, b \in  A $, define
$$
\omega_{\lambda;\,\alpha,\,\beta}(a, b)=N_\alpha(a)\cdot_{\lambda;\,\alpha,\,\beta} b+a\cdot_{\lambda;\,\alpha,\,\beta} N_\beta(b)-N_{\alpha\beta}(a\cdot_{\lambda;\,\alpha,\,\beta} b).
$$
Then $(\omega_{\lambda;\,\alpha,\,\beta})_{\alpha,\,\beta\in\Omega}$ is a deformation of $ A $ and this deformation is a trivial one.
\end{theorem}

\begin{proof}
By Proposition~\ref{prop:associativity}, we know that $(\omega_{\lambda;\,\alpha,\,\beta})_{\alpha,\,\beta\in\Omega}$ is $\Omega$-associative, that is to say Eq.~(\ref{eq:condition2}) holds. From the coboundary operator in the Hochschild cohomology complex of $A$ and Eqs.~(\ref{1-cocycle}-\ref{2-cocycle}), replacing $f$ with $\omega$, we have $(\delta^2(\omega))_{\lambda,\,\mu;\,\alpha,\,\beta,\,\gamma}=0,$ this is exactly Eq.~(\ref{eq:condition1}). Equivalently, Eqs.~(\ref{eq:deformation1}-\ref{eq:deformation2}) are satisfied and therefore $(\omega_{\lambda;\,\alpha,\,\beta})_{\alpha,\,\beta\in\Omega}$ is a trivial deformation of $\Omega$-associative conformal algebras. This completes the proof.
\end{proof}

\begin{remark}
By Theorem~\ref{thm:trival deformation}, for $a,b\in A$ and $\alpha,\beta\in\Omega$, the family
\begin{equation}\label{eq:compatible product}
(\omega^N_{\lambda;\,\alpha,\,\beta})_{\alpha,\,\beta\in\Omega}:(a,b)\mapsto a \circ_{\lambda;\,\alpha,\,\beta}^{N} b=N_\alpha(a)\cdot_{\lambda;\,\alpha,\,\beta} b+a\cdot_{\lambda;\,\alpha,\,\beta} N_\beta(b)-N_{\alpha\beta} \big(a\cdot_{\lambda;\,\alpha,\,\beta} b \big),
\end{equation}
defines an $\Omega$-associative conformal algebra and it also defines a new algebra structure on $A$.
From Proposition~\ref{prop:associativity}(\ref{item:homomorphism}), we have a $2$-cochain $(\psi_{\lambda;\,\alpha,\,\beta}^N)_{\alpha,\,\beta\in\Omega}\in C^2(A,A)$ as follows
\[\psi_{\lambda;\,\alpha,\,\beta}^N(a,b)=N_\alpha(a)\cdot_{\lambda;\,\alpha,\,\beta} N_\beta(b)-N_{\alpha\beta} \big(a \circ_{\lambda;\,\alpha,\,\beta}^{N} b \big),\,\text{ for } a,b\in A,\alpha,\beta\in\Omega.\]

It is obvious $\psi_{\lambda;\,\alpha,\,\beta}^N(a, b)=0$ if and only if $(N_\omega)_{\omega\in\Omega}$ is a family of Nijenhuis operators.
\end{remark}

Now we arrive at our main result in this subsection.
\begin{theorem}\label{thm:main result}
Let $A$ be an $\Omega$-associative conformal algebra and $\Omega$ a semigroup.
\begin{enumerate}
 \item \label{item:compatible} Denoted  by $\omega_{\lambda;\,\alpha,\,\beta}(a,b):=a\cdot_{\lambda;\,\alpha,\,\beta}b,$ then $(\omega^N_{\lambda;\,\alpha,\,\beta})_{\alpha,\,\beta\in\Omega}$ defined by Eq.~(\ref{eq:compatible product}) is an $\Omega$-associative $\lambda$-product compatible with
        $(\omega_{\lambda;\,\alpha,\,\beta})_{\alpha,\,\beta\in\Omega}$, i.e., the maps $(\omega_{\lambda;\,\alpha,\,\beta}+t\omega^N_{\lambda;\,\alpha,\,\beta})_{\alpha,\,\beta\in\Omega}$ are associative for any $t$.
\item \label{item: 2-cocycle} The pair $(A, (\omega^N_{\lambda;\,\alpha,\,\beta})_{\alpha,\,\beta\in\Omega})$ is an $\Omega$-associative conformal algebra if and only if $(\psi_{\lambda;\,\alpha,\,\beta}^N)_{\alpha,\,\beta\in\Omega}$ is a $2$-cocycle in $C^2(A,A)$, that is to say
    \begin{align*}
    &(\delta^2\psi^N)_{\lambda,\,\mu;\,\alpha,\,\beta,\,\gamma}(a,b,c)\\
    ={}&a\cdot_{\lambda;\,\alpha,\,\beta\gamma} \varphi_{\mu;\,\beta,\,\gamma}^{N}(b, c)-\varphi_{\lambda+\mu;\,\alpha\beta,\,\gamma}^{N} \big(a\cdot_{\lambda;\,\alpha,\,\beta} b, c \big)+\varphi_{\lambda;\,\alpha,\,\beta\gamma}^{N} \big(a, b\cdot_{\mu;\,\beta,\,\gamma} c \big)-\varphi_{\lambda;\,\alpha,\,\beta}^{N}(a, b)\cdot_{\lambda+\mu;\,\alpha\beta,\,\gamma} c=0.
    \end{align*}

\end{enumerate}
\end{theorem}

\begin{proof}
 (\ref{item:compatible}). First for all $a, b, c \in A$ and $\alpha,\beta,\gamma\in\Omega$, we show that the $(\omega^N_{\lambda;\,\alpha,\,\beta})_{\alpha,\,\beta\in\Omega}$ is compatible with $(\omega_{\lambda;\,\alpha,\,\beta})_{\alpha,\,\beta\in\Omega}$, that is
\begin{equation}\label{eq:deformation associativity}
\big(a \circ_{\lambda;\,\alpha,\,\beta}^{N} b\big)\cdot_{\lambda+\mu;\,\alpha\beta,\,\gamma} c+\big(a\cdot_{\lambda;\,\alpha,\,\beta} b\big) \circ_{\lambda+\mu;\,\beta,\,\gamma}^{N} c=a\cdot_{\lambda;\,\alpha,\,\beta\gamma}\big(b \circ_{\mu;\,\beta,\,\gamma}^{N} c\big)+a \circ_{\lambda;\,\alpha,\,\beta\gamma}^{N}\big(b\cdot_{\mu;\,\beta,\,\gamma} c\big)
.
\end{equation}
Then we have
\begin{align*}
& \big(a \circ_{\lambda;\,\alpha,\,\beta}^{N} b\big)\cdot_{\lambda+\mu;\,\alpha\beta,\,\gamma} c+ \big(a\cdot_{\lambda;\,\alpha,\,\beta} b\big) \circ_{\lambda+\mu;\,\alpha\beta,\,\gamma}^{N} c \\
={}& \big(N_\alpha(a)\cdot_{\lambda;\,\alpha,\,\beta} b+a\cdot_{\lambda;\,\alpha,\,\beta} N_\beta(b)-N_{\alpha\beta} \big(a\cdot_{\lambda;\,\alpha,\,\beta} b\big)\big)\cdot_{\lambda+\mu;\,\alpha\beta,\,\gamma} c+N_{\alpha\beta} \big(a\cdot_{\lambda;\,\alpha,\,\beta} b\big)\cdot_{\lambda+\mu;\,\alpha\beta,\,\gamma} c\\
&+ \big(a\cdot_{\lambda;\,\alpha,\,\beta} b\big)\cdot_{\lambda+\mu;\,\alpha\beta,\,\gamma} N_\gamma(c)-N_{\alpha\beta\gamma} \big( \big(a\cdot_{\lambda;\,\alpha,\,\beta} b\big)\cdot_{\lambda+\mu;\,\alpha\beta,\,\gamma} c\big) \\
={}&\big(N_\alpha(a)\cdot_{\lambda;\,\alpha,\,\beta\gamma} b+a\cdot_{\lambda;\,\alpha,\,\beta\gamma} N_\beta(b)\big)\cdot_{\lambda+\mu;\,\alpha\beta,\,\gamma} c+ \big(a\cdot_{\lambda;\,\alpha,\,\beta} b\big)\cdot_{\lambda+\mu;\,\alpha\beta,\,\gamma} N_\gamma(c)-N_{\alpha\beta\gamma} \big( \big(a\cdot_{\lambda;\,\alpha,\,\beta} b\big)\cdot_{\lambda+\mu;\,\alpha\beta,\,\gamma} c\big) \\
={}& a\cdot_{\lambda;\,\alpha,\,\beta\gamma} \big(b\cdot_{\mu;\,\beta,\,\gamma} N_\gamma(c)\big)+a\cdot_{\lambda;\,\alpha,\,\beta\gamma} \big(N_\beta(b)\cdot_{\mu;\,\beta,\,\gamma} c\big)+N_\alpha(a)\cdot_{\lambda;\,\alpha,\,\beta\gamma} \big(b\cdot_{\mu;\,\beta,\,\gamma} c\big)-N_{\alpha\beta\gamma} \big(a\cdot_{\lambda;\,\alpha,\,\beta\gamma} \big(b\cdot_{\mu;\,\beta,\,\gamma} c\big)\big) \\
={}&a\cdot_{\lambda;\,\alpha,\,\beta\gamma} \big(b\cdot_{\mu;\,\beta,\,\gamma} N_\gamma(c)\big)+a\cdot_{\lambda;\,\alpha,\,\beta\gamma} \big(N_\beta(b)\cdot_{\mu;\,\beta,\,\gamma} c\big)
-a\cdot_{\lambda;\,\alpha,\,\beta\gamma} N_{\beta\gamma} \big(b\cdot_{\mu;\,\beta,\,\gamma} c\big)+a\cdot_{\lambda;\,\alpha,\,\beta\gamma} N_{\beta\gamma} \big(b\cdot_{\mu;\,\beta,\,\gamma} c\big)\\
&+N_\alpha(a)\cdot_{\lambda;\,\alpha,\,\beta\gamma} \big(b\cdot_{\mu;\,\beta,\,\gamma} c\big)-N_{\alpha\beta\gamma} \big(a\cdot_{\lambda;\,\alpha,\,\beta\gamma} \big(b\cdot_{\mu;\,\beta,\,\gamma} c\big)\big) \\
={}& a\cdot_{\lambda;\,\alpha,\,\beta\gamma} \big(b \circ_{\mu;\,\beta,\,\gamma}^{N} c\big)+a \circ_{\lambda;\,\alpha,\,\beta\gamma}^{N} \big(b\cdot_{\mu;\,\beta,\gamma} c\big) .
\end{align*}
The associativity of $\omega_{\lambda;\,\alpha,\,\beta}+t\omega^N_{\lambda;\,\alpha,\,\beta}$ is exactly  Eq.~(\ref{eq:deformation associativity}). Next, we have
\begin{align}
&\big(a \circ_{\lambda;\,\alpha,\,\beta}^{N} b\big) \circ_{\lambda+\mu;\,\alpha\beta,\,\gamma}^{N} c- a \circ_{\lambda;\,\alpha,\,\beta\gamma}^{N}\big(b \circ_{\mu;\,\beta,\,\gamma}^{N} c\big)\nonumber \\
={}& N_{\alpha\beta}\big(a \circ_{\lambda;\,\alpha,\,\beta}^{N} b\big)\cdot_{\lambda+\mu;\,\alpha\beta,\,\gamma} c+\big(a \circ_{\lambda;\,\alpha,\,\beta}^{N} b\big)\cdot_{\lambda+\mu;\,\alpha\beta,\,\gamma} N_\gamma(c)-N_{\alpha\beta\gamma}\big(\big(a \circ_{\lambda;\,\alpha,\,\beta}^{N} b\big)\cdot_{\lambda+\mu;\,\alpha\beta,\,\gamma} c\big)\nonumber \\
&-N_\alpha(a)\cdot_{\lambda;\,\alpha,\,\beta\gamma}\big(b \circ_{\mu;\,\beta,\,\gamma}^{N} c\big)-a\cdot_{\lambda;\,\alpha,\,\beta\gamma} N_{\beta\gamma}\big(b \circ_{\mu;\,\beta,\,\gamma}^{N} c\big)+N_{\alpha\beta\gamma}\big(a\cdot_{\lambda;\,\alpha,\,\beta\gamma}\big(b \circ_{\mu;\,\beta,\,\gamma}^{N} c\big)\big) \nonumber\\
={}&N_{\alpha\beta}\big(a \circ_{\lambda;\,\alpha,\,\beta}^{N} b\big)\cdot_{\lambda+\mu;\,\alpha\beta,\,\gamma} c-a\cdot_{\lambda;\,\alpha,\,\beta\gamma} N_{\beta\gamma}\big(b \circ_{\mu;\,\beta,\,\gamma}^{N} c\big)-N_{\alpha\beta\gamma}\big(\big(a \circ_{\lambda;\,\alpha,\,\beta}^{N} b\big)\cdot_{\lambda+\mu;\,\alpha\beta,\,\gamma} c\nonumber\\
&-a\cdot_{\lambda;\,\alpha,\,\beta\gamma}\big(b \circ_{\mu;\,\beta,\,\gamma}^{N} c\big)\big)
+\left(a\cdot_{\lambda;\,\alpha,\,\beta}N_\beta(b)+N_\alpha(a)\cdot_{\lambda;\,\alpha,\,\beta}b
-N_{\alpha\beta}(a\cdot_{\lambda;\,\alpha,\,\beta}b)\right)
\circ_{\lambda+\mu;\,\alpha\beta,\,\gamma}N_\gamma(c)\nonumber\\
&-N_\alpha(a)\cdot_{\lambda;\,\alpha,\,\beta\gamma}\left(N_\beta(b)\cdot_{\mu;\,\beta,\,\gamma}c
+b\cdot_{\mu;\,\beta,\,\gamma}N_\gamma(c)
-N_{\beta\gamma}(b\cdot_{\mu;\,\beta,\,\gamma}c)\right)\nonumber\\
={}& N_{\alpha\beta}\big(a \circ_{\lambda;\,\alpha,\,\beta}^{N} b\big)\cdot_{\lambda+\mu;\,\alpha\beta,\,\gamma} c-a\cdot_{\lambda;\,\alpha,\,\beta\gamma} N_{\beta\gamma}\big(b \circ_{\mu;\,\beta,\,\gamma}^{N} c\big)-N_{\alpha\beta\gamma}\big(\big(a \circ_{\lambda;\,\alpha,\,\beta}^{N} b\big)\cdot_{\lambda+\mu;\,\alpha\beta,\,\gamma} c\nonumber\\
&-a\cdot_{\lambda;\,\alpha,\,\beta\gamma}\big(b \circ_{\mu;\,\beta,\,\gamma}^{N} c\big)\big)
+\big(a\cdot_{\lambda;\,\alpha,\,\beta} N_\beta(b)\big)\cdot_{\lambda+\mu;\,\alpha\beta,\,\gamma}N_\gamma(c)
+\big(N_\alpha(a)\cdot_{\lambda;\,\alpha,\,\beta} b\big)
\cdot_{\lambda+\mu;\,\alpha\beta,\,\gamma} N_\gamma(c)\nonumber\\
&
-N_{\alpha\beta}(a\cdot_{\lambda;\,\alpha,\,\beta}b)\cdot_{\lambda+\mu;\,\alpha\beta,\,\gamma}N_\gamma(c)
-N_\alpha(a)\cdot_{\lambda;\,\alpha,\,\beta\gamma}\big(N_\beta(b)\cdot_{\mu;\,\beta,\,\gamma} c)\nonumber\\
&-N_\alpha(a)\cdot_{\lambda;\,\alpha,\,\beta\gamma}(b\cdot_{\mu;\,\beta,\,\gamma}N_\gamma(c))
+N_\alpha(a)\cdot_{\lambda;\,\alpha,\,\beta\gamma}N_{\beta\gamma}(b\cdot_{\mu;\,\beta,\,\gamma}c)\nonumber
\\
={}& N_{\alpha\beta}\big(a \circ_{\lambda;\,\alpha,\,\beta}^{N} b\big)\cdot_{\lambda+\mu;\,\alpha\beta,\,\gamma} c-a\cdot_{\lambda;\,\alpha,\,\beta\gamma} N_{\beta\gamma}\big(b \circ_{\mu;\,\beta,\,\gamma}^{N} c\big)-N_{\alpha\beta\gamma}\big(\big(a \circ_{\lambda;\,\alpha,\,\beta}^{N} b\big)\cdot_{\lambda+\mu;\,\alpha\beta,\,\gamma} c\nonumber\\
&-a\cdot_{\lambda;\,\alpha,\,\beta\gamma}\big(b \circ_{\mu;\,\beta,\,\gamma}^{N} c\big)\big)
+a\cdot_{\lambda;\,\alpha,\,\beta\gamma}\big(N_\beta(b)\cdot_{\mu;\,\beta,\,\gamma} N_\gamma(c)\big)-N_{\alpha\beta}\big(a\cdot_{\lambda;\,\alpha,\,\beta} b\big)\cdot_{\lambda+\mu;\,\alpha\beta,\,\gamma} N_\gamma(c)\nonumber\\
&-\big(N_\alpha(a)\cdot_{\lambda;\,\alpha,\,\beta} N_\beta(b)\big)\cdot_{\lambda+\mu;\,\alpha\beta,\,\gamma} c
+N_\alpha(a)\cdot_{\lambda;\,\alpha,\,\beta\gamma} N_{\beta\gamma}\big(b\cdot_{\mu;\,\beta,\,\gamma} c\big),\label{eq:asso}
\end{align}

and we also have
\begin{align}
&\big(\delta^2\varphi^{N}\big)_{\lambda,\,\mu;\,\alpha,\,\beta,\,\gamma}(a, b, c)\nonumber\\
={}& a\cdot_{\lambda;\,\alpha,\,\beta\gamma} \varphi_{\mu;\,\beta,\,\gamma}^{N}(b, c)-\varphi_{\lambda+\mu;\,\alpha\beta,\,\gamma}^{N}\big(a\cdot_{\lambda;\,\alpha,\,\beta} b, c\big)+\varphi_{\lambda;\,\alpha,\,\beta\gamma}^{N}\big(a, b\cdot_{\mu;\,\beta,\,\gamma} c\big)-\varphi_{\lambda;\,\alpha,\,\beta}^{N}(a, b)\cdot_{\lambda+\mu;\,\alpha\beta,\,\gamma} c \nonumber\\
={}& a\cdot_{\lambda;\,\alpha,\,\beta\gamma}\big(N_\beta(b)\cdot_{\mu;\,\beta,\,\gamma} N_\gamma(c)-N_{\beta\gamma}\big(b \circ_{\mu;\,\beta,\,\gamma}^{N} c\big)\big)-\big(N_{\alpha\beta}\big(a\cdot_{\lambda;\,\alpha,\,\beta} b\big)\cdot_{\lambda+\mu;\,\alpha\beta,\,\gamma} N_\gamma(c)\nonumber\\
&-N_{\alpha\beta\gamma}\big(\big(a\cdot_{\lambda;\,\alpha,\,\beta} b\big) \circ_{\lambda+\mu;\,\alpha\beta,\,\gamma}^{N} c\big)\big)
+N_\alpha(a)\cdot_{\lambda;\,\alpha,\,\beta\gamma} N_{\beta\gamma}\big(b\cdot_{\mu;\,\beta,\,\gamma} c\big)-N_{\alpha\beta\gamma}\big(a \circ_{\lambda;\,\alpha,\,\beta\gamma}^{N}\big(b\cdot_{\mu;\,\beta,\,\gamma} c\big)\big)\nonumber\\
&-\big(N_\alpha(a)\cdot_{\lambda;\,\alpha,\,\beta} N_\beta(b)
-N_{\alpha\beta}\big(a \circ_{\lambda;\,\alpha,\,\beta}^{N} b\big)\big)\cdot_{\lambda+\mu;\,\alpha\beta,\,\gamma} c \nonumber\\
={}& N_{\alpha\beta}\big(a \circ_{\lambda;\,\alpha,\,\beta}^{N} b\big)\cdot_{\lambda+\mu;\,\alpha\beta,\,\gamma} c
-a\cdot_{\lambda;\,\alpha,\,\beta\gamma} N_{\beta\gamma}\big(b \circ_{\mu;\,\beta,\,\gamma}^{N} c\big)
+a\cdot_{\lambda;\,\alpha,\,\beta\gamma}\big(N_\beta(b)\cdot_{\mu;\,\beta,\,\gamma} N_\gamma(c)\big)\nonumber\\
&-N_{\alpha\beta}\big(a\cdot_{\lambda;\,\alpha,\,\beta} b\big)\cdot_{\lambda+\mu;\,\alpha\beta,\,\gamma} N_\gamma(c)
+N_\alpha(a)\cdot_{\lambda;\,\alpha,\,\beta\gamma} N_{\beta\gamma}\big(b\cdot_{\mu;\,\beta,\,\gamma} c\big)
-\big(N_\alpha(a)\cdot_{\lambda;\,\alpha,\,\beta} N_\beta(b)\big)\cdot_{\lambda+\mu;\,\alpha\beta,\,\gamma} c\nonumber\\
&+N_{\alpha\beta\gamma}\big(\big(a\cdot_{\lambda;\,\alpha,\,\beta} b\big) \circ_{\lambda+\mu;\,\alpha\beta,\,\gamma}^{N} c
-a \circ_{\lambda;\,\alpha,\,\beta\gamma}^{N}\big(b\cdot_{\mu;\,\beta,\,\gamma} c\big)\big)\label{eq:cocy}.
\end{align}

By comparing Eq.~(\ref{eq:asso}) and Eq.~(\ref{eq:cocy}), and applying Eq.~(\ref{eq:deformation associativity}), we get
\[\big(\delta^2 \varphi^{N}\big)_{\lambda,\, \mu;\,\alpha,\,\beta,\,\gamma}(a, b, c)
=\big(a \circ_{\lambda;\,\alpha,\,\beta}^{N} b\big) \circ_{\lambda+\mu;\,\alpha,\,\beta\gamma}^{N} c- a \circ_{\lambda;\,\beta,\,\gamma}^{N}\big(b \circ_{\mu;\,\beta,\,\gamma}^{N} c\big)=0.\]
This completes the proof.
\end{proof}

\subsection{Formal deformations of Rota-Baxter family $\Omega$-associative conformal algebras}\
Let $(A, (\mu_{\lambda;\,\alpha,\,\beta})_{\alpha,\,\beta\in\Omega}, (P_\omega)_{\omega\in\Omega})$ be a Rota-Baxter family $\Omega$-associative conformal algebra of weight $q$.   Consider a 1-parameterized family: for any $\omega\in\Omega$, define
\[\mu_{\lambda,\,t;\,\alpha,\,\beta}=\sum_{i=0}^\infty \mu_{\lambda,\,i;\,\alpha,\,\beta}t^i, \ \mu_{\lambda,\,i;\,\alpha,\,\beta}\in C^2_\Alg(A),\quad  P_{t,\,\omega}=\sum_{i=0}^\infty P_{i,\,\omega}t^i,  \ P_{i,\,\omega}\in C^1_{\RBO}(A).\]

\begin{defn}
	A  1-parameter formal deformation of  Rota-Baxter family $\Omega$-associative conformal algebra $(A, (\mu_{\lambda;\,\alpha,\,\beta})_{\alpha,\,\beta\in\Omega},(P_\omega)_{\omega\in\Omega})$ is a pair $((\mu_{\lambda,\,t;\,\alpha,\,\beta})_{\alpha,\,\beta\in\Omega},(P_{t,\,\omega})_{\omega\in\Omega})$ which endows the flat $\mathbb{F}[[t]]$-module $A[[t]]$ with a  Rota-Baxter family $\Omega$-associative conformal algebra structure over $\mathbb{F}[[t]]$ such that $(\mu_{\lambda, 0;\,\alpha,\,\beta},P_{0,\,\omega})=(\mu_{\lambda;\,\alpha,\,\beta},P_\omega)$.
\end{defn}

 Power series $(\mu_{\lambda,\,t;\,\alpha,\,\beta})_{\alpha,\,\beta\in\Omega}$ and $ (P_{t,\,\omega})_{\omega\in\Omega}$ determine a  1-parameter formal deformation of Rota-Baxter family $\Omega$-associative conformal algebra $(A, (\mu_{\lambda;\,\alpha,\,\beta})_{\omega\in\Omega},(P_\omega)_{\omega\in\Omega})$ if and only if for any $a,b,c\in A$, the following equations hold:
 \begin{align*}
 \mu_{\lambda,\,t;\,\alpha,\,\beta\gamma}(a\ot \mu_{\mu,\,t;\,\beta,\,\gamma}(b\ot c))&=\mu_{\lambda+\mu,\,t;\,\alpha\beta,\,\gamma}(\mu_{\lambda,\,t;\,\alpha,\,\beta}(a\ot b)\ot c),\\
\mu_{\lambda,\, t;\,\alpha,\,\beta}(P_{t,\,\alpha}(a)\ot P_{t,\,\beta}(b))&= P_{t,\,\alpha\beta}\Big(\mu_{\lambda,\,t;\,\alpha,\,\beta}(a\ot P_{t,\,\beta}(b))+\mu_{\lambda,\,t;\,\alpha,\,\beta}(P_{t,\,\alpha}(a)\ot b)+q\mu_{\lambda,\,t;\,\alpha,\,\beta}(a\ot b)\Big).
 \end{align*}
By expanding these equations and comparing the coefficient of $t^n$, we obtain  that $(\mu_{\lambda,\,i;\,\alpha,\,\beta})_{i\geqslant0,\,\alpha,\,\beta\in\Omega}$ and $(P_{i,\,\alpha})_{i\geqslant0,\,\alpha\in\Omega}$ have to  satisfy: for any $n\geqslant 0$,
\begin{align} \sum_{i=0}^n\mu_{\lambda+\mu,i;\,\alpha\beta,\,\gamma}\circ(\mu_{\lambda,n-i;\,\alpha,\,\beta}\ot \id)={}&\sum_{i=0}^n\mu_{\lambda,i;\,\alpha,\,\beta\gamma}\circ(\id\ot \mu_{\mu,n-i;\,\beta,\,\gamma}),\label{Eq: deform eq for  products in RBA}\\
\sum_{\substack{i+j+k=n\\ i, j, k\geqslant 0}}	\mu_{\lambda,i;\,\alpha,\,\beta}\circ(P_{j,\alpha}\ot P_{k,\,\beta})={}&\sum_{\substack{i+j+k=n\\ i, j, k\geqslant 0}} P_{i,\alpha\beta}\circ \mu_{\lambda,j;\,\alpha,\,\beta}\circ (\id\ot P_{k,\,\beta})\nonumber\\
&+\sum_{\substack{i+j+k=n\\ i, j, k\geqslant 0}} P_{i,\alpha\beta}\circ\mu_{\lambda,j;\,\alpha,\,\beta}\circ (P_{k,\alpha}\ot \id)+q\sum_{\substack{i+j=n\\ i, j \geqslant 0}}P_{i,\alpha\beta}\circ\mu_{\lambda,j;\,\alpha,\,\beta}.
\label{Eq: Deform RB operator in RBA}
\end{align}
Obviously, when $n=0$, the above conditions are exactly the associativity of $\mu_{\lambda;\,\alpha,\,\beta}=\mu_{\lambda,\,0;\,\alpha,\,\beta}$ and Eq.~(\ref{Eq: Deform RB operator in RBA}) which is the defining relation of Rota-Baxter family operator $P_\omega=P_{0,\,\omega}$.


\begin{prop}\label{Prop: Infinitesimal is 2-cocyle}
	Let $(A[[t]], (\mu_{\lambda,\,t;\,\alpha,\,\beta})_{\alpha,\,\beta\in\Omega},(P_{t,\,\alpha})_{\alpha\in\Omega})$ be a  1-parameter formal deformation of
	Rota-Baxter family $\Omega$-associative conformal algebra 
of weight $q$. Then
	$((\mu_{\lambda,\,1;\,\alpha,\,\beta})_{\alpha,\,\beta\in\Omega},(P_{1,\,\alpha})_{\alpha\in\Omega})$ is a 2-cocycle in the cochain complex
	$C_{\RBA}^\bullet(A)$.
\end{prop}
\begin{proof} When $n=1$,   Eqs.~(\ref{Eq: deform eq for  products in RBA}) and (\ref{Eq: Deform RB operator in RBA})  become
	$$\mu_{\lambda+\mu,1;\,\alpha\beta,\,\gamma}\circ(\mu_{\mu;\,\alpha,\,\beta}\ot \id)+\mu_{\lambda+\mu;\,\alpha\beta,\,\gamma}\circ(\mu_{\mu,1;\alpha,\,\beta}\ot \id)=\mu_{\lambda,1;\,\alpha,\,\beta\gamma}\circ(\id\ot \mu_{\mu;\,\beta,\,\gamma})+\mu_{\lambda;\,\alpha,\,\beta\gamma}\circ (\id\ot \mu_{\mu,1;\,\beta,\,\gamma}),$$
and
\begin{align*}
	&\mu_{\lambda,1;\,\alpha,\,\beta} (P_\alpha\ot P_\beta)-\{P_{\alpha\beta}\circ\mu_{\lambda,1;\,\alpha,\,\beta}\circ(\id\ot P_\beta)+P_{\alpha\beta}\circ\mu_{\lambda,1;\,\alpha,\,\beta}\circ(P_\alpha\ot \id)+q P_{\alpha\beta}\circ \mu_{\lambda,1;\,\alpha,\,\beta}\}\\
	={}&-\{\mu_{\lambda;\,\alpha,\,\beta}\circ(P_\alpha\ot P_{1,\,\beta})-P_{\alpha\beta}\circ\mu_{\lambda;\,\alpha,\,\beta}\circ(\id\ot P_{1,\,\beta})\}\\
&-\{\mu_{\lambda;\,\alpha,\,\beta}\circ(P_{1,\,\alpha}\ot P_\beta)-P_{\alpha\beta}\circ\mu_{\lambda;\,\alpha,\,\beta}\circ(P_{1,\,\alpha}\ot\id)\}\\
	&+\{P_{1,\alpha\beta}\circ\mu_{\lambda;\,\alpha,\,\beta}\circ(\id\ot P_\beta)+P_{1,\alpha\beta}\circ\mu_{\lambda;\,\alpha,\,\beta}\circ(P_\alpha\ot \id)+q P_{1,\alpha\beta}\circ \mu_{\lambda;\,\alpha,\,\beta}\}.
	\end{align*}
Note that  the first equation is exactly $\delta^2(\mu_1)_{\lambda,\,\mu;\,\alpha,\,\beta,\,\gamma}=0\in C^\bullet_{\Alg}(A)$ and that  second equation is exactly to  \[\Phi^2(u_1)_{\lambda;\,\alpha,\,\beta}=-\partial^1(p_{1})_{\lambda;\,\alpha,\,\beta} \in C^\bullet_{\RBO}(A).\]
	So $((\mu_{\lambda,\,1;\,\alpha,\,\beta})_{\alpha,\,\beta\in\Omega},(P_{1,\,\alpha})_{\alpha\in\Omega})$ is a 2-cocycle in $C^\bullet_{\RBA}(A)$.
	\end{proof}


\begin{defn} The 2-cocycle $((\mu_{\lambda,\,1;\,\alpha,\,\beta})_{\alpha,\,\beta\in\Omega},(P_{1,\,\omega})_{\omega\in\Omega})$ is called {\bf the infinitesimal of the 1-parameter formal deformation} $(A[[t]], (\mu_{\lambda,\,t;\,\alpha,\,\beta})_{\alpha,\,\beta\in\Omega}, (P_{t,\,\omega})_{\omega\in\Omega})$ of Rota-Baxter family $\Omega$-associative conformal algebra $(A, (\mu_{\lambda;\,\alpha,\,\beta})_{\alpha,\,\beta\in\Omega},(P_\omega)_{\omega\in\Omega})$.
\end{defn}

\begin{defn}
 We call the two 1-parameter formal deformations $(A[[t]], (\mu_{\lambda,t;\,\alpha,\,\beta})_{\alpha,\,\beta\in\Omega},(P_{t,\,\omega})_{\omega\in\Omega})$ and
	$(A[[t]],$ $(\mu'_{\lambda,t;\,\alpha,\,\beta})_{\alpha,\,\beta\in\Omega},(P'_{t,\,\omega})_{\omega\in\Omega})$ are  {\bf equivalent}
 if there exists a power series formal isomorphism
 \[
 \psi_{t;\,\omega}=\sum_{i=0}\psi_{i;\,\omega}t^i: A[[t]]\rightarrow A[[t]],\]
  where $(\psi_{i;\,\omega})_{\omega\in\Omega}: A\rightarrow A$ are linear maps with $\psi_0=\id_A$, such that:
\begin{eqnarray}\label{Eq: equivalent deformations}
\psi_{t;\,\alpha\beta}\circ \mu_{\lambda,t;\,\alpha,\,\beta}' &=& \mu_{\lambda,t;\,\alpha,\,\beta}\circ (\psi_{t;\,\alpha}\ot \psi_{t;\,\beta}),\\
\psi_{t;\,\omega}\circ P_{t,\,\omega}'&=&P_{t,\,\omega}\circ\psi_{t,\,\omega}. \label{Eq: equivalent deformations2}
	\end{eqnarray}
	
\end{defn}


Given a Rota-Baxter family $\Omega$-associative conformal algebra $(A, (\mu_{\lambda;\,\alpha,\,\beta})_{\alpha,\,\beta\in\Omega},(P_\omega)_{\omega\in\Omega})$, the power series $(\mu_{\lambda,\,t;\,\alpha,\,\beta})_{\alpha,\,\beta\in\Omega}$ and $(P_{t,\,\omega})_{\omega\in\Omega}$
with $\mu_{\lambda,\,i;\,\alpha,\,\beta}=\delta_{i,0}\mu_{\lambda;\,\alpha,\,\beta},\, P_{i,\omega}$ $=\delta_{i,0}P_\omega$ make
$(A[[t]], (\mu_{\lambda,\,t;\,\alpha,\,\beta})$ $_{\alpha,\,\beta\in\Omega},$ $(P_{t,\,\omega})_{\omega\in\Omega})$ into a $1$-parameter formal deformation of
$(A, (\mu_{\lambda;\,\alpha,\,\beta})$ $_{\alpha,\,\beta\in\Omega},(P_\omega)_{\omega\in\Omega})$. Formal deformations equivalent to this one are called {\bf trivial}.

\begin{theorem}
The infinitesimals of two equivalent 1-parameter formal deformations of $(A, (\mu_{\lambda;\,\alpha,\,\beta})$ $_{\alpha,\,\beta\in\Omega},$ $(P_\omega)_{\omega\in\Omega})$ are in the same cohomology class in $\rmH^\bullet_{\RBA}(A)$.
\end{theorem}

\begin{proof} Let $(\psi_{t;\,\alpha,\,\beta})_{\alpha,\,\beta\in\Omega}:(A[[t]], (\mu_{\lambda,\,t';\,\alpha,\,\beta})_{\alpha,\,\beta\in\Omega}, (P'_{t,\,\omega})_{\omega\in\Omega})\rightarrow (A[[t]], (\mu_{\lambda,\,t;\,\alpha,\,\beta})_{\alpha,\,\beta\in\Omega}, (P_{t,\omega})_{\omega\in\Omega})$ be a formal isomorphism.
	Expanding the identities and collecting coefficients of $t$, from Eqs.~(\ref{Eq: equivalent deformations}) and (\ref{Eq: equivalent deformations2}), we have:
	\begin{eqnarray*}
		\mu_{\lambda,1;\,\alpha,\,\beta}'&=&\mu_{\lambda,1;\,\alpha,\,\beta}+\mu_{\lambda;\,\alpha,\,\beta}\circ(\id\ot \psi_{1;\,\beta})-\psi_{1;\,\alpha\beta}\circ\mu_{\lambda;\,\alpha,\,\beta}+\mu_{\lambda;\,\alpha,\,\beta}
\circ(\psi_{1;\,\alpha}\ot \id),\\
		P'_{1,\,\omega}&=&P_{1,\,\omega}+P_\omega\circ\psi_{1;\,\,\omega}-\psi_{1;\,\omega}\circ P_\omega,
		\end{eqnarray*}
	that is, we have
\[
(\mu'_{\lambda,1;\,\alpha,\,\beta},P'_{1,\,\omega})-
(\mu_{\lambda,1;\,\alpha,\,\beta},P_{1,\,\omega})
=(\delta^1(\psi_1)_{\lambda;\,\alpha,\,\beta}, -\Phi^1(\psi_1)_{\omega})=d^1(\psi_{1})_{\lambda;\,\alpha,\,\beta,\,\omega}\in  C^\bullet_{\RBA}(A).\qedhere\]
\end{proof}

\section{Abelian extensions of Rota-Baxter family $\Omega$-associative conformal algebras}
\label{sec:abelian extension}

 In this section, we study abelian extensions of Rota-Baxter family $\Omega$-associative conformal algebras and show that they are classified by the second cohomology, as one would expect of a good cohomology theory.

 Notice that a vector space $M$ together with a family of linear transformations $(P_{M,\,\omega})_{\omega\in\Omega}$ with $P_{M,\,\omega}:M\to M$ is naturally a Rota-Baxter family $\Omega$-associative conformal algebras where the multiplication on $M$ is defined to be $u\Cdot\lambda,\alpha,\,\beta;v=0$ for all $u,v\in M$ and $\alpha,\,\beta\in\Omega.$

 \begin{defn}
 	{\bf An  abelian extension} of Rota-Baxter family $\Omega$-associative conformal algebras is a short exact sequence of  morphisms of Rota-Baxter family $\Omega$-associative conformal algebras
 \begin{eqnarray}\label{Eq: abelian extension}
 0\to (M, (P_{M,\,\omega})_{\omega\in\Omega})\xrightarrow{(i_\omega)_{\omega\in\Omega}} (\hat{A}, (\hat{P}_\omega)_{\omega\in\Omega})\xrightarrow{(p_\omega)_{\omega\in\Omega}} (A, (P_\omega)_{\omega\in\Omega})\to 0,
 \end{eqnarray}
 that is, there exists a commutative diagram:
 	\[\begin{CD}
 		0@>>> {M} @>i_\omega >> \hat{A} @>p_\omega >> A @>>>0\\
 		@. @V {P_{M,\,\omega}} VV @V {\hat{P}_\omega} VV @V P_\omega VV @.\\
 		0@>>> {M} @>i_\omega >> \hat{A} @>p_\omega >> A @>>>0,
 	\end{CD}\]
 where the Rota-Baxter family $\Omega$-associative conformal algebra $(M, (P_{M,\,\omega})_{\omega\in\Omega})$	satisfies  $u\Cdot\lambda,\alpha,\,\beta;v=0$ for all $u,v\in M.$

 We will call $(\hat{A},(\hat{P}_\omega)_{\omega\in\Omega})$ an abelian extension of $(A,(P_\omega)_{\omega\in\Omega})$ by $(M,(P_{M,\,\omega})_{\omega\in\Omega})$.
 \end{defn}

 \begin{defn}
 	Let $(\hat{A}_1,(\hat{P}^1_{\omega})_{\omega\in\Omega})$ and $(\hat{A}_2,(\hat{P}^2_{\omega})_{\omega\in\Omega})$ be two abelian extensions of $(A,(P_\omega)_{\omega\in\Omega})$ by $(M, (P_{M,\,\omega})_{\omega\in\Omega})$. They are said to be {\bf isomorphic}  if there exists an isomorphism of Rota-Baxter family $\Omega$-asssociative conformal algebras $(\zeta_\omega)_{\omega\in\Omega}:(\hat{A}_1,(\hat{P}^1_{\omega})_{\omega\in\Omega})\rar (\hat{A}_2,(\hat{P}^2_{\omega})_{\omega\in\Omega})$ such that the following commutative diagram holds:
 	\begin{eqnarray}\label{Eq: isom of abelian extension}\begin{CD}
 		0@>>> {(M,T_M)} @>i_\omega >> (\hat{A}_1,(\hat{P}^1_{\omega})_{\omega\in\Omega}) @>p^1_\omega >> (A,(P_\omega)_{\omega\in\Omega}) @>>>0\\
 		@. @| @V \zeta_\omega VV @| @.\\
 		0@>>> {(M,T_M)} @>i_\omega >> (\hat{A}_2,(\hat{P}^2_{\omega})_{\omega\in\Omega}) @>p^2_\omega >> (A,(P_\omega)_{\omega\in\Omega}) @>>>0.
 	\end{CD}\end{eqnarray}
 \end{defn}

 A   section of an abelian extension $(\hat{A},(\hat{P}_\omega)_{\omega\in\Omega})$ of $(A,(\hat{P}_\omega)_{\omega\in\Omega})$ by $(M,(P_{M,\,\omega})_{\omega\in\Omega})$ is a linear map $s_\omega:A\rar \hat{A}$ such that $p_\omega\circ s_\omega=\Id_A$.

 We will show that isomorphism classes of  abelian extensions of $(A,(\hat{P}_\omega)_{\omega\in\Omega})$ by $(M,(P_{M,\,\omega})_{\omega\in\Omega})$ are in bijection with the second cohomology group   ${\rmH}_{\RBA}^2(A,M)$.


Let  $(\hat{A},(\hat{P}_\omega)_{\omega\in\Omega})$ be  an abelian extension of $(A,(P_\omega)_{\omega\in\Omega})$ by $(M,(P_{M,\,\omega})_{\omega\in\Omega})$ having the form Eq.~\eqref{Eq: abelian extension}.

 \begin{prop}\label{Prop: new RB bimodules from abelian extensions}
 	Let $(s_\omega)_{\omega\in\Omega}:A\rar \hat{A}$ be a family of sections . Define
 \[
 a\Cdot\lambda,\alpha,\beta;m:=s_\alpha(a)\Cdot\lambda,\alpha,\beta;m,\quad m\Cdot\lambda,\alpha,\beta;a:=m\Cdot\lambda,\alpha,\beta;s_\beta(a), \quad \text{ for }\, a\in A, m\in M,\,\alpha,\,\beta\in\Omega.\]
 Then $(M, (P_{M,\,\omega})_{\omega\in\Omega})$ is a Rota-Baxter family $\Omega$-associative conformal bimodule on $(A, (\mu_{\lambda;\,\alpha,\,\beta})_{\alpha,\,\beta\in\Omega},$ $\,(P_{\omega})_{\omega\in\Omega})$.
 \end{prop}
 \begin{proof}
 	For any $a,b\in A,\,m\in M$, since $s_{\alpha\beta}(a\Cdot\lambda,\alpha,\beta;b)-s_\alpha(a)\Cdot\lambda,\alpha,\beta;s_\beta(b)\in M$ implies \[s_{\alpha\beta}(a\Cdot\lambda,\alpha,\beta;b)\Cdot\lambda+\mu,\alpha\beta,\gamma;m
 =s_\alpha(a)\Cdot\lambda,\alpha,\beta\gamma;
 \left(s_\beta(b)
\Cdot\mu,\beta,\gamma;m\right),\]
 we have
 	\[ (a\Cdot\lambda,\alpha,\beta;b) \Cdot\lambda+\mu,\alpha\beta,\gamma; m=s_{\alpha\beta}(a\Cdot\lambda,\alpha,\beta;b)\Cdot\lambda+\mu,\alpha\beta,\gamma;m
 =\left(s_\alpha(a)\Cdot\lambda,\alpha,\beta;s_\beta(b)\right)\Cdot\lambda+\mu,\alpha,\beta;m
 =a\Cdot\lambda,\alpha,\beta\gamma;(b\Cdot\mu,\,\beta,\gamma;m).\]
 	Hence,  this gives a left $A$-module structure and the case of right module structure is similar.

 Moreover, ${\hat{P}_\omega}(s_\omega(a))-s_\omega(P_\omega(a))\in M$ means that  \[{\hat{P}_\omega}(s_\omega(a))\Cdot\lambda,\alpha,\beta;m=s_\omega(P_\omega(a))\Cdot\lambda,\alpha,\beta;m.\]

  Thus we have
 	\begin{align*}
 		P_\alpha(a)\Cdot\lambda,\alpha,\beta;P_{M,\,\beta}(m)&=s_\omega(P_\alpha(a))\Cdot\lambda,\alpha,\beta;
 P_{M,\,\beta}(m)\\
 		&=\hat{P}_\alpha(s_\omega(a))\Cdot\lambda,\alpha,\beta;P_{M,\,\beta}(m)\\
 		&=\hat{P}_{\alpha\beta}(\hat{P}_\alpha(s_\omega(a))\Cdot\lambda,\alpha,\beta;m
 +s_\omega(a)\Cdot\lambda,\alpha,\beta;
 P_{M,\,\beta}(m)+q s_\omega(a)\Cdot\lambda,\alpha,\beta;m)\\
 		&=P_{M,\,\alpha\beta}(P_\alpha(a)\Cdot\lambda,\alpha,\beta;m+a\Cdot\lambda,\alpha,\beta;P_{M,\,\beta}(m)
 +q a\Cdot\lambda,\alpha,\beta;m).
 	\end{align*}
 	It is similar to see \[P_{M,\,\alpha}(m)\Cdot\lambda,\alpha,\beta;P_\beta(a)
 =P_{M,\,\alpha\beta}(P_{M,\,\alpha}(m)\Cdot\lambda,\alpha,\beta;a+m\Cdot\lambda,\alpha,\beta;P_\beta(a)+q m\Cdot\lambda,\alpha,\beta;a).\]
This completes the proof.
 \end{proof}

 We further  define conformal linear maps $(\psi_{\lambda;\,\alpha,\,\beta})_{\alpha,\,\beta\in\Omega}:A\ot A\rar M$ and $(\chi_\omega)_{\omega\in\Omega}:A\rar M$ respectively by
 \begin{align*}
 	\psi_{\lambda;\,\alpha,\,\beta}(a, b)&=s_\alpha(a)\Cdot\lambda,\alpha,\beta;s_\beta(b)-s_{\alpha\beta}
 (a\Cdot\lambda,\alpha,\beta;b),\quad\text{ for }\, a,b\in A,\,\alpha,\beta\in\Omega,\\
 	\chi_\omega(a)&={\hat{P}_\omega}(s_\omega(a))-s_\omega(P_\omega(a)),\quad
 \text{ for }\, a\in A, \omega\in\Omega.
 \end{align*}

 \begin{prop}
 	 Let $\Omega$ be a semigroup. The pair
 	$((\psi_{\lambda;\,\alpha,\,\beta})_{\alpha,\,\beta\in\Omega}, (\chi_\omega)_{\omega\in\Omega})$ is a 2-cocycle  of   Rota-Baxter family $\Omega$-associative conformal algebra $(A,(P_\omega)_{\omega\in\Omega})$ with  coefficients  in the Rota-Baxter family $\Omega$-associative conformal bimodule $(M,(P_{M,\,\omega})_{\omega\in\Omega})$ introduced in Proposition~\ref{Prop: new RB bimodules from abelian extensions}.
 \end{prop}

 \begin{proof}
 The proof is by direct computations, so it is left to the reader.
 \end{proof}	
 	
 The choice of the section $s_\omega$ in fact determines a splitting
 $$\xymatrix{0\ar[r]&  M\ar@<1ex>[r]^{i_\omega} &\hat{A}\ar@<1ex>[r]^{p_\omega} \ar@<1ex>[l]^{t_\omega}& A \ar@<1ex>[l]^{s_\omega} \ar[r] & 0}$$
 subject to $t_\omega\circ i_\omega=\Id_{M,\,\omega},\, t_\omega\circ s_\omega=0$ and $ i_\omega t_\omega+s_\omega p_\omega=\Id_{\hat{A},\,\omega}$.
 Then there is an induced isomorphism of vector spaces
 $$\left(\begin{array}{cc} p_\omega& t_\omega\end{array}\right): \hat{A}\cong   A\oplus M: \left(\begin{array}{c} s_\omega\\ i_\omega\end{array}\right).$$
We can  transfer the Rota-Baxter family $\Omega$-associative conformal algebra structure on $\hat{A}$ to $A\oplus M$ via this isomorphism.
  It is direct to verify that this  endows $A\oplus M$ with a multiplication $(\cdot^\psi_{\lambda;\,\alpha,\,\beta})_{\alpha,\,\beta\in\Omega}$ and a Rota-Baxter family operator $(P^\chi_\omega)_{\omega\in\Omega}$. For $a,b\in A,\,m,n\in M$ and $\alpha,\beta\in\Omega$,  define
 \begin{align}
 \label{eq:mul}(a,m)\cdot^\psi_{\lambda;\,\alpha,\,\beta}(b,n)&:=(a\Cdot\lambda,\alpha,\beta;b,
 a\Cdot\lambda,\alpha,\beta;n+m\Cdot\lambda,\alpha,\beta;b+\psi_{\lambda;\,\alpha,\,\beta} (a,b)), \\
 \label{eq:dif}P^\chi_\omega(a,m)&:=(P_\omega(a),\chi_\omega(a)+P_{M,\,\omega}(m)).
 \end{align}
 Moreover, we get an abelian extension
\[0\to (M, (P_{M,\,\omega})_{\omega\in\Omega})\xrightarrow{\left(\begin{array}{cc} s_\omega& i_\omega\end{array}\right) } (A\oplus M, P^\chi_\omega)\xrightarrow{\left(\begin{array}{c} p_\omega\\ t_\omega\end{array}\right)} (A, (P_\omega)_{\omega\in\Omega})\to 0\]
 which is easily seen to be  isomorphic to the original one \eqref{Eq: abelian extension}.


 Now we investigate the influence of different choices of   sections.

 \begin{prop}\label{prop: different sections give}
 Let $\Omega$ be a semigroup.
 \begin{itemize}
 \item[(i)] Different choices of the section $(s_\omega)_{\omega\in\Omega}$ give the same  Rota-Baxter family $\Omega$-associative conformal bimodule structures on $(M, (P_{M,\,\omega})_{\omega\in\Omega})$.

 \item[(ii)] The cohomological class of $((\psi_{\lambda;\,\alpha,\,\beta})_{\alpha,\,\beta\in\Omega},(\chi_\omega)_{\omega\in\Omega})$ does not depend on the choice of sections.
 \end{itemize}

 \end{prop}
 \begin{proof}Let $s_\omega^1$ and $s_\omega^2$ be two distinct sections of $(p_\omega)_{\omega\in\Omega}$.
  We define $\gamma_\omega:A\rar M$ by $\gamma_\omega(a)=s_\omega^1(a)-s_\omega^2(a)$.
Since the Rota-Baxter family $\Omega$-associative conformal algebra $(M, (P_{M,\,\omega})_{\omega\in\Omega})$	satisfies  $u\Cdot\lambda,\alpha,\beta;v=0$ for all $u,v\in M$ and $\alpha,\,\beta\in\Omega,$ we have
 \[s_\omega^1(a)\Cdot\lambda,\alpha,\beta;m= s_\omega^2(a)\Cdot\lambda,\alpha,\beta;m+\gamma_\omega(a)\Cdot\lambda,\alpha,\beta;m
 =s_\omega^2(a)\Cdot\lambda,\alpha,\beta;m.\]
So different choices of the section $(s_\omega)_{\omega\in\Omega}$ give the same  Rota-Baxter family $\Omega$-associative conformal bimodule structures on $(M, (P_{M,\,\omega})_{\omega\in\Omega})$.
We show that the cohomological class of $((\psi_{\lambda;\,\alpha,\,\beta})_{\alpha,\,\beta\in\Omega},$ $(\chi_\omega)_{\omega\in\Omega})$ does not depend on the choice of sections.   Then
 	\begin{align*}
 		&\psi_{\lambda;\,\alpha,\,\beta}^1(a,b)\\
 &=s_\alpha^1(a)\Cdot\lambda,\alpha,\beta;s_\beta^1(b)
 -s_{\alpha\beta}^1(a\Cdot\lambda,\alpha,\beta;b)\\
 		&=(s_\alpha^2(a)+\gamma_\alpha(a))\cdot_{\lambda;\,\alpha,\,\beta}(s_\beta^2(b)+\gamma_\beta(b))
 -(s_{\alpha\beta}^2(a\Cdot\lambda,\alpha,\beta;b)
 +\gamma_{\alpha\beta}(a\Cdot\lambda,\alpha,\beta;b))\\
 		&=(s_\alpha^2(a)\Cdot\lambda,\alpha,\beta;s_\beta^2(b)-s_{\alpha\beta}^2(a\Cdot\lambda,\alpha,\beta;b))
 +s_\alpha^2(a)\Cdot\lambda,\alpha,\beta;\gamma_\beta(b)+\gamma_\alpha(a)\Cdot\lambda,\alpha,\beta;s_\beta^2(b)
 -\gamma_{\alpha\beta}(a\Cdot\lambda,\alpha,\beta;b)\\
 		&=(s_\alpha^2(a)\Cdot\lambda,\alpha,\beta;s_\beta^2(b)-s_{\alpha\beta}^2(a\Cdot\lambda,\alpha,\beta;b))
 +a\Cdot\lambda,\alpha,\beta;\gamma_\beta(b)+\gamma_\alpha(a)\Cdot\lambda,\alpha,\beta;b
 -\gamma_{\alpha\beta}(a\Cdot\lambda,\alpha,\beta;b)\\
 		&=\psi^2_{\lambda;\,\alpha,\,\beta}(a,b)+\delta^1(\gamma)_{\lambda;\,\alpha,\,\beta}(a,b),
 	\end{align*}
 	and
 	\begin{align*}
 		\chi^1_\omega(a)&={\hat{P}_\omega}(s_\omega^1(a))-s_\omega^1(P_\omega(a))\\
 		&={\hat{P}_\omega}(s_\omega^2(a)+\gamma_\omega(a))-(s_\omega^2(P_\omega(a))
 +\gamma_\omega(P_\omega(a)))\\
 		&=({\hat{P}_\omega}(s_\omega^2(a))-s_\omega^2(P_\omega(a)))+{\hat{P}_\omega}
 (\gamma_\omega(a))-\gamma_\omega(P_\omega(a))\\
 		&=\chi^2_\omega(a)+P_{M,\omega}(\gamma_\omega(a))-\gamma_\omega(P_{A,\omega}(a))\\
 		&=\chi^2_\omega(a)-\Phi^1(\gamma)_\omega(a).
 	\end{align*}
 So by Eq.~(\ref{eq:diff}), we have
 \[(\psi^1_{\lambda;\,\alpha,\,\beta},\chi^1_\omega)-(\psi^2_{\lambda;\,\alpha,\,\beta},\chi^2_\omega)
 =(\delta^1(\gamma)_{\lambda;\,\alpha,\,\beta},-\Phi^1(\gamma)_\omega)
 =d^1(\gamma)_{\lambda;\,\alpha,\,\beta,\,\omega}.\]
Hence $((\psi^1_{\lambda;\,\alpha,\,\beta})_{\alpha,\,\beta\in\Omega}, (\chi^1_\omega)_{\omega\in\Omega})$ and $((\psi^2_{\lambda;\,\alpha,\,\beta})_{\alpha,\,\beta\in\Omega}, (\chi^2_\omega)_{\omega\in\Omega})$ form the same cohomological class  {in $\rmH_{\RBA}^2(A,M)$}.
 \end{proof}

 We show now the isomorphic abelian extensions give rise to the same cohomology classes.
 \begin{prop}
 Let $(A,(P_\omega)_{\omega\in\Omega})$ be a  Rota-Baxter family $\Omega$-associative conformal algebra.
 Two isomorphic abelian extensions of Rota-Baxter family $\Omega$-associative conformal algebra $(A, (P_\omega)_{\omega\in\Omega})$ by  $(M, (P_{M,\,\omega})_{\omega\in\Omega})$  give rise to the same cohomology class  in $\rmH_{\RBA}^2(A,M)$.
 \end{prop}
 \begin{proof}
  Assume that $(\hat{A}_1,(\hat{P}^1_\omega)_{\omega\in\Omega})$ and $(\hat{A}_2,(\hat{P}^2_\omega)_{\omega\in\Omega})$ are two isomorphic abelian extensions of $(A,(P_\omega)_{\omega\in\Omega})$ by $(M,(P_{M,\,\omega})_{\omega\in\Omega})$ as is given in \eqref{Eq: isom of abelian extension}. Let $s_\omega^1$ be a section of $(\hat{A}_1,(\hat{P}^1_\omega)_{\omega\in\Omega})$. As $p^2_\omega\circ\zeta_\omega=p^1_\omega$, we have
 	\[p^2_\omega\circ(\zeta_\omega\circ s_\omega^1)=p^1_\omega\circ s_\omega^1=\Id_{A,\,\omega}.\]
 	Therefore, $\zeta_\omega\circ s_\omega^1$ is a section of $(\hat{A}_2,(\hat{P}^1_\omega)_{\omega\in\Omega})$. Denote $s_\omega^2:=\zeta_\omega\circ s_\omega^1$. Since $\zeta_\omega$ is a homomorphism of Rota-Baxter family $\Omega$-associative conformal algebras such that
 \[\zeta_\omega|_M=\Id_{M,\,\omega},\quad \zeta_\omega(a\Cdot\lambda,\alpha,\beta;m)=\zeta_\omega(s_\omega^1(a)\Cdot\lambda,\alpha,\beta;m)
 =s_\omega^2(a)\Cdot\lambda,\alpha,\beta;m=a\Cdot\lambda,\alpha,\beta;m,\]
  so $\zeta_\omega|_M: M\to M$ is compatible with the induced  Rota-Baxter family $\Omega$-associative conformal bimodule structures.
 We have
 	\begin{align*}
 		\psi^2_{\lambda;\,\alpha,\,\beta}(a, b)&=s_\alpha^2(a)\Cdot\lambda,\alpha,\beta;s_\beta^2(b)-s_{\alpha\beta}^2(a\Cdot\lambda,\alpha,\beta;b)
 =\zeta_\alpha(s_\alpha^1(a))\Cdot\lambda,\alpha,\beta;\zeta_\beta(s_\beta^1(b))
 -\zeta_{\alpha\beta}(s_{\alpha\beta}^1(a\Cdot\lambda,\alpha,\beta;b))\\
 		&=\zeta_{\alpha\beta}(s_\alpha^1(a)\Cdot\lambda,\alpha,\beta;s_\beta^1(b)
 -s_{\alpha\beta}^1(a\Cdot\lambda,\alpha,\beta;b))
 =\zeta_{\alpha\beta}(\psi^1_{\lambda;\,\alpha,\,\beta}(a,b))\\
 		&=\psi^1_{\lambda;\,\alpha,\,\beta}(a,b),
 	\end{align*}
 	and
 	\begin{align*}
 		\chi^2_\omega(a)&=\hat{P}^2_\omega(s_\omega^2(a))-s_\omega^2(P_\omega(a))
 =\hat{P}^2_\omega(\zeta_\omega(s_\omega^1(a)))-\zeta_\omega(s_\omega^1(P_\omega(a)))\\
 		&=\zeta_\omega(\hat{P}^1_\omega(s_\omega^1(a))-s_\omega^1(P_\omega(a)))
 =\zeta_\omega(\chi^1_\omega(a))\\
 		&=\chi^1_\omega(a).
 	\end{align*}
 	So two isomorphic abelian extensions give rise to the same element in {$\rmH_{\RBA}^2(A,M)$}.
\end{proof}

 Now let us consider the reverse direction.
Given two families of conformal linear maps  $(\psi_{\lambda;\,\alpha,\,\beta})_{\alpha,\,\beta\in\Omega}:A\ot A\rar M$ and $(\chi_\omega)_{\omega\in\Omega}:A\rar M$, one can define  a family of multiplication operations $(\cdot^\psi_{\lambda;\,\alpha,\,\beta})_{\alpha,\,\beta\in\Omega}$ and a family of $(P^\chi_\omega)_{\omega\in\Omega}$ on $A\oplus M$ by Eqs.~(\ref{eq:mul}-\ref{eq:dif}).
 The following fact is important:
 \begin{prop}
 	The triple $(A\oplus M, (\cdot^\psi_{\lambda;\,\alpha,\,\beta})_{\alpha,\,\beta\in\Omega},(P^\chi_\omega)_{\omega\in\Omega})$ is a Rota-Baxter family $\Omega$-associative conformal algebra   if and only if
 	$((\psi_{\lambda;\,\alpha,\,\beta})_{\alpha,\,\beta\in\Omega},(\chi_{\omega})_{\omega\in\Omega})$ is a 2-cocycle  of the Rota-Baxter family $\Omega$-associative conformal algebra $(A,(P_\omega)_{\omega\in\Omega})$ with  coefficients  in $(M,(P_{M,\omega})_{\omega\in\Omega})$.
 \end{prop}
 \begin{proof}
 	If $(A\oplus M, (\cdot^\psi_{\lambda;\,\alpha,\,\beta})_{\alpha,\,\beta\in\Omega}, (P^\chi_\omega)_{\omega\in\Omega})$ is a Rota-Baxter family $\Omega$-associative conformal algebra, then the associativity of $(\cdot^\psi_{\lambda;\,\alpha,\,\beta})_{\alpha,\,\beta\in\Omega}$ implies
 	\begin{equation*}
 a\cdot^\psi_{\lambda;\,\alpha,\,\beta\gamma}\psi_{\mu;\,\beta,\,\gamma}(b, c)-\psi_{\lambda+\mu;\,\alpha\beta,\,\gamma}
 (a\cdot^\psi_{\lambda;\,\alpha,\,\beta}b, c)+\psi_{\lambda;\,\alpha,\,\beta\gamma}(a, b\cdot^\psi_{\lambda;\,\alpha,\,\beta}c)-\psi_{\lambda;\,\alpha,\,\beta}(a, b)\cdot^\psi_{\lambda+\mu;\,\alpha\beta,\,\gamma}c=0,
 	\end{equation*}
 	which means $\delta^2(\psi)_{\lambda,\,\mu;\,\alpha,\,\beta,\,\gamma}=0$ in $C^\bullet(A,M)$.
 	Since $(P^\chi_\omega)_{\omega\in\Omega}$ is a Rota-Baxter family operator,
 	for any $a,b\in A, m,n\in M$, we have
 \begin{align*}	
 &P^\chi_{\alpha}((a,m))\cdot^\psi_{\lambda;\,\alpha,\,\beta} P^\chi_{\beta}((b,n))\\
 ={}&P^\chi_{\alpha\beta}
 \Big(P^\chi_{\alpha}(a,m)\cdot^\psi_{\lambda;\,\alpha,\,\beta}(b,n)+(a,m)
 \cdot^\psi_{\lambda;\,\alpha,\,\beta} P^\chi_{\beta}(b,n)+q(a,m)\cdot^\psi_{\lambda;\,\alpha,\,\beta}(b,n)\Big).
 \end{align*}
 Then $(\chi_\omega)_{\omega\in\Omega}, (\psi_{\lambda;\,\alpha,\,\beta})_{\alpha,\,\beta\in\Omega}$ satisfy the following equations:
 	\begin{align*}
 		&P_\alpha(a)\cdot^\psi_{\lambda;\,\alpha,\,\beta}\chi_\beta(b)
 +\chi_\alpha(a)\cdot^\psi_{\lambda;\,\alpha,\,\beta}
 P_\beta(b)+\psi_{\lambda;\,\alpha,\,\beta}(P_\alpha(a), P_\beta(b))\\
 		={}&P_{M,\alpha\beta}(\chi_\alpha(a)\cdot^\psi_{\lambda;\,\alpha,\,\beta}b)
 +P_{M,\alpha\beta}(\psi_{\lambda;\,\alpha,\,\beta}(P_\alpha(a), b))+\chi_{\alpha\beta}(P_\alpha(a)\cdot^\psi_{\lambda;\,\alpha,\,\beta}b)\\
 		&+P_{M,\alpha\beta}(a\cdot^\psi_{\lambda;\,\alpha,\,\beta}\chi_\beta(b))
 +P_{M,\alpha\beta}(\psi_{\lambda;\,\alpha,\,\beta}(a, P_\beta(b)))+\chi_{\alpha\beta}(a\cdot^\psi_{\lambda;\,\alpha,\,\beta}P_\beta(b))\\
 		&+q P_{M,\alpha\beta}(\psi_{\lambda;\,\alpha,\,\beta}(a, b))+q \chi_{\alpha\beta}(a\cdot^\psi_{\lambda;\,\alpha,\,\beta}b).
 	\end{align*}
 	
 	That is,
 	\[ \partial^1(\chi)_{\lambda;\,\alpha,\,\beta}+\Phi^2(\psi)_{\lambda;\,\alpha,\,\beta}=0.\]
 	Hence, $((\psi_{\lambda;\,\alpha,\,\beta})_{\alpha,\,\beta\in\Omega},(\chi_{\omega})_{\omega\in\Omega})$ is a  2-cocycle.

 	 Conversely, if $((\psi_{\lambda;\,\alpha,\,\beta})_{\alpha,\,\beta\in\Omega},(\chi_{\omega})_{\omega\in\Omega})$ is a 2-cocycle, one can easily check that $(A\oplus M, (\cdot^\psi_{\lambda;\,\alpha)_{\alpha,\,\beta\in\Omega},\,\beta},$ $(P^\chi_\omega)_{\omega\in\Omega})$ is a Rota-Baxter family $\Omega$-associative conformal algebra.
 \end{proof}

 Finally, we show the following result:
 \begin{prop}
 	Two cohomologous $2$-cocyles give rise to isomorphic abelian extensions.
 \end{prop}

 \begin{proof}
Given two 2-cocycles $((\psi^1_{\lambda;\,\alpha,\,\beta})_{\alpha,\,\beta\in\Omega}, (\chi^1_\omega)_{\omega\in\Omega})$ and $((\psi^2_{\lambda;\,\alpha,\,\beta})_{\alpha,\,\beta\in\Omega}, (\chi^2_\omega)_{\omega\in\Omega})$, we can construct two abelian extensions $(A\oplus M, (\cdot^{\psi_{1}}_{\lambda;\,\alpha,\,\beta})_{\alpha,\,\beta\in\Omega}, (P^{\chi^1}_{\omega})_{\omega\in\Omega})$ and  $(A\oplus M, (\cdot^{\psi_{2}}_{\lambda;\,\alpha,\,\beta})_{\alpha,\,\beta\in\Omega}, (P^{\chi^2}_{\omega})_{\omega\in\Omega})$ via Eqs.~\eqref{eq:mul} and \eqref{eq:dif}. If they represent the same cohomology  class {in $\rmH_{\RBA}^2(A,M)$}, then there exists two linear maps $\gamma_\omega^0:k\rightarrow M, \gamma_\omega^1:A\to M$ such that $$(\psi^1_{\lambda;\,\alpha,\,\beta},\chi^1_\omega)=(\psi^2_{\lambda;\,\alpha,\,\beta},\chi^2_\omega)
 +(\delta^1(\gamma^1)_{\lambda;\,\alpha,\,\beta},-\Phi^1(\gamma^1)_\omega
 -\partial^0(\gamma^0)_\omega).$$
 	Notice that $\partial^0_\omega=\Phi^1_\omega\circ\delta^0_\omega$. Define $\gamma_\omega: A\rightarrow M$ to be $\gamma^1_\omega+\delta^0(\gamma^0)_\omega$. Then $(\gamma_\omega)_{\omega\in\Omega}$ satisfies
 	\[(\psi^1_{\lambda;\,\alpha,\,\beta},\chi^1_\omega)=(\psi^2_{\lambda;\,\alpha,\,\beta},\chi^2_\omega)
 +\left(\delta^1(\gamma)_{\lambda;\,\alpha,\,\beta},-\Phi^1(\gamma)_\omega\right).\]
 	Define $(\zeta_\omega)_{\omega\in\Omega}:A\oplus M\rar A\oplus M$ by
 	\[\zeta_\omega(a,m):=(a, -\gamma_\omega(a)+m).\]
 	Then $\zeta_\omega$ is an isomorphism of these two abelian extensions $(A\oplus M, (\cdot^{\psi_{1}}_{\lambda;\,\alpha,\,\beta})_{\alpha,\,\beta\in\Omega}, (P^{\chi^1}_{\omega})_{\omega\in\Omega})$ and  $(A\oplus M, (\cdot^{\psi_{2}}_{\lambda;\,\alpha,\,\beta})_{\alpha,\,\beta\in\Omega}, (P^{\chi^2}_{\omega})_{\omega\in\Omega})$.
 \end{proof}

\section{Appendix A: Proof of Proposition~\ref{Prop: Chain map Phi} }
\label{Appendix}

\begin{proof}
	$ (i) $ For $ n=0 $, we need to check the following diagram is commutative.
	\[\xymatrix{
		\C^0_{\Alg}(A,M)\ar[r]^-{\delta^0}\ar[d]^-{\Phi^0}& \C^1_{\Alg}(A,M)\ar[d]^{\Phi^1} \\
		\C^0_{\RBO}(A,M)\ar[r]^-{\partial^0}&\C^1_{\RBO}(A,M).
	}\]
	
	For $m\in \C_{\Alg}^{0} (A,M), a\in A$,
	\begin{eqnarray*}
		\Phi^{1} \Big(\delta^{0}(m)\Big)_\alpha(a)
		&=& \delta^{0}(m)_\alpha(P_\alpha(a))-P_{M,\,\alpha}  (\delta^{0}(m)_\alpha(a) )\\
		&=& P_\alpha(a) \cdot_{0;\,\alpha,\,1} m - m \cdot_{-\partial;\,1,\,\alpha} P_\alpha(a) - P_{M,\,\alpha}( a \cdot_{0;\,\alpha,\,1} m - m \cdot_{-\partial;\,1,\,\alpha} a)\\
		&=& \partial^{0}(m)_\alpha(a) \\
		&=& \partial^{0} \Big(\Phi^{0}(m)\Big)_\alpha(a).
	\end{eqnarray*}
	
	So
	\[ \partial^{0} \circ \Phi^{0} = \Phi^{1} \circ \delta^{0}.\]

	$ (ii) $ For the general case, we need to prove that the square below is commutative, for all $ n \geqslant 1 $.
	\[\xymatrix{
		\C^n_{\Alg}(A,M)\ar[r]^-{\delta^n}\ar[d]^-{\Phi^{n}}& \C^{n+1}_{\Alg}(A,M)\ar[d]^{\Phi^{n+1}} \\
		\C^n_{\RBO}(A,M)\ar[r]^-{\partial^n}&\C^{n+1}_{\RBO}(A,M).
	}\]
	
	For $f \in \C_{\Alg}^{n}(A,M) , a_{1}, \dots , a_{n+1} \in A$, On the one hand,
		\begin{eqnarray*}
			&&\Phi^{n+1} \circ \delta^{n}(f)_{\lambda_1,\,\dots,\,\lambda_n;\,\alpha_1,\dots,\alpha_{n+1}}(a_{1},\dots,a_{n+1})\\
			&=& \delta^{n}(f)_{\lambda_1,\,\dots,\,\lambda_n;\,\alpha_1,\dots,\alpha_{n+1}} \left( P_{\alpha_1}(a_{1}),\dots,P_{\alpha_{n+1}}(a_{n+1})\right)\\
			&& -\sum_{k=0}^{n}  q ^{n-k}
			\sum_{1\leq i_{1}<\dots<i_{k}\leq n+1}
			P_{M,\,\alpha_1\dots\alpha_{n+1}}
			\left(\delta^{n}(f)_{\lambda_1,\,\dots,\,\lambda_n;\,\alpha_1,\dots,\alpha_{n+1}}
			\left(a_{1},\dots,P_{\alpha_{i_1}}(a_{i_{1}}),\dots,
P_{\alpha_{i_k}}(a_{i_{k}}),\dots,a_{n+1}\right)\right)  \\
			&=&  P_{\alpha_1}(a_{1}) \cdot_{\lambda_1;\,\alpha_1,\,\alpha_2\dots\alpha_{n+1}} f_{\lambda_2,\,\dots,\,\lambda_n;\,\alpha_2,\dots,\alpha_{n+1}} \left( P_{\alpha_2}(a_{2}),\dots,P_{\alpha_{n+1}}(a_{n+1})\right)\\
			&& +\sum_{i=1}^{n}(-1)^{i}f_{\lambda_1,\,\dots,\,\lambda_n;\,\alpha_1,\dots,\alpha_{n+1}}
\left(P_{\alpha_1}(a_{1}),\dots,\uwave{P_{\alpha_i}(a_{i})\cdot_{\lambda_i;\,\alpha_i,\alpha_{i+1}}
P_{\alpha_{i+1}}(a_{i+1})},\dots,P_{\alpha_{n+1}}(a_{n+1})\right)\\
			&& + (-1)^{n+1}f_{\lambda_1,\,\dots,\,\lambda_n;\,\alpha_1,\dots,\alpha_{n+1}} \left( P_{\alpha_1}(a_{1}),\dots,P_{\alpha_n}(a_{n})\right)
\cdot_{\lambda_1+\dots+\lambda_n;\,\alpha_1\dots\alpha_n,\,\alpha_{n+1}}
P_{\alpha_{n+1}}(a_{n+1})\\
			&& -\sum_{k=0}^{n}  q ^{n-k}
			\sum_{2\leq i_{1}<\dots<i_{k}\leq n+1}
			P_{M,\,\alpha_1\dots\alpha_{n+1}} \Big(a_{1}
\cdot_{\lambda_1;\,\alpha_1,\,\alpha_2\dots\alpha_{n+1}} f_{\lambda_2,\,\dots,\,\lambda_n;\,\alpha_2,\dots,\alpha_{n+1}}\\
&&\hspace{5cm}\left(a_{2},\dots,
P_{\alpha_{i_1}}(a_{i_{1}}),\dots,P_{\alpha_{i_k}}(a_{i_{k}}),\dots,a_{n+1}
\right)\Big)\\
			&& -\sum_{k=1}^{n}  q ^{n-k}
			\sum_{2\leq i_{1}<\dots<i_{k-1}\leq n+1}
			P_{M,\,\alpha_1\dots\alpha_{n+1}}
 \Big(P_{\alpha_1}(a_{1})
 \cdot_{\lambda_1;\,\alpha_1,\,\alpha_2\dots\alpha_{n+1}}
  f_{\lambda_2,\dots,\lambda_n;\,\alpha_2,\dots,\alpha_{n+1}}\\
&&\hspace{5cm}\left(a_{2},\dots,
P_{\alpha_{i_1}}(a_{i_{1}}),\dots,P_{\alpha_{i_{k-1}}}
(a_{i_{k-1}}),\dots,a_{n+1}\right)\Big)\\
			&& -\sum_{k=0}^{n-1}  q ^{n-k}
			\sum_{1\leq i_{1}<\dots<i_{k}\leq n+1}
			\sum_{i=1}^{n}(-1)^{i}
			P_{M,\alpha_1\dots\alpha_{n+1}}   \Big(f_{\lambda_1,\,\dots,\lambda_n;\,\alpha_1,\dots,\alpha_{n+1}}\\
&&\hspace{5cm}
\left(a_{1},\dots,P_{\alpha_{i_1}}(a_{i_{1}}),\dots,a_{i}\cdot_{\lambda_i;\,\alpha_i,\alpha_{i+1}}
a_{i+1},\dots,P_{\alpha_{i_k}}(a_{i_{k}}),\dots,a_{n+1}\right)\Big)\\
			&& -\sum_{k=1}^{n}  q ^{n-k}
			\sum_{1\leq i_{1}<\dots<i_{k}\leq n+1}
			\sum_{i=1}^{n}(-1)^{i}
			P_{M,\,\alpha_1\dots\alpha_{n+1}}
			\Big(f_{\lambda_1,\,\dots,\lambda_i+\lambda_{i+1},\dots,\lambda_n;\,\alpha_1,\dots,\alpha_i\alpha_{i+1},
\dots,\alpha_{n+1}}\\
&&\hspace{5cm}
\left(a_{1},\dots,P_{\alpha_{i_1}}(a_{i_{1}}),\dots,P_{\alpha_i}(a_{i})
\cdot_{\lambda_i;\,\alpha_i,\alpha_{i+1}}a_{i+1},\dots,P_{\alpha_{i_k}}(a_{i_{k}}),\dots,a_{n+1}\right)\Big)   \\
			&& -\sum_{k=1}^{n}  q ^{n-k}
			\sum_{1\leq i_{1}<\dots<i_{k}\leq n+1}
			\sum_{i=1}^{n}(-1)^{i}
			P_{M,\,\alpha_1\dots\alpha_{n+1}}
			\Big(f_{\lambda_1,\,\dots,\lambda_i+\lambda_{i+1},\dots,\lambda_n;\,\alpha_1,\dots,\alpha_i\alpha_{i+1},
\dots,\alpha_{n+1}}\\
&&\hspace{5cm}
\left(a_{1},\dots,P_{\alpha_{i_1}}(a_{i_{1}}),\dots,
a_{i}\cdot_{\lambda_i;\,\alpha_i,\,\alpha_{i+1}}P_{\alpha_{i+1}}(a_{i+1}),\dots,
P_{\alpha_{i_k}}(a_{i_{k}}),\dots,a_{n+1}\right)\Big)   \\
			&& -\sum_{k=2}^{n}  q ^{n-k}
			\sum_{1\leq i_{1}<\dots<i_{k}\leq n+1}
			\sum_{i=1}^{n}(-1)^{i}
			P_{M,\,\alpha_1\dots\alpha_{n+1}}
			\Bigg(f_{\lambda_1,\,\dots,\lambda_i+\lambda_{i+1},\dots,\lambda_n;\,\alpha_1,\dots,\alpha_i\alpha_{i+1},
\dots,\alpha_{n+1}}\\
&&\hspace{3cm}
\left(a_{1},\dots,P_{\alpha_{i_1}}(a_{i_{1}}),\dots,
\uwave{P_{\alpha_i}(a_{i})\cdot_{\lambda_i;\,\alpha_i,\alpha_{i+1}}
P_{\alpha_{i+1}}(a_{i+1})},\dots,P_{\alpha_{i_k}}(a_{i_{k}}),\dots,a_{n+1}\right)\Bigg)   \\
			&& -(-1)^{n+1}\sum_{k=0}^{n}  q ^{n-k}
			\sum_{1\leq i_{1}<\dots<i_{k}\leq n}
			P_{M,\,\alpha_1\dots\alpha_{n+1}} \Big(f_{\lambda_1,\,\dots,,\dots,\lambda_{n-1};\,\alpha_1,\dots,
\dots,\alpha_{n}}\\
&&\hspace{5cm}
\left(a_{1},\dots,P_{\alpha_{i_1}}(a_{i_{1}}),\dots,P_{\alpha_{i_k}}(a_{i_{k}}),\dots,a_{n} \right) \cdot_{\lambda_1+\dots+\lambda_n;\,\alpha_1\dots\alpha_n,\alpha_{n+1}} a_{n+1}
			\Big)  \\
			&& -(-1)^{n+1}\sum_{k=1}^{n}  q ^{n-k}
			\sum_{1\leq i_{1}<\dots<i_{k-1}\leq n}
			P_{M,\,\alpha_1\dots\alpha_{n+1}}  \Big(f_{\lambda_1,\dots,\lambda_{n-1};\,\alpha_1,\dots,\alpha_n}\\
&&\hspace{3cm}
\left(a_{1},\dots,P_{\alpha_{i_1}}(a_{i_{1}}),\dots,P_{\alpha_{k-1}}(a_{i_{k-1}}),\dots,a_{n} \right) \cdot_{\lambda_1+\dots+\lambda_n;\,\alpha_1\dots\alpha_n,\alpha_{n+1}} P_{\alpha_{n+1}}(a_{n+1})
			\Big)  \\
			&=& \uline{ P_{\alpha_1}(a_{1})
 \cdot_{\lambda_1;\,\alpha_1,\alpha_2\dots\alpha_{n+1}}
  f_{\lambda_2,\dots,\lambda_n;\,\alpha_2,\dots,\alpha_{n+1}}
   \left( P_{\alpha_2}(a_{2}),\dots,P_{\alpha_{n+1}}(a_{n+1})\right) }_{(1)}\\
			&& \uuline{ +\sum_{i=1}^{n}(-1)^{i}
f_{\lambda_1,\dots,\lambda_i+\lambda_{i+1},\dots,\lambda_n;\,\alpha_1,\dots,
\alpha_i\alpha_{i+1},\dots,\alpha_{n+1}}}_{(2)}\\
&&\hspace{3cm}
\uuline{\left( P_{\alpha_1}(a_{1}),\dots,{P_{\alpha_i\alpha_{i+1}}(a_{i}\cdot_{\lambda_i;\,\alpha_i,\alpha_{i+1}}
P_{\alpha_{i+1}}(a_{i+1}))},\dots,P_{\alpha_{n+1}}(a_{n+1})\right)}_{(2)}\\
			&& \uwave{ +\sum_{i=1}^{n}(-1)^{i}
f_{\lambda_1,\dots,\lambda_i+\lambda_{i+1},\dots,\lambda_n;\,\alpha_1,\dots,\alpha_i\alpha_{i+1},
\dots,\alpha_{n+1}}
\left( P_{\alpha_1}(a_{1}),\dots,{P_{\alpha_i\alpha_{i+1}}(P_{\alpha_i}(a_{i})
\cdot_{\lambda_i;\,\alpha_i,\alpha_{i+1}}a_{i+1})},\dots,P_{\alpha_{n+1}}(a_{n+1})\right) }_{(3)}\\
			&& \dashuline{ +\sum_{i=1}^{n}(-1)^{i}  q
f_{\lambda_1,\dots,\lambda_i+\lambda_{i+1},\dots,\lambda_n;\,\alpha_1,\dots,\alpha_i\alpha_{i+1},
\dots,\alpha_{n+1}}
 \left( P_{\alpha_1}(a_{1}),\dots,{P_{\alpha_i\alpha_{i+1}}(a_{i}\cdot_{\lambda_i;\,\alpha_i,\alpha_{i+1}}a_{i+1})},
 \dots,P_{\alpha_{n+1}}(a_{n+1})\right) }_{(4)}\\
			&& \dotuline{ +(-1)^{n+1}
f_{\lambda_1,\dots,\lambda_{n-1};\,\alpha_1,\dots,\alpha_n}
\left( P_{\alpha_1}(a_{1}),\dots,P_{\alpha_n}(a_{n})\right) \cdot_{\lambda_1+\dots+\lambda_n;\,\alpha_1\dots\alpha_n,\alpha_{n+1}} P_{\alpha_{n+1}}(a_{n+1}) }_{(5)}\\
			&& \uline{ -\sum_{k=0}^{n}  q ^{n-k}
				\sum_{2\leq i_{1}<\dots<i_{k}\leq n+1}
				P_{M,\,\alpha_1\dots\alpha_{n+1}}
\Big(a_{1} \cdot_{\lambda_1;\,\alpha_1,\alpha_2\dots\alpha_{n+1}}
f_{\lambda_2,\dots,\lambda_n;\,\alpha_2,\dots,\alpha_{n+1}}}_{(6)}\\
&&\hspace{5cm}
\uline{\left(a_{2},\dots,P_{\alpha_{i_1}}(a_{i_{1}}),\dots,P_{\alpha_{i_k}}(a_{i_{k}}),\dots,a_{n+1}
\right)\Big) }_{(6)}\\
			&& \uuline{ -\sum_{k=1}^{n}  q ^{n-k}
				\sum_{2\leq i_{1}<\dots<i_{k-1}\leq n+1}
				P_{M,\,\alpha_1\dots\alpha_{n+1}}
\Big(P_{\alpha_1}(a_{1})
\cdot_{\lambda_1;\,\alpha_1,\,\alpha_2\dots\alpha_{n+1}}
f_{\lambda_2,\dots,\lambda_n;\,\alpha_2,\dots,\alpha_{n+1}}}_{(7)}\\
&&\hspace{5cm}
\uuline{
\left(a_{2},\dots,P_{\alpha_{i_1}}(a_{i_{1}}),\dots,P_{\alpha_{i_{k-1}}}(a_{i_{k-1}}),
\dots,a_{n+1}\right)\Big) }_{(7)}\\
			&& \uwave{ -\sum_{k=0}^{n-1}  q ^{n-k}
				\sum_{1\leq i_{1}<\dots<i_{k}\leq n+1}
				\sum_{i=1}^{n}(-1)^{i}
				P_{M,\,\alpha_1\dots\alpha_{n+1}}
				\Big(
f_{\lambda_1,\dots,\lambda_i+\lambda_{i+1},\dots,\lambda_n;\,\alpha_1,
\dots,\alpha_i\alpha_{i+1},\dots,\alpha_{n+1}}}_{(8)}\\
&&\hspace{5cm}
\uwave{\left(a_{1},\dots,P_{\alpha_{i_1}}(a_{i_{1}}),\dots,a_{i}\cdot_{\lambda_i;\,\alpha_i,\alpha_{i+1}}
a_{i+1},\dots,P_{\alpha_{i_k}}(a_{i_{k}}),\dots,a_{n+1}\right)\Big)}_{(8)} \\
			&& \dashuline{ -\sum_{k=1}^{n}  q ^{n-k}
				\sum_{1\leq i_{1}<\dots<i_{k}\leq n+1}
				\sum_{i=1}^{n}(-1)^{i}
				P_{M,\,\alpha_1\dots\alpha_{n+1}}
				\Big(
f_{\lambda_1,\dots,\lambda_i+\lambda_{i+1},\dots,\lambda_n;\,\alpha_1,
\dots,\alpha_i\alpha_{i+1},\dots,\alpha_{n+1}}}_{(9)}\\
&&\hspace{3cm}
\uwave{\left(a_{1},\dots,P_{\alpha_{i_1}}(a_{i_{1}}),\dots,
P_{\alpha_i}(a_{i})\cdot_{\lambda_i;\,\alpha_i,\alpha_{i+1}}a_{i+1},
\dots,P_{\alpha_{i_k}}(a_{i_{k}}),\dots,a_{n+1}\right)\Big) }_{(9)}   \\
			&& \dotuline{ -\sum_{k=1}^{n}  q ^{n-k}
				\sum_{1\leq i_{1}<\dots<i_{k}\leq n+1}
				\sum_{i=1}^{n}(-1)^{i}
				P_{M,\,\alpha_1\dots\alpha_{n+1}}
				\Big(
f_{\lambda_1,\dots,\lambda_i+\lambda_{i+1},\dots,\lambda_n;\,\alpha_1,
\dots,\alpha_i\alpha_{i+1},\dots,\alpha_{n+1}}}_{(10)}\\
&&\hspace{3cm}
\uwave{\left(a_{1},\dots,P_{\alpha_{i_1}}(a_{i_{1}}),\dots,
a_{i}\cdot_{\lambda_i;\,\alpha_i,\alpha_{i+1}}P_{\alpha_{i+1}}(a_{i+1}),\dots,
P_{\alpha_{i_k}}(a_{i_{k}}),\dots,a_{n+1}\right)\Big) }_{(10)}
			\\
			&& \uline{ -\sum_{k=2}^{n}  q ^{n-k}
				\sum_{1\leq i_{1}<\dots<i_{k}\leq n+1}
				\sum_{i=1}^{n}(-1)^{i}
				P_{M,\,\alpha_1\dots\alpha_{n+1}}
				\Big(
f_{\lambda_1,\dots,\lambda_i+\lambda_{i+1},\dots,\lambda_n;\,\alpha_1,
\dots,\alpha_i\alpha_{i+1},\dots,\alpha_{n+1}}}_{(11)}\\
&&\hspace{3cm}
\uwave{
\left(a_{1},\dots,P_{\alpha_{i_1}}(a_{i_{1}}),\dots,{P_{\alpha_i\alpha_{i+1}}
(a_{i}\cdot_{\lambda_i;\,\alpha_i,\alpha_{i+1}}P_{\alpha_{i+1}}(a_{i+1}))},\dots,
P_{\alpha_{i_k}}(a_{i_{k}}),\dots,a_{n+1}\right)\Big) }_{(11)}   \\
			&& \uuline{ -\sum_{k=2}^{n}  q ^{n-k}
				\sum_{1\leq i_{1}<\dots<i_{k}\leq n+1}
				\sum_{i=1}^{n}(-1)^{i}
				P_{M,\,\alpha_1\dots\alpha_{n+1}}
				\Big(
f_{\lambda_1,\dots,\lambda_i+\lambda_{i+1},\dots,\lambda_n;\,\alpha_1,
\dots,\alpha_i\alpha_{i+1},\dots,\alpha_{n+1}}}_{(12)}\\
&&\hspace{3cm}
\uwave{
\left(a_{1},\dots,P_{\alpha_{i_1}}(a_{i_{1}}),\dots,{P_{\alpha_i\alpha_{i+1}}}
(P_{\alpha_i}(a_{i})\cdot_{\lambda_i;\,\alpha_i,\alpha_{i+1}}a_{i+1}),\dots,P_{\alpha_{i_k}}
(a_{i_{k}}),\dots,a_{n+1}\right)\Big) }_{(12)}   \\
			&& \uwave{ -\sum_{k=2}^{n}  q ^{n-k+1}
				\sum_{1\leq i_{1}<\dots<i_{k}\leq n+1}
				\sum_{i=1}^{n}(-1)^{i}
				P_{M,\,\alpha_1\dots\alpha_{n+1}}
				\Big(
f_{\lambda_1,\dots,\lambda_i+\lambda_{i+1},\dots,\lambda_n;\,\alpha_1,
\dots,\alpha_i\alpha_{i+1},\dots,\alpha_{n+1}}}_{(13)}\\
&&\hspace{3cm}
\uwave{
\left(a_{1},\dots,P_{\alpha_{i_1}}(a_{i_{1}}),\dots,{P_{\alpha_i\alpha_{i+1}}
(a_{i}\cdot_{\lambda_i;\,\alpha_i,\alpha_{i+1}}a_{i+1})},
\dots,P_{\alpha_{i_k}}(a_{i_{k}}),\dots,a_{n+1}\right)\Big) }_{(13)}   \\
			&& \dashuline{ -(-1)^{n+1}\sum_{k=0}^{n}  q ^{n-k}
				\sum_{1\leq i_{1}<\dots<i_{k}\leq n}
				P_{M,\,\alpha_1\dots\alpha_{n+1}}
\Big(
f_{\lambda_1,\dots,\lambda_{n-1};\,\alpha_1,\dots,\alpha_n}}_{(14)}\\
&&\hspace{3cm}
\uwave{
\left(a_{1},\dots,P_{\alpha_{i_1}}(a_{i_{1}}),\dots,P_{\alpha_{i_k}}(a_{i_{k}}),\dots,a_{n} \right) \cdot_{\lambda_1+\dots+\lambda_n;\,\alpha_1\dots\alpha_n,\,\alpha_{n+1}}
 a_{n+1}
				\Big) }_{(14)}  \\
			&& \dotuline{ -(-1)^{n+1}\sum_{k=1}^{n}  q ^{n-k}
				\sum_{1\leq i_{1}<\dots<i_{k-1}\leq n}
				P_{M,\,\alpha_1\dots\alpha_{n+1}}
\Big(
f_{\lambda_1,\dots,\lambda_{n-1};\,\alpha_1,\dots,\alpha_n}}_{(15)}\\
&&\hspace{3cm}
\uwave{
\left(a_{1},\dots,P_{\alpha_{i_1}}(a_{i_{1}}),\dots,P_{\alpha_{i_{k-1}}}(a_{i_{k-1}}),\dots,a_{n} \right)
\cdot_{\lambda_1+\dots+\lambda_n;\,\alpha_1\dots\alpha_n,\,\alpha_{n+1}}
 P_{\alpha_{n+1}}(a_{n+1})
				\Big) }_{(15)}  .
		\end{eqnarray*}

	On the other hand,
		\begin{eqnarray*}
			&&\partial^{n} \circ \Phi^{n} (f)_{\lambda_1,\dots,\lambda_n;\,\alpha_1,\dots,\alpha_{n+1}}(a_{1},\dots,a_{n+1})\\
			&=&
			P_{\alpha_1}(a_{1}) \cdot_{\lambda_1;\,\alpha_1,\,\alpha_2\dots\alpha_{n+1}}
\Phi^{n}(f)_{\lambda_2,\dots,\lambda_n;\,\alpha_2,\dots,\alpha_{n+1}}
 (a_{2},\dots,a_{n+1})\\
			&&-
			P_{M,\,\alpha_1\dots\alpha_{n+1}}
\left( a_{1} \cdot_{\lambda_1;\,\alpha_1,\alpha_2\dots\alpha_{n+1}}
 \Phi^{n}(f)_{\lambda_2,\dots,\lambda_n;\,\alpha_2,\dots,\alpha_{n+1}} (a_{2},\dots,a_{n+1}) \right)\\
			&& +\sum_{i=1}^{n}(-1)^{i}
			\Phi^{n}(f)_{\lambda_1,\dots,\lambda_i+\lambda_{i+1},\dots,\lambda_n;\,
\alpha_1,\dots,\alpha_i\alpha_{i+1},\dots,\alpha_{n+1}}
\left(a_{1},\dots,a_{i} \cdot_{\lambda_i;\,\alpha_i,\,\alpha_{i+1}} P_{\alpha_{i+1}}(a_{i+1}),\dots,a_{n+1}\right)\\
			&& +\sum_{i=1}^{n}(-1)^{i}
			\Phi^{n}(f)_{\lambda_1,\dots,\lambda_i+\lambda_{i+1},\dots,\lambda_n;\,
\alpha_1,\dots,\alpha_i\alpha_{i+1},\dots,\alpha_{n+1}}
\left(a_{1},\dots,P_{\alpha_i}(a_{i})\cdot_{\lambda_i;\,\alpha_i,\,\alpha_{i+1}}a_{i+1},\dots,a_{n+1}\right)\\
			&& +\sum_{i=1}^{n}(-1)^{i}
			 q  \Phi^{n}(f)_{\lambda_1,\dots,\lambda_i+\lambda_{i+1},\dots,\lambda_n;\,
\alpha_1,\dots,\alpha_i\alpha_{i+1},\dots,\alpha_{n+1}}
\left(a_{1},\dots,a_{i}\cdot_{\lambda_i;\,\alpha_i,\,\alpha_{i+1}} a_{i+1},\dots,a_{n+1}\right)\\
			&& +(-1)^{n+1}
\Phi^{n}(f)_{\lambda_1,\dots,\dots,\lambda_{n-1};\,
\alpha_1,\dots,\alpha_{n}}
(a_{1},\dots,a_{n}) \cdot_{\lambda_1+\dots\lambda_n;\,\alpha_1\dots\alpha_n,\,\alpha_{n+1}} P_{\alpha_{n+1}}(a_{n+1})\\
			&& -(-1)^{n+1}P_{M,\,\alpha_1\dots\alpha_{n+1}}
\left(
\Phi^{n}(f)_{\lambda_1,\dots,\lambda_{n-1};\,
\alpha_1,\dots,\alpha_{n}}
(a_{1},\dots,a_{n}) \cdot_{\lambda_1+\dots\lambda_n;\,\alpha_1\dots\alpha_n,\,\alpha_{n+1}}
 a_{n+1}\right)\\
			&=&  P_{\alpha_1}(a_{1}) \cdot_{\lambda_1;\,\alpha_1,\alpha_2\dots\alpha_{n+1}}
 f_{\lambda_2,\dots,\lambda_n;\,\alpha_2,\dots,\alpha_{n+1}}
  \left(P_{\alpha_2}(a_{2}),\dots,P_{\alpha_{n+1}}(a_{n+1})\right)\\
			&& -
			\sum_{k=0}^{n-1}  q ^{n-k-1}
			\sum_{2\leq i_{1}<\dots<i_{k}\leq n+1}
			\uwave{P_{\alpha_1}(a_{1})
\cdot_{\lambda_1;\,\alpha_1,\,\alpha_2\dots\alpha_{n+1}}
 P_{M,\,\alpha_2\dots\alpha_{n+1}}
 \Big(
  f_{\lambda_2,\dots,\lambda_n;\,\alpha_2,\dots,\alpha_{n+1}}}\\
  &&\hspace{5cm}
  \uwave{\left(a_{2},\dots,P_{\alpha_{i_1}}(a_{i_{1}}),\dots,P_{\alpha_{i_k}}(a_{i_{k}}),\dots,a_{n+1}\right)
  \Big)}\\
			&& -
			P_{M,\,\alpha_1\dots\alpha_{n+1}}
\left(a_{1} \cdot_{\lambda_1;\,\alpha_1,\,\alpha_2\dots\alpha_{n+1}}
 f_{\lambda_2,\dots,\lambda_n;\,\alpha_2,\dots,\alpha_{n+1}}
 \left(P_{\alpha_2}(a_{2}),\dots,P_{\alpha_{n+1}}(a_{n+1})\right)\right)\\
			&& +
			\sum_{k=0}^{n-1}  q ^{n-k-1}
			\sum_{2\leq i_{1}<\dots<i_{k}\leq n+1}
			P_{M,\alpha_1\dots\alpha_n}
			\Big(a_{1} \cdot_{\lambda_1;\,\alpha_1,\,\alpha_2\dots\alpha_{n+1}}
 P_{M,\alpha_2\dots\alpha_{n+1}} \Big(
 f_{\lambda_2,\dots,\lambda_n;\,\alpha_2,\dots,\alpha_{n+1}}\\
 &&\hspace{5cm}
 \left(a_{2},\dots,P_{\alpha_{i_1}}(a_{i_{1}}),\dots,P_{\alpha_{i_k}}(a_{i_{k}}),\dots,a_{n+1}
 \right)\Big)\Big)\\
			&& + \sum_{i=1}^{n}(-1)^{i}
			f_{\lambda_1,\dots,\lambda_i+\lambda_{i+1},\dots,\lambda_n;\,\alpha_1,\dots,
\alpha_i\alpha_{i+1}\dots,\alpha_{n+1}}\\
&&\hspace{5cm}\left(P_{\alpha_1}(a_{1}),\dots,P_{\alpha_i\alpha_{i+1}}(a_{i}\cdot_{\lambda_i;\,\alpha_i,\alpha_{i+1}}
P_{\alpha_{i+1}}(a_{i+1})),\dots,P_{\alpha_{n+1}}(a_{n+1})\right)\\
			&& -\sum_{i=1}^{n}(-1)^{i}
			\sum_{k=0}^{n-1}  q ^{n-k-1}
			\sum_{1\leq i_{1}<\dots<i_{k}\leq n+1}
			P_{M,\,\alpha_1\dots\alpha_{n+1}}
			\Big(
f_{\lambda_1,\dots,\lambda_i+\lambda_{i+1},\dots,\lambda_n;\,\alpha_1,\dots,
\alpha_i\alpha_{i+1}\dots,\alpha_{n+1}}\\
&&\hspace{3cm}
\left(a_{1},\dots,P_{\alpha_{i_1}}(a_{i_{1}}),\dots,a_{i}\cdot_{\lambda_i;\alpha_i,\alpha_{i+1}}
P_{\alpha_{i+1}}(a_{i+1})\dots,P_{\alpha_{i_k}}(a_{i_{k}}),\dots,a_{n+1}\right)\Big)
			\\
			&& -\sum_{i=1}^{n}(-1)^{i}
			\sum_{k=1}^{n-1}  q ^{n-k-1}
			\sum_{1\leq i_{1}<\dots<i_{k}\leq n+1}
			P_{M,\,\alpha_1\dots\alpha_{n+1}}
			\Big(
f_{\lambda_1,\dots,\lambda_i+\lambda_{i+1},\dots,\lambda_n;\,\alpha_1,\dots,
\alpha_i\alpha_{i+1}\dots,\alpha_{n+1}}\\
&&\hspace{3cm}
\left(a_{1},\dots,P_{\alpha_{i_1}}(a_{i_{1}}),\dots,P_{\alpha_i\alpha_{i+1}}
(a_{i}\cdot_{\lambda_i;\alpha_i,\alpha_{i+1}} P_{\alpha_{i+1}}(a_{i+1}))\dots,P_{\alpha_{i_k}}(a_{i_{k}}),\dots,a_{n+1}\right)\Big)
			\\
			&& + \sum_{i=1}^{n}(-1)^{i}
			f_{\lambda_1,\dots,\lambda_n;\alpha_1,\dots,\alpha_{n+1}}
\left(P_{\alpha_1}(a_{1}),\dots,P_{\alpha_i\alpha_{i+1}}(P_{\alpha_i}(a_{i})\cdot_{\lambda_i;\,\alpha_i,\alpha_{i+1}}
a_{i+1}),\dots,P_{\alpha_{n+1}}(a_{n+1})\right)\\
			&& -\sum_{i=1}^{n}(-1)^{i}
			\sum_{k=0}^{n-1}  q ^{n-k-1}
			\sum_{1\leq i_{1}<\dots<i_{k}\leq n+1}
			P_{M,\alpha_1\dots\alpha_{n+1}}
			\Big(
f_{\lambda_1,\dots,\lambda_i+\lambda_{i+1},\dots,\lambda_n;\,\alpha_1,\dots,
\alpha_i\alpha_{i+1},\dots,\alpha_{n+1}}\\
&&\hspace{3cm}
\left(a_{1},\dots,P_{\alpha_{i_1}}(a_{i_{1}}),\dots,P_{\alpha_i}(a_{i})
\cdot_{\lambda_i;\alpha_i,\alpha_{i+1}}
a_{i+1}\dots,P_{\alpha_{i_k}}(a_{i_{k}}),\dots,a_{n+1}\right)\Big)
			\\
			&& -\sum_{i=1}^{n}(-1)^{i}
			\sum_{k=1}^{n-1}  q ^{n-k-1}
			\sum_{1\leq i_{1}<\dots<i_{k}\leq n+1}
			P_{M,\alpha_1\dots\alpha_{n+1}}
			\Big(
f_{\lambda_1,\dots,\lambda_i+\lambda_{i+1},\dots,\lambda_n;\,\alpha_1,\dots,
\alpha_i\alpha_{i+1},\dots,\alpha_{n+1}}\\
&&\hspace{3cm}
\left(a_{1},\dots,P_{\alpha_{i_1}}(a_{i_{1}}),\dots,P_{\alpha_i\alpha_{i+1}}
(P_{\alpha_i}(a_{i})\cdot_{\lambda_i;\,\alpha_i,\alpha_{i+1}}a_{i+1})\dots,P_{
\alpha_{i_k}}(a_{i_{k}}),\dots,a_{n+1}\right)\Big)
			\\
			&& + \sum_{i=1}^{n}(-1)^{i}  q
			f_{\lambda_1,\dots,\lambda_i+\lambda_{i+1},\dots,\lambda_n;\,\alpha_1,\dots,
\alpha_i\alpha_{i+1},\dots,\alpha_{n+1}}
\left(P_{\alpha_1}(a_{1}),\dots,P_{\alpha_i\alpha_{i+1}}(a_{i}\cdot_{\lambda_i;\alpha_i,\alpha_{i+1}}
a_{i+1}),\dots,P_{\alpha_{n+1}}(a_{n+1})\right)\\
			&& -\sum_{i=1}^{n}(-1)^{i}
			\sum_{k=0}^{n-1}  q ^{n-k}
			\sum_{1\leq i_{1}<\dots<i_{k}\leq n+1}
			P_{M,\alpha_1\dots\alpha_{n+1}}
			\Big(
f_{\lambda_1,\dots,\lambda_i+\lambda_{i+1},\dots,\lambda_n;\,\alpha_1,\dots,
\alpha_i\alpha_{i+1},\dots,\alpha_{n+1}}\\
&&\hspace{3cm}
\left(a_{1},\dots,P_{\alpha_{i_1}}(a_{i_{1}}),\dots,
a_{i}\cdot_{\lambda;\,\alpha_i,\alpha_{i+1}}a_{i+1}\dots,P_{\alpha_{i_k}}
(a_{i_{k}}),\dots,a_{n+1}\right)\Big)
			\\
			&& -\sum_{i=1}^{n}(-1)^{i}
			\sum_{k=1}^{n-1}  q ^{n-k}
			\sum_{1\leq i_{1}<\dots<i_{k}\leq n+1}
			P_{M,\alpha_1\dots\alpha_{n+1}}
			\Big(
f_{\lambda_1,\dots,\lambda_i+\lambda_{i+1},\dots,\lambda_n;\,\alpha_1,\dots,
\alpha_i\alpha_{i+1},\dots,\alpha_{n+1}}\\
&&\hspace{3cm}
\left(a_{1},\dots,P_{\alpha_{i_1}}(a_{i_{1}}),\dots,
P_{\alpha_i\alpha_{i+1}}(a_{i}\cdot_{\lambda;\,\alpha_i,\alpha_{i+1}}a_{i+1})
\dots,P_{\alpha_{i_k}}(a_{i_{k}}),\dots,a_{n+1}\right)\Big)
			\\
			&& +(-1)^{n+1}
f_{\lambda_1,\dots,\lambda_{n-1};\,\alpha_1,\dots,\alpha_n}
\left(P_{\alpha_1}(a_{1}),\dots,P_{\alpha_n}(a_{n})\right)
 \cdot_{\lambda_1+\dots+\lambda_n;\,\alpha_1\dots\alpha_n,\,\alpha_{n+1}}
 P_{\alpha_{n+1}}(a_{n+1})\\
			&& -(-1)^{n+1}
			\sum_{k=0}^{n-1}  q ^{n-k-1}
			\sum_{1\leq i_{1}<\dots<i_{k}\leq n}
			\uwave{P_{M,\,\alpha_1\dots\alpha_{n+1}}
\Big(
f_{\lambda_1,\dots,\lambda_{n-1};\,\alpha_1,\dots,\alpha_n}}\\
&&\hspace{3cm}
\uwave{\left(a_{1},\dots,P_{\alpha_{i_1}}(a_{i_{1}}),\dots,P_{\alpha_{i_k}}(a_{i_{k}}),
\dots,a_{n}\right)\Big)
 \cdot_{\lambda_1+\dots+\lambda_n;\,\alpha_1\dots\alpha_n,\,\alpha_{n+1}}
  P_{\alpha_{n+1}}(a_{n+1})}\\
			&& -(-1)^{n+1}
			P_{M,\alpha_1\dots\alpha_n}
\left(
 f_{\lambda_1,\dots,\lambda_{n-1};\,\alpha_1,\dots,\alpha_n}
 \left(P_{\alpha_1}(a_{1}),\dots,P_{\alpha_n}(a_{n})\right)
 \cdot_{\lambda_1+\dots+\lambda_n;\,\alpha_1\dots\alpha_n,\,\alpha_{n+1}}
  a_{n+1}\right)\\
			&& +(-1)^{n+1}
			\sum_{k=0}^{n-1}  q ^{n-k-1}
			\sum_{1\leq i_{1}<\dots<i_{k}\leq n}
			P_{M,\,\alpha_1\dots\alpha_{n+1}}
			\Big(P_{M,\,\alpha_1\dots\alpha_n}
\Big(
f_{\lambda_1,\dots,\lambda_{n-1};\,\alpha_1,\dots,\alpha_n}\\
&&\hspace{3cm}
\left(a_{1},\dots,P_{\alpha_{i_1}}(a_{i_{1}}),\dots,P_{\alpha_{i_k}}(a_{i_{k}}),\dots,a_{n}\right)\Big) \cdot_{\lambda_1+\dots+\lambda_n;\,\alpha_1\dots\alpha_n,\alpha_{n+1}}
 a_{n+1} \Big)\\
			&=& \uline{ P_{\alpha_1}(a_{1})
 \cdot_{\lambda_1;\,\alpha_1,\alpha_2\dots\alpha_{n+1}}
  f_{\lambda_2,\dots,\lambda_n;\alpha_2,\dots,\alpha_{n+1}}
  \left(P_{\alpha_2}(a_{2}),\dots,P_{\alpha_{n+1}}(a_{n+1})\right) }_{(1)}\\
			&& \uline{ -
				\sum_{k=0}^{n-1}  q ^{n-k-1}
				\sum_{2\leq i_{1}<\dots<i_{k}\leq n+1}
				P_{M,\,\alpha_1\dots\alpha_{n+1}}
\Big(a_{1}
\cdot_{\lambda_1;\,\alpha_1,\alpha_2\dots\alpha_{n+1}}
 P_{M,\,\alpha_2\dots\alpha_{n+1}}
  \Big(
  f_{\lambda_2,\dots,\lambda_n;\,\alpha_2,\dots,\alpha_{n+1}}}_{(16)} \\
  &&\hspace{3cm}
  \uline{\left(a_{2},\dots,P_{\alpha_{i_1}}(a_{i_{1}}),\dots,P_{\alpha_{i_k}}(a_{i_{k}}),\dots,a_{n+1}\right)\Big)\Big) }_{(16)} \\
			&& \uuline{ -
				\sum_{k=0}^{n-1}  q ^{n-k-1}
				\sum_{2\leq i_{1}<\dots<i_{k}\leq n+1}
				P_{M,\,\alpha_1\dots\alpha_{n+1}}
\Big(P_{\alpha_1}(a_{1})
 \cdot_{\lambda_1;\,\alpha_1,\alpha_2\dots\alpha_{n+1}}
   f_{\lambda_2,\dots,\lambda_n;\,\alpha_2,\dots,\alpha_{n+1}}}_{(7)}\\
   &&\hspace{3cm}
   \uuline{\left(a_{2},\dots,P_{\alpha_{i_1}}(a_{i_{1}}),\dots,P_{\alpha_{i_k}}
   (a_{i_{k}}),\dots,a_{n+1}\right)\Big) }_{(7)} \\
			&& \uline{ -
				\sum_{k=0}^{n-1}  q ^{n-k}
				\sum_{2\leq i_{1}<\dots<i_{k}\leq n+1}
				P_{M,\,\alpha_1\dots\alpha_{n+1}}
\Big(a_{1}
\cdot_{\lambda_1;\,\alpha_1,\alpha_2\dots\alpha_{n+1}}
 f_{\lambda_2,\dots,\lambda_n;\,\alpha_2,\dots,\alpha_{n+1}}}_{(6)}\\
 &&\hspace{3cm}
 \uline{\left(a_{2},\dots,P_{\alpha_{i_1}}(a_{i_{1}}),\dots,P_{\alpha_{i_k}}(a_{i_{k}}),
 \dots,a_{n+1}\right)\Big) }_{(6)} \\
			&& \uline{ -
				P_{M,\,\alpha_1\dots\alpha_{n+1}}
\left(a_{1} \cdot_{\lambda_1;\,\alpha_1,\alpha_2\dots\alpha_{n+1}}
f_{\lambda_2,\dots,\lambda_n;\,\alpha_2,\dots,\alpha_{n+1}}
\left(P_{\alpha_2}(a_{2}),\dots,P_{\alpha_{n+1}}(a_{n+1})\right)\right) }_{(6)}\\
			&& \uline{ +
				\sum_{k=0}^{n-1}  q ^{n-k-1}
				\sum_{2\leq i_{1}<\dots<i_{k}\leq n+1}
				P_{M,\,\alpha_1\dots\alpha_{n+1}}
				\Big(a_{1}
\cdot_{\lambda_1;\,\alpha_1,\alpha_2\dots\alpha_{n+1}}
 P_{M,\,\alpha_2\dots\alpha_{n+1}}
 \Big(
 f_{\lambda_2,\dots,\lambda_n;\,\alpha_2,\dots,\alpha_{n+1}}}_{(16)}\\
 &&\hspace{3cm}
 \uline{\left(a_{2},\dots,P_{\alpha_{i_1}}(a_{i_{1}}),\dots,P_{\alpha_{i_k}}(a_{i_{k}}),\dots,a_{n+1}\right)\Big)\Big) }_{(16)}\\
			&& \uuline{ + \sum_{i=1}^{n}(-1)^{i}
				f_{\lambda_1,\dots,\lambda_i+\lambda_{i+1},\dots,\lambda_n;\,\alpha_1,
\dots,\alpha_i+\alpha_{i+1},\dots,\alpha_{n+1}}
\left(P_{\alpha_1}(a_{1}),\dots,P_{\alpha_i\alpha_{i+1}}(a_{i}\cdot_{\lambda_i;\,\alpha_i,\alpha_{i+1}}
P_{\alpha_{i+1}}(a_{i+1})),\dots,P_{\alpha_{n+1}}(a_{n+1})\right) }_{(2)}\\
			&& \dotuline{ -\sum_{i=1}^{n}(-1)^{i}
				\sum_{k=0}^{n-1}  q ^{n-k-1}
				\sum_{1\leq i_{1}<\dots<i_{k}\leq n+1}
				P_{M,\,\alpha_1\dots\alpha_{n+1}}
				\Big(
f_{\lambda_1,\dots,\lambda_i+\lambda_{i+1},\dots,\lambda_n;\,\alpha_1,
\dots,\alpha_i\alpha_{i+1},\dots,\alpha_{n+1}}}_{(10)}
		\\
&&\hspace{3cm}
\dotuline{\left(a_{1},\dots,P_{\alpha_{i_1}}(a_{i_{1}}),\dots,a_{i}\cdot_{\lambda_i;\,\alpha_i,\,\alpha_{i+1}}
P_{\alpha_{i+1}}(a_{i+1})\dots,P_{\alpha_{i_k}}(a_{i_{k}}),\dots,a_{n+1}\right)\Big) }_{(10)}
			\\
			&& \uline{ -\sum_{i=1}^{n}(-1)^{i}
				\sum_{k=1}^{n-1}  q ^{n-k-1}
				\sum_{1\leq i_{1}<\dots<i_{k}\leq n+1}
				P_{M,\,\alpha_1\dots\alpha_{n+1}}
				\Big(
f_{\lambda_1,\dots,\lambda_i+\lambda_{i+1},\dots,\lambda_n;\,\alpha_1,
\dots,\alpha_i\alpha_{i+1},\dots,\alpha_{n+1}}}_{(11)}\\
&&\hspace{3cm}
\uline{\left(a_{1},\dots,P_{\alpha_{i_1}}(a_{i_{1}}),\dots,P_{\alpha_i\alpha_{i+1}}
(a_{i}\cdot_{\lambda_i;\,\alpha_i,\alpha_{i+1}}
P_{\alpha_{i+1}}(a_{i+1}))\dots,P_{\alpha_{i_k}}(a_{i_{k}}),\dots,a_{n+1}\right)\Big) }_{(11)}
			\\
			&& \uwave{ + \sum_{i=1}^{n}(-1)^{i}
				f_{\lambda_1,\dots,\lambda_i+\lambda_{i+1},\dots,\lambda_n;\,\alpha_1,
\dots,\alpha_i\alpha_{i+1},\dots,\alpha_{n+1}}
\left(P_{\alpha_1}(a_{1}),\dots,P_{\alpha_i\alpha_{i+1}}(P_{\alpha_i}(a_{i})
\cdot_{\lambda_i;\,\alpha_i,\alpha_{i+1}}a_{i+1}),\dots,P_{\alpha_{n+1}}(a_{n+1})\right) }_{(3)}\\
			&& \dashuline{ -\sum_{i=1}^{n}(-1)^{i}
				\sum_{k=0}^{n-1}  q ^{n-k-1}
				\sum_{1\leq i_{1}<\dots<i_{k}\leq n+1}
				P_{M,\,\alpha_1\dots\alpha_{n+1}}
				\Big(
f_{\lambda_1,\dots,\lambda_i+\lambda_{i+1},\dots,\lambda_n;\,\alpha_1,
\dots,\alpha_i\alpha_{i+1},\dots,\alpha_{n+1}}}_{(9)}\\
&&\hspace{3cm}
\dashuline{\left(a_{1},\dots,P_{\alpha_{i_1}}(a_{i_{1}}),\dots,P_{\alpha_i}(a_{i})
\cdot_{\lambda_i;\,\alpha_i,\alpha_{i+1}}a_{i+1}\dots,P_{\alpha_{i_k}}(a_{i_{k}}),\dots,a_{n+1}\right)\Big) }_{(9)}
			\\
			&& \uuline{ -\sum_{i=1}^{n}(-1)^{i}
				\sum_{k=1}^{n-1}  q ^{n-k-1}
				\sum_{1\leq i_{1}<\dots<i_{k}\leq n+1}
				P_{M,\,\alpha_1\dots\alpha_{n+1}}
				\Big(
f_{\lambda_1,\dots,\lambda_i+\lambda_{i+1},\dots,\lambda_n;\,\alpha_1,
\dots,\alpha_i\alpha_{i+1},\dots,\alpha_{n+1}}}_{(12)}\\
&&\hspace{3cm}
\uuline{\left(a_{1},\dots,P_{\alpha_{i_1}}(a_{i_{1}}),\dots,P_{\alpha_i\alpha_{i+1}}
(P_{\alpha_i}(a_{i})\cdot_{\lambda_i;\,\alpha_i,\alpha_{i+1}}a_{i+1})\dots,P_{\alpha_{i_k}}(a_{i_{k}}),
\dots,a_{n+1}\right)\Big) }_{(12)}
			\\
			&& \dashuline{ + \sum_{i=1}^{n}(-1)^{i}  q
				f_{\lambda_1,\dots,\lambda_i+\lambda_{i+1},\dots,\lambda_n;\,\alpha_1,
\dots,\alpha_i\alpha_{i+1},\dots,\alpha_{n+1}}
\left(P_{\alpha_1}(a_{1}),\dots,P_{\alpha_i\alpha_{i+1}}(a_{i}\cdot_{\lambda_i;\,\alpha_i,\alpha_{i+1}}
a_{i+1}),\dots,P_{\alpha_{n+1}}(a_{n+1})\right) }_{(4)}\\
			&& \uwave{ -\sum_{i=1}^{n}(-1)^{i}
				\sum_{k=0}^{n-1}  q ^{n-k}
				\sum_{1\leq i_{1}<\dots<i_{k}\leq n+1}
				P_{M,\,\alpha_1\dots\alpha_{n+1}}
				\Big(
f_{\lambda_1,\dots,\lambda_i+\lambda_{i+1},\dots,\lambda_n;\,\alpha_1,
\dots,\alpha_i\alpha_{i+1},\dots,\alpha_{n+1}}}_{(8)}\\
&&\hspace{3cm}
\uwave{\left(a_{1},\dots,P_{\alpha_{i_1}}(a_{i_{1}}),\dots,a_{i}\cdot_{\lambda_i;\,\alpha_i,\alpha_{i+1}}
a_{i+1}\dots,P_{\alpha_{i_k}}(a_{i_{k}}),\dots,a_{n+1}\right)\Big) }_{(8)}
			\\
			&& \uwave{ -\sum_{i=1}^{n}(-1)^{i}
				\sum_{k=1}^{n-1}  q ^{n-k}
				\sum_{1\leq i_{1}<\dots<i_{k}\leq n+1}
				P_{M,\,\alpha_1\dots\alpha_{n+1}}
				\Big(
f_{\lambda_1,\dots,\lambda_i+\lambda_{i+1},\dots,\lambda_n;\,\alpha_1,
\dots,\alpha_i\cdot_{\lambda_i;\,\alpha_i,\alpha_{i+1}}\alpha_{i+1},\dots,\alpha_{n+1}}}_{(13)}\\
&&\hspace{3cm}
\uwave{\left(a_{1},\dots,P_{\alpha_{i_1}}(a_{i_{1}}),\dots,P_{\alpha_i\alpha_{i+1}}
(a_{i}\cdot_{\lambda_i;\,\alpha_i,\alpha_{i+1}}a_{i+1})\dots,P_{\alpha_{i_k}}(a_{i_{k}}),\dots,a_{n+1}\right)\Big) }_{(13)}
			\\
			&& \dotuline{ +(-1)^{n+1}
f_{\lambda_1,\dots,\lambda_{n-1};\,\alpha_1,\dots,\alpha_n}
 \left(P_{\alpha_1}(a_{1}),\dots,P_{\alpha_n}(a_{n})\right)
 \cdot_{\lambda_1+\dots+\lambda_n;\alpha_1\dots\alpha_n,\alpha_{n+1}}
  P_{\alpha_{n+1}}(a_{n+1}) }_{(5)}\\
			&& \dotuline{ -(-1)^{n+1}
				\sum_{k=0}^{n-1}  q ^{n-k-1}
				\sum_{1\leq i_{1}<\dots<i_{k}\leq n}
				P_{M,\,\alpha_1\dots\alpha_{n+1}}
\Big(
 f_{\lambda_1,\dots,\lambda_{n-1};\,\alpha_1,\dots,\alpha_n}}_{(15)}\\
 &&\hspace{3cm}
\dotuline{ \left(a_{1},\dots,P_{\alpha_{i_1}}(a_{i_{1}}),\dots,P_{\alpha_{i_k}}(a_{i_{k}}),\dots,a_{n}\right)
 \cdot_{\lambda_1+\dots+\lambda_n;\alpha_1\dots\alpha_n,\alpha_{n+1}}
  P_{\alpha_{n+1}}(a_{n+1}) \Big) }_{(15)} \\
			&& \uuline{ -(-1)^{n+1}
				\sum_{k=0}^{n-1}  q ^{n-k-1}
				\sum_{1\leq i_{1}<\dots<i_{k}\leq n}
				P_{M,\,\alpha_1\dots\alpha_{n+1}}
\Big(
				P_{M,\,\alpha_1\dots\alpha_{n}}
\Big(
f_{\lambda_1,\dots,\lambda_{n-1};\,\alpha_1,\dots,\alpha_n}}_{(17)}\\
&&\hspace{3cm}
\uuline{\left(a_{1},\dots,P_{\alpha_{i_1}}(a_{i_{1}}),\dots,P_{\alpha_{i_k}}(a_{i_{k}}),
\dots,a_{n}\right)\Big)
\cdot_{\lambda_1+\dots+\lambda_n;\,\alpha_1\dots\alpha_n,\alpha_{n+1}}
 a_{n+1} \Big) }_{(17)} \\
			&& \dashuline{ -(-1)^{n+1}
				\sum_{k=0}^{n-1}  q ^{n-k}
				\sum_{1\leq i_{1}<\dots<i_{k}\leq n}
				P_{M,\,\alpha_1\dots\alpha_{n+1}}
\Big(
				f_{\lambda_1,\dots,\lambda_{n-1};\,\alpha_1,\dots,\alpha_n}}_{(14)}\\
&&\hspace{3cm}
\dashuline{\left(a_{1},\dots,P_{\alpha_{i_1}}(a_{i_{1}}),\dots,P_{\alpha_{i_k}}(a_{i_{k}}),
\dots,a_{n}\right)
\cdot_{\lambda_1+\dots+\lambda_n;\,\alpha_1\dots\alpha_n,\alpha_{n+1}}
 a_{n+1} \Big) }_{(14)} \\
			&& \dashuline{ -(-1)^{n+1}
				P_{M,\,\alpha_1\dots\alpha_{n+1}}
\Big(
 f_{\lambda_1,\dots,\lambda_{n-1};\,\alpha_1,\dots,\alpha_n}
\left(P_{\alpha_1}(a_{1}),\dots,P_{\alpha_n}(a_{n})\right)
 \cdot_{\lambda_1+\dots+\lambda_n;\,\alpha_1\dots\alpha_n,\alpha_{n+1}}
  a_{n+1}\Big) }_{(14)}\\
			&& \uuline{ +(-1)^{n+1}
				\sum_{k=0}^{n-1}  q ^{n-k-1}
				\sum_{1\leq i_{1}<\dots<i_{k}\leq n}
				P_{M,\,\alpha_1\dots\alpha_{n+1}}
				\Big(
P_{M,\,\alpha_1\dots\alpha_{n}}
\Big(
f_{\lambda_1,\dots,\lambda_{n-1};\,\alpha_1,\dots,\alpha_n}}_{(17)}\\
&&\hspace{3cm}
\uuline{\left(a_{2},\dots,P_{\alpha_{i_1}}(a_{i_{1}}),\dots,P_{\alpha_{i_k}}
(a_{i_{k}}),\dots,a_{n}\right)\Big)
\cdot_{\lambda_1+\dots+\lambda_n;\,\alpha_1\dots\alpha_n,\alpha_{n+1}}
 a_{n+1} \Big). }_{(17)} \\
		\end{eqnarray*}
	Compare the above two equations and it is easy to see that they are equal, i.e.,
	\[ \Phi^{n+1} \circ \delta^{n} = \partial^{n} \circ \Phi^{n} .\]
	
	This completes the proof.
\end{proof}

\smallskip
\noindent
{{\bf Acknowledgments.} This work is supported in part by Natural Science Foundation of China (Grant No. 12101183 and 12201182) and project funded by China
Postdoctoral Science Foundation (Grant No. 2021M690049).

\end{document}